\documentclass[12pt]{article}
\usepackage{amsmath}
\usepackage{amssymb}
\usepackage{amsthm}
\usepackage{amscd}
\usepackage{amsxtra}
\usepackage{verbatim}
\usepackage{xcolor}
\usepackage{color}
\usepackage{enumerate}
\usepackage{mathrsfs}
\usepackage{stmaryrd}
\usepackage[all]{xy} 

\usepackage{array,arydshln}
\usepackage{bm}

\allowdisplaybreaks

\numberwithin{equation}{section}

\definecolor{dblue}{rgb}{0,0,0.45}
\definecolor{red}{rgb}{0.7,0,0}

 \RequirePackage{geometry}
\newtheorem{theorem}{Theorem}[section]
\newtheorem{lemma}[theorem]{Lemma}
\newtheorem*{lemma*}{Lemma}
\newtheorem{corollary}[theorem]{Corollary}
\newtheorem{proposition}[theorem]{Proposition}

\theoremstyle{definition}

\newtheorem{remark}[theorem]{Remark}

\theoremstyle{remark}

\newcommand{\N}{{\mathbb N}}
\newcommand{\R}{{\mathbb R}}

\newcommand{\Z}{{\mathbb Z}}

\newcommand{\cC}{{\mathcal C}}

\newcommand{\cH}{{\mathcal H}}

\newcommand{\cK}{{\mathcal K}}

\newcommand{\cM}{{\mathcal M}}
\newcommand{\cN}{{\mathcal N}}
\newcommand{\cP}{{\mathcal P}}

\newcommand{\cV}{{\mathcal V}}
\newcommand{\cW}{{\mathcal W}}

\newcommand{\cY}{{\mathcal Y}}

\newcommand{\la}{\langle}
\newcommand{\ra}{\rangle}
\newcommand{\nn}{\nonumber}
\newcommand{\ve}{\varepsilon}

\newcommand{\vertiii}[1]{{\left\vert\kern-0.25ex\left\vert\kern-0.25ex\left\vert #1 
    \right\vert\kern-0.25ex\right\vert\kern-0.25ex\right\vert}}

\date{}
\begin{document}

\title{
Fractional Diffusion Bridges
}
\author{   Yuzuru \textsc{Inahama} 
}
\maketitle

\begin{abstract}
Consider ``stochastic differential equations" driven by 
fractional Brownian motion with Hurst parameter $H\in (1/4, 1)$.
Their solutions are sometimes called fractional diffusion processes.
The main purpose of this paper is  
conditioning these processes to reach a given terminal point. 
We call the conditioned processes fractional diffusion bridges.
Our main tool for mathematically rigorous conditioning 
 is quasi-sure analysis, which is a potential theoretic part of 
 Malliavin calculus.
 We also prove a small-noise large deviation principle 
of Freidlin-Wentzell type for scaled fractional diffusion bridges 
under a mild ellipticity assumption on the coefficient vector fields. 
\end{abstract}

\section{Introduction}

Consider a diffusion process on the Euclidean space
$\R^e$, $e\ge 1$, which is realized as (the law of)
a unique solution of the following stochastic differential equation (SDE) of Stratonovich type driven by the standard $d$-dimensional
Brownian motion $w = (w_t)_{t \in [0,T]}$:
\[
dy_t = \sum_{j=1}^d V_j (y_t)\circ dw^{j}_t  + V_0 (y_t) dt,
\qquad 
y_0 =a\in \R^e.
\]
Here, $V_j~(0\le j \le d)$ are sufficiently regular vector fields on $\R^e$
and $T\in (0, \infty)$ is the time horizon.
When we emphasize the initial value $a$, we write $y (t,a)$
instead of $y_t$. 
If the law of $(y (t,a))_{t \in [0,T]}$ is denoted by $\mathbb{Q}_a$,
then $\{\mathbb{Q}_a \mid a \in \R^e\}$ becomes a diffusion process 
associated with the generator $V_0 +(1/2) \sum_{j=1}^d V_j^2$,
which is a very important research object in probability theory.

In many problems in the theory of diffusion processes, 
conditioning the process $(y (t,a))_{t \in [0,T]}$ to reach
a given point $b\in\R^e$ at the terminal time $T$ is quite important and useful.
The conditioned process is called a diffusion bridge 
(or pinned diffusion process) from $a$ to $b$, 
whose law will be denoted by $\mathbb{Q}_{a,b}$.
Rigorous construction of a diffusion bridge is not so easy.
As far as the author knows, there are four methods:
\begin{enumerate} 
\item[{\bf (a)}]~Regular conditional probability given a random variable.
One defines $\mathbb{Q}_{a,b}$ as $\mathbb{Q}_{a} (\,\bullet\, | \Pi_T =b)$,
or equivalently, as $\mu ( y (\,\cdot\,,a) \in \,\bullet\, | y (T,a) =b)$.
Here, $\mu$ is the $d$-dimensional Wiener measure and
$\Pi_T$ stands for the evaluation map at the time $T$.
\item[{\bf (b)}]~
Define finite dimensional distributions of $\mathbb{Q}_{a,b}$ 
by using Chapman-Kolmogorov type formula
when the law of $y (t,a)$ has a continuous density function
$p(t, a, b)$ for all $a\in\R^e$ and $t>0$. (See the right hand side of 
\eqref{def.CK} below.)
\item[{\bf (c)}]~When the law of $y (t,a)$ has a 
density function $p(t, a, b)$ for all $a\in\R^e$ and 
$p$ is sufficiently regular as a function of $(t, a, b)$,  one can obtain 
$\mathbb{Q}_{a,b}$ as the law of the solution of the following SDE 
whose drift is quite singular near time $T$:
\[
d\hat{y}_t = \sum_{j=1}^d V_j (\hat{y}_t)\circ dw^{j}_t  
+ V_0 (\hat{y}_t) dt +
(\mathbf{V}\mathbf{V}^{\top}) (\hat{y}_t) \nabla \log p(T-t, \hat{y}_t ,b) dt,
\quad 
\hat{y}_0 =a.
\]
Here, $\nabla p(t,a ,b)$ stands for the gradient of $a \mapsto p(t,a ,b)$
and $\mathbf{V}\mathbf{V}^{\top}$ is the diffusion matrix
associated with $\{V_j\}_{ 1\le j \le d}$
(see at the beginning of 
Section \ref{sec.LDP_state} for a precise definition).
\item[{\bf (d)}]~Quasi-sure analysis, which is a potential theoretic part 
of Malliavin calculus. This theory enables us to define 
a Borel probability measure on a very thin subset of the Wiener space, which looks like a ``pullback" of the diffusion bridge measure $\mathbb{Q}_{a,b}$. One can obtain $\mathbb{Q}_{a,b}$
as the law of the It\^o map (i.e. the solution map of the SDE)
 under the pullback measure.
\end{enumerate}

The main purpose of our present paper 
is constructing bridge processes for solutions of ``SDEs" driven by 
fractional Brownian motion (FBM) with Hurst parameter $H \in (1/4, 1)$.
These ``SDEs" are understood 
in the rough path sense when $H \in (1/4, 1/2]$
and in the Young sense when $H \in (1/2, 1)$.
The solutions of these equations are sometimes called fractional diffusion processes. 
When $H=1/2$, the solutions
 coincide with those of Stratonovich SDEs.
Our main tool is quasi-sure analysis as in Item {\bf (d)} above.
Recall that Malliavin calculus for those type of 
stochastic systems is now well-developed.
We call these bridge processes fractional diffusion bridges.
We also prove a small-noise large deviation principle of 
Freidlin-Wentzell type for these (scaled) fractional diffusion bridges
when the diffusion coefficients satisfy a suitable ellipticity condition.

Now we quickly explain why Items {\bf (a)}--{\bf (c)} are unavailable 
for conditioning fractional diffusion processes.
Chapman-Kolmogorov formula as in {\bf (b)} is in a sense
the Markov property itself. 
The main reasons why the SDE in {\bf (c)} produces a diffusion bridge 
is Doob's $h$-transform, which is a famous theorem 
in the theory of Markov processes.
Since our processes are not Markov, 
neither {\bf (b)} nor {\bf (c)} are available.
Though Method {\bf (a)} is very general, it is merely defined for
almost all $b$ with respect to $(\Pi_T)_*\mathbb{Q}_{a}^H= y (T, a)_* \mu$. 
Therefore, it is meaningless to discuss 
$\mathbb{Q}_{a}^H (\,\bullet\, | \Pi_T =b)$ or 
$\mu ( y (\,\cdot\,,a) \in \,\bullet\, | y (T,a) =b)$ for a fixed $b$
since the singleton $\{b\}$
is usually of measure zero with respect to $y (T, a)_* \mu$.
(See Remark \ref{rem.RCPD} for details.)
In contrast, a  remarkable advantage 
of using {\bf (d)} is that we can define 
$\mathbb{Q}_{a, b}^H$ for every $a$ and $b$ with 
$p (T, a,b) >0$ when $y (T, a)$ is non-degenerate in the sense of Malliavin.
(In this paragraph, $(y (t, a))_{t \in [0,T]}$ is the fractional diffusion process starting at $a$, whose law is denoted by $\mathbb{Q}_{a}^H$.
The law of $(y (t, a))_{t \in [0,T]}$ conditioned that 
$y (T,a) =b$ is denoted by $\mathbb{Q}_{a, b}^H$.)

\begin{remark} 
Our fractional diffusion processes/bridges do not satisfy
the Markov property (unless $H=1/2$).
Therefore, when we say fractional diffusion processes/bridges,
the term {\it diffusion} is not used as a terminology 
from the theory of Markov processes, 
but belongs to everyday vocabulary of English.
\end{remark}

The organization of this paper is as follows.
In Section \ref{sec.RPMal}, many important facts on
Besov-type rough paths and Malliavin calculus are reviewed. 
Those will be used in later sections. 
In Section \ref{sec.QSlift} we provide a quasi-sure refinement 
of the canonical rough path lift for a certain class of Gaussian processes
including FBM with Hurst parameter $H \in (1/4,1/2]$.
We then prove the lift map is $\infty$-quasi-continuous.
Section \ref{sec.conditionFDP} is a core part of this paper.
In this section, we condition a fractional diffusion process, that is,
a unique solution of a rough differential equation driven by 
the canonical lift of FBM with Hurst parameter $H \in (1/4,1/2]$. 
The conditioned processes (or measures)
 obtained in this way are called fractional diffusion bridges.
In Sections \ref{sec.LDP_state} and \ref{sec.proof.low}
we state and prove a small-noise large deviation principle of 
Freidlin-Wentzell type for scaled fractional diffusion bridges with $H \in (1/4,1/2]$
under a certain
 ``everywhere ellipticity"" condition on the diffusion coefficient.
In Section \ref{sec.Young}, we consider the case $H \in (1/2, 1)$
and show that all the results in Sections \ref{sec.conditionFDP}--\ref{sec.proof.low} still hold true. Note that this case is much simpler 
because rough path theory is not needed to define fractional 
diffusion processes. So, our explanations are a little bit quick.

\medskip
\noindent
{\bf Notation}
\\
\noindent
Before closing Introduction, 
 we introduce the notation which will be used throughout the paper.
We write $\N =\{1,2, \ldots\}$ and $\N_0 :=\N \cup \{0\}$.
We set $\llbracket j, k\rrbracket :=\{ j, j+1,\ldots, k\}$ for $j, k\in\N_0$
with $j\le k$.
Let $T >0$ be arbitrary and we work on the time interval
$[0,T]$ unless otherwise specified.
We set $\triangle_{T} =\{ (s,t)\in \R^2 \mid  0\le s\le t \le T\}$.
For an Euclidean space $\cV$, 
the truncated tensor algebra 
of degree $k~(k\in\N)$ over $\cV$ is defined by 
$T^k (\cV) := \oplus_{i=0}^k \cV^{\otimes i}$, 
where we set $\cV^{\otimes 0}:=\R$.
Let $\gamma \in (0,1]$, $p \in [1,\infty)$ and 
$(\alpha, m) \in (0,1]\times [1,\infty)$ with $\alpha -m^{-1} >0$.
For a continuous map $\eta \colon \triangle_{T} \to \cV$, 
we define three kinds of norms as follows:
\begin{align*}  
\|\eta\|_{\gamma\textrm{-Hld}} :&=\sup _{0\le s<t\le T} 
\frac{\left|\eta_{s,t}\right|}{(t-s)^{\gamma}},
\qquad
\|\eta\|_{\alpha, m\textrm{-Bes}} :=
\left(
\iint_{ \triangle_T}   \frac{| \eta_{s,t}|^m}{(t-s)^{1+\alpha m}}dsdt
\right)^{\frac{1}{m}},
\nn\\
\|\eta\|_{p\textrm{-var}} :&=
\left(
\sup_{\cP}  \sum_{i=1}^N | \eta_{t_{i-1}, t_i}|^p
\right)^{\frac{1}{p}},
\qquad \mbox{where}\quad
\cP =\{0=t_0 <t_1 <\cdots <t_N =T\}.
\end{align*}
Here, the supremum over $\cP$ runs over all partitions of $[0,T]$.
These are called
the $\gamma$-H\"older norm, the $(\alpha, m)$-Besov norm 
and the $p$-variation norm, respectively.
Clearly, 
if $0< \alpha < \gamma \le  1$, there exists a positive constant 
$C=C_{\gamma', \gamma, m}$ independent of $\eta$ such that
$
\|\eta\|_{\alpha, m\textrm{-Bes}} \le C \|\eta\|_{\gamma\textrm{-Hld}}
$
holds for all $\eta$.

%

\begin{itemize} 
\item
The set of all continuous path $\varphi\colon [0,T] \to\cV$
is denoted by $\cC (\cV)$. 
With the usual sup-norm $\|\varphi\|_{\infty}$ over the interval $[0,T]$,
$\cC (\cV)$ is a Banach space.
The difference of $\varphi$ is denoted by $\varphi^1$,
that is, $\varphi^1_{s,t} := \varphi_t - \varphi_s$ for $(s,t)\in \triangle_{T}$.
For $a, b \in \cV$, we set 
$\cC_a (\cV):=\{ \varphi \in \cC (\cV)\mid \varphi_{0}=a\}$ and 
$\cC_{a, b} (\cV):=\{ \varphi \in \cC (\cV)\mid \varphi_{0}=a, \varphi_{T}=b\}$ .

\item
Let $0< \gamma \le 1$.
The space of all $\gamma$-H\"older continuous paths on $[0,T]$ is denoted by $\cC^{\gamma\textrm{-Hld}} (\cV)= \{\varphi \in \cC (\cV) \mid 
\|\varphi^1\|_{\gamma\textrm{-Hld}} <\infty\}$.
We simply write $\|\varphi\|_{\gamma\textrm{-Hld}}
=\|\varphi^1\|_{\gamma\textrm{-Hld}}$, which is called $\gamma$-H\"older seminorm.
The norm on this Banach space is
 $|\varphi_0|+\|\varphi\|_{\gamma\textrm{-Hld}}$.
For $a, b \in \cV$, we set 
$\cC^{\gamma\textrm{-Hld}}_a (\cV) :=
\cC^{\gamma\textrm{-Hld}} (\cV)\cap \cC_a (\cV)$ and  
 $\cC^{\gamma\textrm{-Hld}}_{a,b} (\cV) :=
\cC^{\gamma\textrm{-Hld}} (\cV)\cap \cC_{a, b} (\cV)$.
The closure of the set of all $C^1$-paths 
in $\cC^{\gamma\textrm{-Hld}} (\cV)$ is denoted by 
$\tilde\cC^{\gamma\textrm{-Hld}} (\cV)$, which is a separable 
Banach subspace and is often
called the little $\alpha$-H\"odler space.
We set $\tilde\cC^{\gamma\textrm{-Hld}}_a (\cV)$
and $\tilde\cC^{\gamma\textrm{-Hld}}_{a,b} (\cV)$ in an analogous way.

\item
For  $p \in [1,\infty)$ and 
$(\alpha, m) \in (0,1]\times [1,\infty)$ with $\alpha -m^{-1} >0$,
we define 
$\cC^{p\textrm{-var}} (\cV)= \{\varphi \in \cC (\cV) \mid 
\|\varphi^1\|_{p\textrm{-var}} <\infty\}$
and  
$\cC^{\alpha, m\textrm{-Bes}} (\cV)= \{\varphi \in \cC (\cV) \mid 
\|\varphi^1\|_{\alpha, m\textrm{-Bes}} <\infty\}$.
We simply write 
$\|\varphi\|_{p\textrm{-var}}=\|\varphi^1\|_{p\textrm{-var}}$ and 
$\|\varphi\|_{\alpha, m\textrm{-Bes}}=\|\varphi^1\|_{\alpha, m\textrm{-Bes}}$.
The norms on these Banach spaces are 
$|\varphi_0|+ \|\varphi\|_{p\textrm{-var}}$ and 
$|\varphi_0|+\|\varphi\|_{\alpha, m\textrm{-Bes}}$, respectively.
We set $\cC^{p\textrm{-var}}_a (\cV)$, $\cC^{p\textrm{-var}}_{a,b} (\cV)$,
$\cC^{\alpha, m\textrm{-Bes}}_a (\cV)$, $\cC^{\alpha, m\textrm{-Bes}}_{a,b} (\cV)$ in an analogous way.

\item
Let $U \subset \cV$ be a domain.
For $k \in \N_0$,  $C^k (U, \cW)$ stands for the set of 
$C^k$-functions from $U$ to another Euclidean space $\cW$.
(When $k=0$, we simply write $C (U, \cW)$ 
instead of $C^0 (U, \cW)$.)
The set of bounded $C^k$-functions $f \colon U\to \cW$
whose derivatives up to order $k$ are all bounded 
is denoted by $C_{{\rm b}}^k (U, \cW)$, which is a Banach space with the norm
$\| f\|_{C_{{\rm b}}^k } := \sum_{i=0}^k \|\nabla^i f\|_{\infty}$.
(Here, $ \|\cdot\|_{\infty}$ stands for the usual sup-norm over $U$.)
As usual, we write 
$C^\infty (U, \cW):=\cap_{k\in \N} C^k (U, \cW)$
and 
$C_{{\rm b}}^\infty (U, \cW):=\cap_{k\in \N} C_{{\rm b}}^k (U, \cW)$.
The set of $f \in C_{{\rm b}}^k (U, \cW)$ with compact support 
is denoted by $C_{{\rm K}}^k (U, \cW)$ for $k \in \N_0 \cup \{\infty\}$.
When $\cW =\R$, we simply write $C^k (U)$, $C_{{\rm b}}^k (U)$,
$C_{{\rm K}}^k (U)$, etc. for simplicity.

\item
Let  $U \subset \cV$ be a domain
and $\gamma >0$. We write $\gamma = k+\alpha$ 
for $k \in \N$ and $\alpha \in (0,1]$ in a unique way. 
We say $f \colon U\to \cW$ is of ${\rm Lip}^\gamma$ if 
$f\in C^k_{\mathrm{b}}(U, \cW)$ and $\nabla^k f$ is 
$\alpha$-H\"older continuous on $U$. 
The set of all such ${\rm Lip}^\gamma$-functions is 
denoted by ${\rm Lip}^\gamma(U, \cW)$.
The ${\rm Lip}^\gamma$-norm is defined by
\[
\|f\|_{{\rm Lip}^\gamma} := \| f\|_{C_{\mathrm{b}}^k }
+ \sup_{x,y\in U, x\neq y} \frac{|f(x)- f(y)|}{|x-y|^\alpha}.
\]
Note that for $C^k_{\mathrm{b}}(U, \cW)\subsetneq
{\rm Lip}^k (U, \cW)$ if $k\in\N$.

\item
We denote by
$\mathscr{S}(\cV)$ the Schwartz class 
of smooth rapidly decreasing functions on $\cV$.
We denote by $\mathscr{S}^{\prime}(\cV)$
its dual, i.e. the space of tempered Schwartz distributions on $\cV$.

\item
Let $\gamma =1/q \in (1/2,1]$ and $m\in\N$. 
If $x$ belongs to $\cC_0^{\gamma\textrm{-Hld}} (\cV)$ 
or to $\cC_0^{q\textrm{-var}} (\cV)$, then we can define 
\[
S(x)^m_{s,t} := \int_{s\le t_1 \le \cdots\le t_m\le t}  dx_{t_1} \otimes\cdots\otimes dx_{t_m},
\qquad 
(s, t) \in \triangle_T
\]
as an iterated Young integral. 
We call $S(x)^m$ the $m$th signature of $x$.
It is well-known that $S_k (x)_{s,t} :=
(1, S(x)^1_{s,t}, \ldots, S(x)^k_{s,t})\in T^k (\cV)$ and 
Chen's relation holds, that is,
\[
S_k (x)_{s,t} = S_k (x)_{s,u}\otimes S_k (x)_{u,t}, \qquad 0\le s\le u \le t \le T.
\]
Here, $\otimes$ stands for the multiplication in $T^k (\cV)$.

\item
Let $\alpha \in (1/4, 1/2]$ and write $k:=\lfloor 1/\alpha\rfloor$. 
We recall the definition of an $\alpha$-H\"older rough path.
A continuous map 
$X=(1, X^{1},\ldots,  X^{k})\colon \triangle_{T} \to T^k (\cV)$
is called an $\alpha$-H\"older rough path over $\cV$
if 
$\|X^i\|_{i \alpha\textrm{-Hld}} <\infty$ for all $i\in \llbracket 1,k \rrbracket$ 
and Chen's relation
\begin{equation} \label{eq.0711-1}
X_{s, t}=X_{s, u}\otimes X_{u, t},
\qquad
0\le s\le u \le t \le T.
\end{equation}
holds in $T^k (\cV)$. 
The set of all $\alpha$-H\"older rough paths over $\cV$
 is denoted by 
$\Omega_{\alpha\textrm{-Hld}} (\cV)$.
With the distance 
$d_{\alpha\textrm{-Hld}} (X, \hat{X}) :=\sum_{i=1}^k\|X^i - \hat{X}^i\|_{i \alpha\textrm{-Hld}}$,
$\Omega_{\alpha\textrm{-Hld}} (\cV)$ is a complete metric space.
The homogeneous norm of $X$ is defined by 
$\vertiii{X}_{\alpha\textrm{-Hld}}:= \sum_{i=1}^k \|X^i\|_{i\alpha\textrm{-Hld}}^{1/i}$.
The dilation by $\delta \in \R$ is defined by 
$\delta X =(1, \delta X^{1},\ldots,  \delta^k X^{k})$. 
It is clear that 
$\vertiii{\delta X}_{\alpha\textrm{-Hld}}
=|\delta|\cdot\vertiii{X}_{\alpha\textrm{-Hld}}$.
A typical example of rough path 
is $S_k (w)$ for $w \in \cC_0^{\gamma\textrm{-Hld}} (\cV)$ 
with $\gamma \in (1/2,1]$, which is called a natural lift of $w$.
We view $S_k$ as a continuous map from $\cC_0^{\gamma\textrm{-Hld}} (\cV)$ to $\Omega_{\alpha\textrm{-Hld}} (\cV)$ and call it the (rough path) lift map.

\item
Let $\alpha \in (1/4, 1/2]$ and write $k:=\lfloor 1/\alpha\rfloor$. 
We define $G\Omega_{\alpha\textrm{-Hld}} (\cV)$ to be the 
$d_{\alpha\textrm{-Hld}}$-closure
of $S_k (\cC_0^{1\textrm{-Hld}} (\cV))$.
It is called the $\alpha$-H\"older 
geometric rough path space over $\cV$
and is a complete and separable metric space.
It also coincides with the $d_{\alpha\textrm{-Hld}}$-closure
of $S_k (\cC_0^{\gamma\textrm{-Hld}} (\cV))$ for any $1/2 <\gamma \le 1$.
A geometric rough path
 $X \in G\Omega_{\alpha\textrm{-Hld}} (\cV)$ satisfies another important algebraic property called the Shuffle relations.
To explain it, we identify $\cV =\R^d~(d=\dim \cV)$
and the coordinate of $X^i$'s are denoted by $X^{1,p}$, 
$X^{2,pq}$ and $X^{3,pqr}$ ($p,q,r \in \llbracket 1,d\rrbracket$).
Then we have
\begin{equation} \label{eq.0711-2}
X^{1,p}_{s,t}\,X^{1,q}_{s,t}= X^{2,pq}_{s,t}+X^{2,qp}_{s,t}, 
\quad
X^{1,p}_{s,t}\,X^{2,qr}_{s,t}=X^{3,pqr}_{s,t}+X^{3,qpr}_{s,t}+X^{3,qrp}_{s,t}
\end{equation}
for all $p,q,r \in \llbracket 1,d\rrbracket$ and $(s,t) \in \triangle_{T}$.
(When $1/3 <\alpha \le 1/2$, we only have the first formula.)
For basic information on $\alpha$-H\"older geometric rough paths,
the reader is referred to \cite[Chapter 9]{fvbook}. 
\end{itemize}

\begin{remark} 
For $p \in [2,4)$, the geometric rough path space 
$G\Omega_{p\textrm{-var}} (\cV)$ is defined 
as the closure of 
$S_k (\cC_0^{q\textrm{-var}} (\cV))$ 
 with respect to the $p$-variation topology 
for any $1 \le q <2$ in essentially the same way as the $1/p$-H\"older case.
(Since we will not use $G\Omega_{p\textrm{-var}} (\cV)$ 
very often in this work, we do not  explanation it in details.)
For basic information on geometric rough paths 
in the H\"older and the variational topologies,  
the reader is referred to \cite{lclbook} or \cite[Chapter 9]{fvbook}
among others. 
\end{remark}

\section{Preliminaries from rough path theory, Mallavin calculus
and quasi-sure analysis}
\label{sec.RPMal}

In this section we recall several known facts from rough path theory, 
Mallavin calculus and quasi-sure analysis.
Those will be used in the main parts of this paper.
There are no new results in this section.

\subsection{
Review of Besov rough path spaces
}

For the pair of parameters $(\alpha, m)$, we assume the 
following condition:
\begin{equation} \label{cond.0310-1}
\frac14 < \alpha \le \frac12,\quad m \in \N
\quad\mbox{and}\quad
\alpha -\frac{1}{12m} > \frac{1}{\lfloor 1/\alpha\rfloor +1}
\end{equation}
For $(\alpha, m)$ as in \eqref{cond.0310-1}, we define a Besov-type geometric rough path space over a Euclidean space $\cV$. 
In this subsection
we write $k:=\lfloor 1/\alpha\rfloor$ and $\beta := \alpha -1/(12m)$
for simplicity of notation.

A continuous map 
$X=(1, X^{1},\ldots,  X^{k})\colon \triangle_{T} \to T^k (\cV)$
is called an $(\alpha, 12m)$-Besov rough path 
over $\cV$ if 
$\|X^i\|_{i \alpha, 12m/i\textrm{-Bes}} <\infty$ for all $i\in \llbracket 1,k \rrbracket$ and Chen's relation
$
X_{s, t}=X_{s, u}\otimes X_{u, t}
$
holds in $T^k (\cV)$ for all $0\le s\le u \le t\le T$.
Note that the second parameter $12m/i$ is an even integer for every $i$.
(This is important
when Besov rough paths are used in conjunction with Malliavin calculus.)
The set of all $(\alpha, 12m)$-Besov rough paths over $\cV$ is denoted by 
$\Omega_{\alpha, 12m\textrm{-Bes}} (\cV)$
with the distance 
$d_{\alpha, 12m\textrm{-Bes}} (X, \hat{X}) 
:=\sum_{i=1}^k\|X^i - \hat{X}^i\|_{i \alpha, 12m/i\textrm{-Bes}}$.
The homogeneous norm of $X$ is denoted by 
$\vertiii{X}_{\alpha, 12m\textrm{-Bes}}:= 
\sum_{i=1}^k \|X^i\|_{i\alpha, 12m/i\textrm{-Bes}}^{1/i}$.
Clearly,
$\vertiii{\delta X}_{\alpha, 12m\textrm{-Bes}}
=|\delta|\cdot\vertiii{X}_{\alpha, 12m\textrm{-Bes}}$ for $\delta \in \R$.

If $x\in \cC_0^{1\textrm{-Hld}} (\cV)$, then
 $S_k (x) \in \Omega_{\alpha, 12m\textrm{-Bes}} (\cV)$. (We write $X= S_k (x)$ in this paragraph.)
For $x, \hat{x}\in \cC_0^{1\textrm{-Hld}}  (\cV)$, the following estimates are known
(see \cite[Appendix A.2]{fvbook}): There exists a constant 
$C>0$ which depends only on $(\alpha, 12m)$ such that
\begin{align}  
\|X^1 - \hat{X}^1\|_{\beta\textrm{-Hld}}
&\le C \|X^1 - \hat{X}^1\|_{ \alpha, 12m\textrm{-Bes}},
\label{ineq.0311-1}
\\
\|X^2 - \hat{X}^2\|_{2\beta\textrm{-Hld}}
&\le 
C \{ 
\|X^2 - \hat{X}^2\|_{2 \alpha, 6m\textrm{-Bes}}
+
K_1 \|X^1 - \hat{X}^1\|_{ \alpha, 12m\textrm{-Bes}}
\},
\label{ineq.0311-2}
\\
\|X^3 - \hat{X}^3\|_{3\beta\textrm{-Hld}}
&\le C\{  \|X^3 - \hat{X}^3\|_{3 \alpha, 4m\textrm{-Bes}}
\nn\\
&\quad +
K_1 \|X^2 - \hat{X}^2\|_{ 2\alpha, 6m\textrm{-Bes}}
+
K_2 \|X^1 - \hat{X}^1\|_{ \alpha, 12m\textrm{-Bes}}
\}
\label{ineq.0311-3}
\end{align}
for all $x, \hat{x}\in \cC_0^1 (\cV)$,
where we set 
\begin{align} 
K_1 :&= \|X^1\|_{ \alpha, 12m\textrm{-Bes}}
\vee \|\hat{X}^1\|_{ \alpha, 12m\textrm{-Bes}},
\nn\\
K_2 :&= \|X^1\|_{ \alpha, 12m\textrm{-Bes}}^2
\vee \|\hat{X}^1\|_{ \alpha, 12m\textrm{-Bes}}^2
\vee \|X^2\|_{ 2\alpha, 6m\textrm{-Bes}}
\vee \|\hat{X}^2\|_{ 2\alpha, 6m\textrm{-Bes}}.
\nn
\end{align} 
(Note that \eqref{ineq.0311-3} does not exist if $\alpha >1/3$.)

If $\{x(m)\}_{m\in\N}\subset  \cC_0^{1\textrm{-Hld}} (\cV)$ is a sequence 
such that $\{X(m):= S_k (x(m))\}_{m\in\N}$ is Cauchy 
in the $(\alpha, 12m)$-Besov topology, 
its limit $X(\infty)$ belongs to 
$G\Omega_{\beta\textrm{-Hld}} (\cV)$
due to \eqref{ineq.0311-1}--\eqref{ineq.0311-3}
and therefore $X(\infty)$ is continuous on $\triangle_T$ and 
$X(\infty)\in \Omega_{\alpha, 12m\textrm{-Bes}} (\cV)$. 
Define 
\[
G\Omega_{\alpha, 12m\textrm{-Bes}} (\cV)
:= 
\overline{
\left\{ S_k (x) \mid x \in \cC_0^1 (\cV)
\right\}
},
\]
where the closure is taken with respect to the $(\alpha, 12m)$-Besov topology.
As we have seen $G\Omega_{\alpha, 12m\textrm{-Bes}} (\cV)
\subset G\Omega_{\beta\textrm{-Hld}} (\cV)$.
Moreover, \eqref{ineq.0311-1}--\eqref{ineq.0311-3} still hold 
for all $X, \hat{X} \in G\Omega_{\alpha, 12m\textrm{-Bes}} (\cV)$.
By way of construction, $G\Omega_{\alpha, 12m\textrm{-Bes}} (\cV)$
is a complete and separable metric space.

In summary, we have the following locally Lipschitz continuous
\footnote{
In this paper, we say that 
a map from a metric space to another metric space is  
locally Lipschitz continuous if its restriction to 
every bounded subset is Lipschitz continuous.
}
embeddings:
\begin{equation}\label{ineq.0311-4}
G\Omega_{\alpha'\textrm{-Hld}} (\cV)
\hookrightarrow
G\Omega_{\alpha, 12m\textrm{-Bes}} (\cV)
\hookrightarrow
G\Omega_{\beta\textrm{-Hld}} (\cV)
\qquad
(\beta := \alpha -\tfrac{1}{12m}).
\end{equation}
The left embedding holds if $1/(k+1) <\alpha <\alpha' \le 1/k$, while the right 
embedding holds under \eqref{cond.0310-1}.

\subsection{Review of Malliavin calculus and quasi-sure analysis}

Let $(\cW, \cH, \mu)$ be an abstract Wiener space,
that is, ~$(\cW, \|\cdot\|_{\cW})$~is a separable Banach space, ~$(\cH, \|\cdot\|_{\cH})$~ is a separable Hilbert space, $\cH$~is a dense subspace of~$\cW$~and the inclusion map is continuous, and~$\mu$~is a (necessarily unique) Borel probability measure
 on $\cW$ with the property that
\begin{equation}\label{abspro}
\int_{\cW}\exp\Bigl(\sqrt{-1}_{\cW^*}\langle \lambda, w \rangle_{\cW}\Bigr)\mu (d w)=\exp\Bigl(-\frac{1}{2}\|\lambda\|^2_{\cH}\Bigr),
\qquad
\lambda \in \cW^* \subset \cH^*,
\end{equation}
where we have used the fact that $\cW^*$ becomes a dense subspace of $\cH$ when we make the natural identification between $\cH^*$ and $\cH$ itself. 
Hence, $\cW^* \hookrightarrow \cH^* =\cH \hookrightarrow\cW$
and both inclusions are continuous and dense.
A generic element of $\cW$ is denoted by $w$, while 
that of $\cH$ is denoted by $h$ or $k$.

We denote by $\{ \langle k, \,\cdot\, \rangle \mid k\in \cH\}$ 
the family of centered Gaussian random variables 
defined on $\cW$ indexed by $\cH$
(i.e. the homogeneous Wiener chaos of order $1$).
If $\langle k,  \,\cdot\, \rangle_{\cH} \in \cH^*$ 
extends to an element of $\cW^*$,
then the extension coincides with the  random 
variable $\langle k,  \,\cdot\, \rangle$.
(When $\langle k,  \,\cdot\, \rangle_{\cH} \in \cH^*$ does not 
extend to an element of $\cW^*$, $\langle k,  \,\cdot\, \rangle$
is define as  the $L^2$-limit of $\{\langle k_n,  \,\cdot\, \rangle \}_{n=1}^\infty$, where  $\{k_n\}_{n=1}^\infty$ 
is any sequence of $\cH$ such that 
$\langle k_n,  \,\cdot\, \rangle_{\cH}\in \cW^*$ for all $n$ and $\lim_{n\to\infty} \|k_n -k\|_\cH=0$.)
We also denote by $\tau_k \colon \cW \to \cW$ the translation 
$\tau_k (w) =w+k$.
(For basic information on abstract Wiener spaces, see \cite{sh, hu} among others.)

Now let us recall a few important properties of the Wiener chaos.
For $m\in \N_0$, denote by
$\mathscr{C}_m$ the $m$th homogeneous Wiener chaos
(see \cite[Chapter 1]{sh} for example).
It is well-known that $\mathscr{C}_0$ is the space of 
constant functions
and $\mathscr{C}_1 = \{ \langle k,  \,\cdot\, \rangle \mid k\in \cH\}$.
The $m$th inhomogeneous Wiener chaos is denoted by
 $\mathscr{C}^{\prime}_m := \oplus_{i=0}^m  \mathscr{C}_i$.
Heuristically, $\mathscr{C}^{\prime}_m$ is the set of 
all real-valued
``Gaussian polynomials" on $\cW$ of order at most $m$.
The most important fact is the orthogonal decomposition
$L^2 (\mu) =  \oplus_{i=0}^\infty  \mathscr{C}_i$.
Moreover, $\mathscr{C}_m$ is the eigenspace of
the Ornstein-Uhlenbeck operator
 on $L^2 (\mu)$ associated with the eigenvalue $-m$.
If each component of $\R^n$-valued function $F =(F^1, \ldots, F^n)$ on $\cW$ 
belongs to $\mathscr{C}_m$ (resp. $\mathscr{C}^\prime_m$),
we say that $F$ belongs to $\mathscr{C}_m (\R^n)$
(resp. $\mathscr{C}^\prime_m (\R^n)$).
Restricted to each $\mathscr{C}_m~(m\in\N_0)$,
all the $L^p$-norms ($1 < p <\infty$) are equivalent,
that is,
there exists a constant $C=C_{m,p} \ge 1$ such that
\begin{equation}\nn
C^{-1} \| F\|_{L^2}  \le  \| F\|_{L^p} \le  C \| F\|_{L^2},
\qquad
F \in \mathscr{C}_m.
\end{equation}
This is a consequence of the hypercontractivity of the Ornstein-Uhlenbeck
semigroup. For details, see \cite[Theorem 2.14]{sh} among others.

Here we summarize some basic facts in Malliavin calculus,
which are related to Watanabe distributions
(i.e. generalized Wiener functionals) and quasi-sure analysis.
Most of the contents and the notation
in this subsection are found in 
 \cite[Sections V.8--V.10]{iwbk} with trivial modifications.
Also, \cite{sh, nu, hu, mt} are good textbooks of Malliavin calculus.
For basic results of quasi-sure analysis, we refer to \cite[Chapter II]{ma}.
We will use the following notation and facts in the main parts of this paper.
\begin{enumerate}
\item[{\bf (a)}]
Sobolev spaces ${\bf D}_{p,r} ({\cal K})$ of ${\cal K}$-valued 
(generalized) Wiener functionals, 
where ${\cal K}$ is a real separable Hilbert space and
$p \in (1, \infty)$, $r \in {\mathbb R}$.
As usual, we will use the spaces 
${\bf D}_{\infty} ({\cal K})= \cap_{k=1 }^{\infty} \cap_{1<p<\infty} {\bf D}_{p,k} ({\cal K})$, 
$\tilde{{\bf D}}_{\infty} ({\cal K}) 
= \cap_{k=1 }^{\infty} \cup_{1<p<\infty}  {\bf D}_{p,k} ({\cal K})$ of test functionals 
and  the spaces ${\bf D}_{-\infty} ({\cal K}) = \cup_{k=1 }^{\infty} \cup_{1<p<\infty} {\bf D}_{p,-k} ({\cal K})$, 
$\tilde{{\bf D}}_{-\infty} ({\cal K}) = \cup_{k=1 }^{\infty} \cap_{1<p<\infty} {\bf D}_{p,-k} ({\cal K})$ of 
 Watanabe distributions as in \cite{iwbk}.
When ${\cal K} ={\mathbb R}$, we simply write ${\bf D}_{p, r}$, etc.
\item[{\bf (b)}] 
For 
$F =(F^1, \ldots, F^n) \in {\bf D}_{\infty} ({\mathbb R}^n)$, we denote by 
$\sigma^{ij}_F (w) =  \la DF^i (w),DF^j (w)\ra_{{\cal H}}$
 the $(i,j)$-component of Malliavin covariance 
 matrix ($n\in \N$, $1 \le i,j \le n$).
We say that $F$ is non-degenerate in the sense of Malliavin
if $(\det \sigma_F)^{-1} \in  \cap_{1<p< \infty} L^p (\mu)$.
Here, $D$ is the $\cH$-derivative
(the gradient operator in the sense of Malliavin calculus).
If $F \in {\bf D}_{\infty} ({\mathbb R}^n)$
is non-degenerate, its law on $\R^n$ admits a bounded 
and continuous
density $p_F =p_F (y)$ with respect to the Lebesgue measure $dy$,
that is, $\mu \circ F^{-1}= p_F (y)dy$.
(This fact is quite famous. See nice textbooks on Malliavin calculus.
In fact, $p_F \in  \mathscr{S}({\mathbb R}^n)$, 
where $\mathscr{S}({\mathbb R}^n)$ stands for the Schwartz class
of smooth rapidly decreasing functions.)
\item
[{\bf (c)}] Pullback $T(F)=T \circ F \in \tilde{\bf D}_{-\infty}$ of a tempered Schwartz distribution $T \in \mathscr{S}^{\prime}({\mathbb R}^n)$
on ${\mathbb R}^n$
by a non-degenerate Wiener functional $F \in {\bf D}_{\infty} ({\mathbb R}^n)$. 
Loosely speaking, the mapping $T \mapsto T (F)$ is continuous.
The most important example of $T$ is Dirac's delta function.
In that case, ${\mathbb E}[\delta_z (F)] :=\la \delta_z (F),1\ra =p_F (z)$ 
holds for every $z\in\R^n$.
Here, $\la \star, *\ra$ denotes the pairing of ${\bf D}_{-\infty}$ and ${\bf D}_{\infty}$ as usual
(see \cite[Section 5.9]{iwbk}).
The key to prove this pullback is an integration by parts formula 
in the sense of Mallavin calculus.
(Its generalization is given in Item {\bf (d)} below.)

\item[{\bf (d)}]
A generalized version of the integration by parts formula in the sense 
of Malliavin calculus for Watanabe distribution,
which is given as follows (see \cite[p. 377]{iwbk}):
For a non-degenerate Wiener functional
$F =(F^1, \ldots, F^n) \in {\mathbf D}_{\infty} ({\mathbb R}^n)$, 
we denote by $\gamma^{ij}_F (w)$ the $(i,j)$-component of the inverse matrix $\sigma^{-1}_F (w)$.
Note that $\sigma^{ij}_F \in {\mathbf D}_{\infty} $ and
$D \gamma^{ij}_F =- \sum_{k,l} \gamma^{ik}_F ( D\sigma^{kl}_F ) \gamma^{lj}_F $.
Hence, derivatives of $\gamma^{ij}_F$ can be written in terms of
$\gamma^{ij}_F$'s and the derivatives of $\sigma^{ij}_F$'s.
In particular, $\gamma^{ij}_F \in {\mathbf D}_{\infty} $, too.
If
$G \in {\mathbf D}_{\infty}$ and $T \in \mathscr{S}^{\prime} ({\mathbb R}^n)$,
then the following integration by parts holds for all $i~(1 \le i \le n)$:
\begin{align}
{\mathbb E} \bigl[
(\partial_i T \circ F )  G 
\bigr]
=
{\mathbb E} \bigl[
(T \circ F )  \Phi_i (\, \cdot\, ;G)
\bigr],
\label{ipb1.eq}
\end{align}
where $\Phi_i (\,\cdot\, ;G) \in  {\mathbf D}_{\infty}$ is given by 
\begin{align}
\Phi_i (w ;G) &=
\sum_{j=1}^n  D^* \Bigl(   \gamma^{ij }_F(w)  G (w) DF^j(w)
\Bigr)
\nn\\
&=
-
\sum_{j=1}^n
\Bigl\{
-\sum_{k,l =1}^d G(w) \gamma^{ik }_F (w)\gamma^{jl }_F (w)    \la D\sigma^{kl}_F (w),DF^j (w)\ra_{{\cal H}}
\nn\\
&
\qquad\qquad
+
\gamma^{ij }_F (w) \la DG (w),DF^j (w)\ra_{{\cal H}} + \gamma^{ij }_F (w) G (w) LF^j (w)
\Bigr\}.
\label{ipb2.eq}
\end{align}
Note that the expectations in (\ref{ipb1.eq}) are in fact  the generalized ones and that $L = -D^*D$ is the Ornstein-Uhlenbeck operator.
(We often use this formula repeatedly. For this, see Remark \ref{notation_ibp} below.)


\item[{\bf (e)}]
If $\eta \in {\bf D}_{-\infty}$ satisfies that
${\mathbb E}[\eta \, G]:=\la \eta, G \ra \ge 0$ for every non-negative $G \in {\bf D}_{\infty}$,
it is called a positive Watanabe distribution.
According to Sugita's theorem (see \cite[P. 101]{ma}), for every positive Watanabe distribution $\eta$,
there exists a unique finite Borel measure $\mu_{\eta}$ on $\cW$
such that
\begin{equation}\label{eq.sgt_formula}
\la \eta, G\ra = \int_{\cW} \tilde{G} (w)  \mu_{\eta} (dw), 
\qquad G \in {\bf D}_{\infty}
\end{equation}
holds,
where $\tilde{G}$ stands for an $\infty$-quasi-continuous modification of 
$G$.
If $\eta \in {\bf D}_{p, -k}$ is positive, then it holds that 
\begin{equation}\label{sugisugi}
\mu_{\eta}(A) \le \| \eta \|_{\mathbf{D}_{p, -k}} {\rm Cap}_{q,k} (A)
\qquad
\mbox{for every Borel subset $A\subset \cW$,}
\end{equation}
where $p, q  \in (1, \infty)$ with $1/p +1/q =1$, $k \in \N$,
and ${\rm Cap}_{q,k}$ stands for the $(q,k)$-capacity
associated with ${\bf D}_{q,k}$.
(For more details, see \cite[Chapter II]{ma}.)

\item[{\bf (f)}]
Restricted to each $\mathscr{C}^\prime_m (\R^n)~(m\in\N_0, n\in\N)$, 
all $\mathbf{D}_{p,k}$-norms ($2 \le p <\infty$, $k \in \N_0$) are equivalent, that is, 
there exists a constant $C=C_{m, n, p, k} >0$ satisfying that
\begin{equation}\label{eq.0312-2}
\| F\|_{L^2}  \le  \| F\|_{\mathbf{D}_{p,k}} \le  C \| F\|_{L^2},
\qquad
F \in \mathscr{C}^\prime_m (\R^n).
\end{equation}

\end{enumerate}

\begin{remark}\label{notation_ibp}
To use the integration by parts formula \eqref{ipb1.eq}--\eqref{ipb2.eq}
repeatedly, we set the notation.
For $\mathbf{i} =(i_1, \ldots, i_k) \in \{1,\ldots, n\}^k$ ($k \in\N$)
and $G \in {\mathbf D}_{\infty}$,  
we define 
$\Phi_{\mathbf{i}} (w ;G) \in  {\mathbf D}_{\infty}$ 
in a recursive way as follows.
For $k=1$, 
$\Phi_{\mathbf{i}} (w ;G)= \Phi_{i_1} (w ;G)$
is given by \eqref{ipb2.eq} with $i =i_1$.
For $k\ge 2$, we set 
\[
\Phi_{\mathbf{i}} (w ;G)
:= 
\Phi_{i_1} \left(w ;   \Phi_{(i_2, \ldots, i_k)} (\,\cdot\, ;G)  \right).
\]
Notice that this involves $D^l G~(0\le l\le  k)$, but not $D^{k+1} G$.
Then, we have 
\begin{align}
{\mathbb E} \bigl[
(\partial^k_{\mathbf{i}} T \circ F )  G 
\bigr]
=
{\mathbb E} \bigl[
(T \circ F )  \Phi_{\mathbf{i}} (\, \cdot\, ;G)
\bigr],
\label{ipb8.eq}
\end{align}
where we set $\partial^k_{\mathbf{i}} :=\partial_{i_1} \cdots \partial_{i_k}$.
\end{remark}

\begin{remark} \label{rem.SnoISO}
In this remark, we explain the precise meaning of the continuity 
of Watanabe's composition map $T \mapsto T(F)=T \circ F$ 
when $F \in {\bf D}_{\infty} ({\mathbb R}^n)$ is non-degenerate.
It should be noted that $\mathscr{S}^\prime ({\mathbb R}^n)$ is not 
equipped with the usual topology of ``pointwise" convergence.

For $k \in \Z$, we define the $\mathscr{S}_{2k}$-norm 
of $g \in \mathscr{S} ({\mathbb R}^n)$ by
$
\|g\|_{\mathscr{S}_{2k}}= \| (1+|\,\cdot\, |^2 -\Delta/2)^k g\|_\infty$.
Here, $\mathscr{S} ({\mathbb R}^n)$ is the Schwartz space of 
rapidly decreasing functions on $\R^n$, 
$\| \cdot\|_\infty$ is the usual sup-norm and $\Delta$ is the Laplacian on 
$\R^n$.
Set $\mathscr{S}_{2k}({\mathbb R}^n)$ to be the closure of 
$\mathscr{S} ({\mathbb R}^n)$ with respect to this  $\mathscr{S}_{2k}$-norm.
Then, it is known that 
$\{\mathscr{S}_{2k}({\mathbb R}^n)\}_{k\in\Z}$ is non-increasing in $k$
and $\cap_{k \in \N}\mathscr{S}_{2k}({\mathbb R}^n) =\mathscr{S}({\mathbb R}^n)$ and $\cup_{k \in \N}\mathscr{S}_{-2k}({\mathbb R}^n)
 =\mathscr{S}^\prime ({\mathbb R}^n)$. Moreover, 
 $\mathscr{S}_0 ({\mathbb R}^n)$ is the Banach space of 
 continuous functions on $\R^n$ which vanish at $\infty$.

Then, for every $k \in N$ and $1<p<\infty$, 
$g \mapsto g (F)$ extends to a bounded linear map from 
$\mathscr{S}_{-2k}({\mathbb R}^n)$ to ${\bf D}_{p, -2k}$.
(See \cite[Theorem 9.1, pp. 378-379]{iwbk} or \cite[Section 5.4]{mt}.)
This continuity is the main reason why we can
 define $T \circ F$ for an
 arbitrary $T\in \mathscr{S}^\prime ({\mathbb R}^n)$.
\end{remark}

A prominent example of $T\in \mathscr{S}^\prime ({\mathbb R}^n)$
is Dirac's delta function $\delta_z ~(z\in \R^n)$.
Here, we recall some basic properties of 
Watanabe's pullback of the delta functions $\delta_z (F)$ 
when $F \in {\bf D}_{\infty} ({\mathbb R}^n)$ is non-degenerate.

First, recall that, for sufficiently large constant $l \in \N$, the map
\begin{equation}\nn
\R^n \ni z \mapsto \delta_z \in  \mathscr{S}_{-2l}({\mathbb R}^n)
\end{equation}
is continuous and bounded 
(see \cite[Lemma 9.1]{iwbk} and its proof).
Then,  by the continuity of Watanabe's decomposition (see Remark \ref{rem.SnoISO}), the map
\[
\R^n \ni z \mapsto \delta_z (F) \in  {\bf D}_{p, -2l}
\]
is bounded and continuous for every $p \in (1,\infty)$. 
This in turn implies that the map
\begin{equation}\label{conti1}
\R^n \ni z \mapsto {\mathbb E}[\delta_z (F) G] \in  \R
\end{equation}
is bounded and continuous for every $G \in {\bf D}_{\infty}$.
(In fact, this function is known to belong to 
$\mathscr{S} ({\mathbb R}^n)$, but this fact is not used in this paper.)
In particular, the density function 
$p_F (z) := {\mathbb E}[\delta_z (F)] $ of the law of $F$ under $\mu$ is 
continuous and bounded in $z$.

Furthermore, for every $\chi \in \mathscr{S}({\mathbb R}^n)$,
\[
\chi = \int_{\R^n} \chi (z) \delta_z \, dz \, \in \mathscr{S}_{-2l}({\mathbb R}^n).
\]
Here, the right hand side is a Bochner integral.
(Clearly,  $z \mapsto |\chi (z)| \|\delta_z\|_{\mathscr{S}_{-2l}}$ is integrable 
with respect to the Lebesgue measure $dz$ on $\R^n$.)
Then, it immediately follows that 
\begin{equation}\label{eq.1008-1}
 {\mathbb E}[\chi (F) G]=\int_{\R^n} \chi (z){\mathbb E}[\delta_z (F) G]\, dz,
 \qquad 
 G \in {\bf D}_{\infty}, \, \chi \in \mathscr{S}({\mathbb R}^n).
\end{equation}


Let $z\in {\mathbb R}^n$ and let 
$\{ \chi_n \}_{n\in\N} \subset \mathscr{S}({\mathbb R}^n)$ such that
$\lim_{n\to \infty} \int f (\hat{z}) \chi_n (\hat{z})d\hat{z} = f(z)$ for every 
$f\in C_{{\rm b}} (\R^n)$.
By the continuity in \eqref{conti1}, we then have 
\begin{equation} \label{conti2}
\lim_{n\to \infty}  {\mathbb E}[\chi_n (F) G]
={\mathbb E}[\delta_z (F) G],
 \qquad 
 G \in {\bf D}_{\infty}.
\end{equation}

As one can easily guess, Watanabe's pullback of delta functions are positive.
This fact is well-known and can be found in \cite[Proposition 5.4.13]{mt}.
Since this is of particular importance in this paper, we 
provide a short proof, which is different from that in \cite{mt}.

\begin{lemma}\label{lem.jigoku1}
Suppose that $F \in {\bf D}_{\infty} ({\mathbb R}^n)$ is non-degenerate.
Then, $\delta_z (F)$ is a positive Watanabe distribution for every $z \in \R^n$.
\end{lemma}

\begin{proof} 
Take any non-negative $G \in {\bf D}_{\infty}$ and any $z\in \R^n$.
It is sufficient to show that ${\mathbb E}[\delta_z (F) G] \ge 0$. 
Then, there exists 
$\{ \chi_n \}_{n\in\N} \subset \mathscr{S}({\mathbb R}^n)$ such that
{\rm (i)}~$\chi_n$ is non-negative for all $n\in\N$ and {\rm (ii)}~$\lim_{n\to \infty} \int f (\hat{z}) \chi_n (\hat{z})d\hat{z} = f(z)$ for every $f\in C_{{\rm b}} (\R^n)$.
Then, the lemma immediately follows from \eqref{conti2}.
\end{proof}

Intuitively, it should be clear that the measure 
corresponding to $\delta_z (F)$
via Sugita's theorem is concentrated on $F^{-1} (z)$.
The precise statement should be as in the following
 lemma, which turns out to be quite useful when we do quasi-sure analysis.
(This is also well-known and 
written in \cite[p. 198]{tw} for instance, but without proof.)

\begin{lemma}\label{lem.jigoku2}
Suppose that $F \in {\bf D}_{\infty} ({\mathbb R}^n)$ is non-degenerate
and ${\mathbb E}[\delta_z (F)] > 0$.
Denote by $\nu_z$ the finite Borel measure that corresponds to 
the positve distributonion $\delta_z (F)$.
Then, 
\[
\nu_z (\{ w \in \cW\mid \tilde{F} \neq z\})
= (\nu_z \circ  \tilde{F}^{-1} )(\R^n \setminus \{z\}) =0.
\] 
Here, $\tilde{F} $ stands for an $\infty$-quasi-continuous modification of $F$.
\end{lemma}

\begin{proof} 
It is sufficient to show the following: For every non-negative 
$\phi \in C_{{\rm K}}^\infty (\R^n)$ whose support is away from $b$,
we have
\begin{equation}\label{eq.1008-5}
0={\mathbb E}[\delta_z (F) \phi (F)] =
 \int_{\cW} \phi(\tilde{F} (w)) \nu_{z} (dw) 
 = \int_{\R^n} \phi (\hat{z}) \, (\nu_z \circ  \tilde{F}^{-1} )(d\hat{z}).
\end{equation}
We prove this by using \eqref{conti2}.
Take an open neighborhood $O$ of $b$ which does not intersect 
the support of $\phi$.
Then, there exists 
$\{ \chi_n \}_{n\in\N} \subset C_{{\rm K}}^\infty (\R^n)$ such that
{\rm (i)}~$\chi_n$ is non-negative for all $n\in\N$, 
{\rm (ii)}~ the support of $\chi_n$ is included in $O$ for all $n\in\N$
and {\rm (iii)}~$\lim_{n\to \infty} \int f (\hat{z}) \chi_n (\hat{z})d\hat{z} = f(z)$ for every $f\in C_{{\rm b}} (\R^n)$.
Then, \eqref{eq.1008-5} immediately follows from \eqref{conti2}.
\end{proof}

\begin{remark} 
In some of the books cited in this subsection (in particular \cite{iwbk, ma, mt}),  results are formulated on  a special Gaussian space.
However, almost most all of them 
(at least, those that will be used in this paper) still hold true
 on any abstract Wiener space.
\end{remark}

\section{Quasi-sure refinement of 
canonical rough path lift of Gaussian processes}
\label{sec.QSlift}

Let $\xi= (\xi_t)_{t\in [0,T]}$ be a real-valued  
centered continuous Gaussian process.
We denote
the covariance of $\xi$ by $R (s,t) :={\mathbb E} [\xi_s \xi_t]$
$(s, t \in [0,T])$. 
The rectangular increment of $R$ is defined by
\[
\hat{R}
\begin{pmatrix}
s, \,t \\
u, \, v
\end{pmatrix}
:=
{\mathbb E} [(\xi_t -\xi_s)(\xi_v -\xi_u)]
=R(t,v) -R(s,v)  - R(t,u) +R(s,u)
\]
for $s, t, u, v \in [0,T]$.
For $\rho \ge 1$ and $[s,t]\times [s', t'] \subset [0,T]^2$,
the 2D $\rho$-variation norm of $\hat{R}$ on $[s,t]\times [s', t']$
is defined by 
\begin{equation}\label{def.2Dvar}
V_\rho (\hat{R}, [s,t]\times [s', t']) 
:=
\sup_{\pi, \pi'} \left(
\sum_{j=1}^N \sum_{k=1}^M \left|
\hat{R}
\begin{pmatrix}
t_{j-1}, \, t_j \\
u_{k-1}, \, u_k
\end{pmatrix}
\right|^\rho
\right)^{1/\rho},
\end{equation}
where  the supremum runs over all the partitions
$\pi=\{s=t_0 <t_1 <\cdots <t_N =t\}$ and 
$\pi'=\{s'=u_0 <u_1 <\cdots <u_M =t'\}$.
We say that $R$ has finite 2D H\"older-controlled $\rho$-variation 
if there exists a constant $C>0$ such that 
$
V_\rho (\hat{R}, [s,t]^2) \le C (t-s)^{1/\rho}
$
holds for all $(s, t)\in \triangle_T$.
Note that if $\rho \le \rho'$ and 
$R$ has finite 2D H\"older-controlled $\rho$-variation, 
then $R$ has finite 2D H\"older-controlled $\rho'$-variation.

Let $(w_t)_{t\in [0,T]} =(w^1_t, \ldots, w^d_t)_{t\in [0,T]}$ be 
a $d$-dimensional  
centered continuous Gaussian process with independent components.
Since we consider a rough path lift,
 we may assume $w_0 =0$, a.s. without loss of generality.
We say that the covariance of $w$
has finite 2D H\"older-controlled $\rho$-variation 
if the covariance of each component $w^i~(i \in \llbracket 1, d\rrbracket)$
has finite 2D H\"older-controlled $\rho$-variation.

We denote the law and the Cameron-Martin space of $(w_t)$ 
by $\mu$ and $\mathcal{H}$, respectively.
The former is a probability measure on $\cC_0 (\R^d)$
and the latter is a linear subspace of $\cC_0 (\R^d)$.
If we set $\mathcal{W}:=\overline{\mathcal{H}}$, where the closure is taken in $\cC_0 (\R^d)$, then $(\mathcal{W}, \mathcal{H}, \mu)$ becomes an abstract Wiener space.
In what follows, we may and will assume that $(w_t)_{t\in [0,T]}$ is 
canonically realized 
(i.e. the coordinate process) on $(\mathcal{W}, \mathcal{H}, \mu)$.

For $l\in\N$, set $\mathcal{P}_l =\{ jT/2^l \mid 
 j\in \llbracket 0, 2^l\rrbracket \}$
be the partition of $[0,T]$ with equal length $T/2^l$.
For $w \in \cC_0 (\R^d)$, $w(l)$ is the piecewise linear 
approximation of $w$ associated with $\mathcal{P}_l$, that is,
$w(l)_{jT/2^l} = w_{jT/2^l}$ for each $j\in \llbracket 0, 2^l\rrbracket$
and $w(l)$ is linearly interpolated on each subinterval
$[(j-1)T/2^l, jT/2^l]~(j\in \llbracket 1, 2^l\rrbracket)$.
For $\alpha \in (1/4, 1/2]$ and $l\in\N$,  we set
$W (l) :=S_{\lfloor 1/\alpha\rfloor} (w(l))$, which is 
an $G\Omega_{\alpha\textrm{-Hld}} (\R^d)$-valued random variable 
under $\mu$.
We set 
\[
\mathcal{Z}_{\alpha\textrm{-Hld}}
:=
\{ w \in  \mathcal{W}\mid  \mbox{$\lim_{l\to\infty} W (l)$ exists in 
$G\Omega_{\alpha\textrm{-Hld}} (\R^d)$}
\}
\]
and 
\begin{equation} \label{def.0314-1}
\mathcal{S}_{\alpha\textrm{-Hld}} (w) 
:=
\begin{cases}
 \lim_{l\to\infty} W(l) & \mbox{(if $w \in \mathcal{Z}_{\alpha\textrm{-Hld}}$),}\\
{\bf 0} &  \mbox{(otherwise),}
\end{cases}
\end{equation}
where ${\bf 0}$ stands for the zero rough path
(in some literarture, it is denoted by $\mathbf{1}$).
We view $\mathcal{S}_{\alpha\textrm{-Hld}}$ is a measurable map
from $\mathcal{W}$ to $G\Omega_{\alpha\textrm{-Hld}} (\R^d)$.
Also, for $(\alpha, m)$ satisfying \eqref{cond.0310-1},
we set $\mathcal{Z}_{\alpha, 12m\textrm{-Bes}}$ and 
$\mathcal{S}_{\alpha, 12m\textrm{-Bes}}$ in essentially the same way.
We will often write 
$W= \mathcal{S}_{\alpha\textrm{-Hld}}$ or 
$W=\mathcal{S}_{\alpha, 12m\textrm{-Bes}}$ without the parameters
of functional spaces for notational simplicity.
We have $\cC^{1\textrm{-Hld}}_0 (\R^d)\subset \mathcal{Z}_{\alpha\textrm{-Hld}}$ and 
$\mathcal{S}_{\alpha\textrm{-Hld}} (w)=S_{\lfloor 1/\alpha\rfloor} (w)$
for every $w\in \cC^{1\textrm{-Hld}}_0 (\R^d)$
since
$\lim_{l\to\infty}\| w(l) -w\|_{\delta\textrm{-Hld}} =0$ for every $\delta \in (0,1)$.
An analogous fact holds in the Besov case, too.


It is natural to ask whether $\{W (l)\}_{l\in\N}$ is convergent 
in the geometric rough path space in an appropriate  sense
or how large the subsets $\mathcal{Z}_{\alpha\textrm{-Hld}}$
and $\mathcal{Z}_{\alpha, 12m\textrm{-Bes}}$ are.
Concerning this, the following result is available
 (see \cite[Theorem 15.42]{fvbook}).
The limit
$W:= \lim_{m\to\infty} W(m)$ is called a canonical lift of
the Gaussian process $w=(w_t)$.

\begin{proposition} \label{prop.fv1542}
Assume that the covariance of $w$
has finite 2D H\"older-controlled $\rho$-variation for $\rho \in [1,2)$.
Take any $\alpha \in (1/4, 1/2]$ with 
$2\rho < 1/\alpha <\lfloor 2\rho\rfloor+1$.
Then, there exists a constant $\kappa \in (0,1)$ with
the following property:
For every $q \in [1, \infty)$, 
there exists a constant $C=C_q >0$ independent of $l, l'$ such that
\[
\| d_{\alpha\textrm{-}{\rm Hld}} (W(l), W(l')) \|_{L^q} 
\le C\kappa^{l\wedge l'},
\qquad l, l'\in\N
\]
holds. 
In particular, 
$W:= \lim_{l\to\infty} W(l)$ exists a.s. (i.e. $\mu (\mathcal{Z}_{\alpha\textrm{-}{\rm Hld}})=1$)
and in $L^q$-sense for every $q \in [1,\infty)$.
\end{proposition}

Hence, under the assumptions of the above proposition, we have
\[
\bigl\| \|W^i\|_{i\alpha\textrm{-}{\rm Hld}} \bigr\|_{L^q}\vee 
\sup_{l\in \N}\bigl\| \|W(l)^i\|_{i\alpha\textrm{-}{\rm Hld}} \bigr\|_{L^q}
<\infty
\]
for every $q\in [1,\infty)$ and $i=1, \ldots, \lfloor 1/\alpha\rfloor$.
It immediately follows that if $(\alpha, m)$ satisfies
\begin{equation} \label{cond.0315-1}
\frac{1}{ \lfloor 2\rho\rfloor+1} 
<\alpha -\frac{1}{12m}
< \alpha < \frac{1}{2\rho}
\qquad\mbox{and}\qquad m \in \N,
\end{equation}
then we have 
\begin{equation} \label{ineq.0315-2}
\bigl\| \|W^i\|_{i\alpha, 12m/i\textrm{-}{\rm Bes}} \bigr\|_{L^q}\vee 
\sup_{l\in \N}
\bigl\| \|W(l)^i\|_{i\alpha, 12m/i\textrm{-}{\rm Bes}} \bigr\|_{L^q}
<\infty
\end{equation}
for every $q\in [1,\infty)$ and $i=1, \ldots, \lfloor 1/\alpha\rfloor$.

\begin{corollary} \label{cor.0315-3}
Assume that the covariance of $w$
has finite 2D H\"older-controlled $\rho$-variation for $\rho \in [1,2)$.
Let $(\alpha, m)$ satisfy \eqref{cond.0315-1}.
Then, there exists a constant $\kappa \in (0,1)$ with
the following property:
For every $q \in [1, \infty)$, 
there exists a constant $C=C_q >0$ independent of $l, l'$ such that
\[
\| d_{\alpha, 12m\textrm{-}{\rm Bes}} (W(l), W(l')) \|_{L^q} 
\le C\kappa^{l\wedge l'},
\qquad l, l'\in\N
\]
holds. 
In particular, 
$W:= \lim_{l\to\infty} W(l)$ exists a.s. 
(i.e. $\mu (\mathcal{Z}_{\alpha, 12m\textrm{-}{\rm Bes}})=1$)
and in $L^q$-sense for every $q \in [1,\infty)$.
\end{corollary}

\begin{proof} 
This inequality is immediate from Proposition \ref{prop.fv1542}.
\end{proof}


In the following theorem and its corollary,
we provide main results of this section.
They claim that the canonical lift map for Gaussian processes 
admits a quasi-sure refinement and, moreover,
 is $\infty$-quasi-continuous.
The $\infty$-quasi-continuity will play an indispensable role 
when we define fractional diffusion bridges in the next section.

\begin{theorem} \label{thm.0321}
Assume that the covariance of $w$
has finite 2D H\"older-controlled $\rho$-variation for $\rho \in [1,2)$.
Let $(\alpha, m)$ satisfy \eqref{cond.0315-1}.
Then, the complement of 
$\mathcal{Z}_{\alpha, 12m\textrm{-}{\rm Bes}}$ is a slim set.
Moreover, 
$\mathcal{S}_{\alpha, 12m\textrm{-}{\rm Bes}}\colon \mathcal{W} \to 
G\Omega_{\alpha, 12m\textrm{-}{\rm Bes}} (\R^d)$ is $\infty$-quasi-continuous.
\end{theorem}

\begin{proof} 
In this proof $(\alpha, m)$ is basically fixed.
For $l\in\N$ and $i=1, \ldots, \lfloor 1/\alpha\rfloor$, we set 
\[
\mathcal{A}^i_{l} := \left\{w \in \mathcal{W}
\,\middle| \,
\|W(l+1)^i -W(l)^i\|_{i\alpha, 12m/i\textrm{-}{\rm Bes}}^{12m/i} 
> (l^{-2})^{12 m/i }
\right\}.
\]
For every $l \in \N$, 
$\|W(l+1)^i -W(l)^i\|_{i\alpha, 12m/i\textrm{-}{\rm Bes}}^{12m/i}  
\in \mathscr{C}^\prime_{12m/i}$
and 
$w \mapsto W(l)=S_{\lfloor 1/\alpha\rfloor} (w(l))$ is continuous
from $\mathcal{W}$ to 
$G\Omega_{\alpha, 12m\textrm{-}{\rm Bes}} (\R^d)$.
Hence, $\mathcal{A}^i_{l}$ is open in $\cW$.

Pick $q \in (1, \infty)$ and $k \in \N_0$ arbitrarily. 
By the Chebyshev-type inequality for the capacities 
(see \cite[P. 90, Lemma 1.2.5]{ma}), \eqref{eq.0312-2} and Corollary \ref{cor.0315-3}, we have
\begin{align}  
\mathrm{Cap}_{q,k} (\mathcal{A}^i_{l}) 
&\le
M_{q,k} l^{24 m/i } 
\left\| 
\|W(l+1)^i -W(l)^i\|_{i\alpha, 12m/i\textrm{-}{\rm Bes}}^{12m/i} 
\right\|_{\mathbf{D}_{q,k}}
\nn\\
&\le
M'_{q,k} l^{24 m/i } 
\left\| 
\|W(l+1)^i -W(l)^i\|_{i\alpha, 12m/i\textrm{-}{\rm Bes}}^{12m/i} 
\right\|_{L^2}
\nn\\
&\le
M'_{q,k} C_2  l^{24 m/i }  \kappa^l.
\label{ineq.0321-1}
\end{align}
Here, $M_{q,k}$ and $M'_{q,k}$ are positive constants
depending only on $(q, k)$ and
$C_2>0$ and $\kappa \in (0,1)$ are constants in Corollary \ref{cor.0315-3}.
None of these four constants depend on $l$, which 
implies that 
\[
\sum_{l\in\N}\mathrm{Cap}_{q,k} (\mathcal{A}^i_{l}) <\infty.
\]
Using the Borel-Cantelli lemma for the capacities (see \cite[P. 90, Corollary 1.2.4]{ma}), 
we can  see from the above fact that
\[
\mathrm{Cap}_{q,k}\left(
\limsup_{l\to\infty}\mathcal{A}^i_{l}
\right) =0.
\]

If $w \in \cap_{i=1, \ldots, \lfloor 1/\alpha\rfloor} (\limsup_{l\to\infty}\mathcal{A}^i_{l})^c$,
there exists $l_0 (w) \in \N$ such that
\[
\|W(l+1)^i -W(l)^i\|_{i\alpha, 12m/i\textrm{-}{\rm Bes}} 
\le l^{-2}
\quad
\mbox{for all $l \ge l_0 (w)$ and $i=1, \ldots, \lfloor 1/\alpha\rfloor$,}
\]
from which it immediately follows that
\[
\sum_{l\in\N}\|W(l+1)^i -W(l)^i\|_{i\alpha, 12m/i\textrm{-}{\rm Bes}} <\infty
\quad
\mbox{for all  $i=1, \ldots, \lfloor 1/\alpha\rfloor$.}
\]
Note that 
the above condition implies that $\{W(l)\}_{l\in\N}$ is Cauchy in 
$G\Omega_{\alpha, 12m\textrm{-}{\rm Bes}} (\R^d)$.
Hence, 
$\cap_{i=1, \ldots, \lfloor 1/\alpha\rfloor} (\limsup_{l\to\infty}\mathcal{A}^i_{l})^c \subset \mathcal{Z}_{\alpha, 12m\textrm{-}{\rm Bes}}$ 
and 
$\mathrm{Cap}_{q,k} ((\mathcal{Z}_{\alpha, 12m\textrm{-}{\rm Bes}})^c) =0$.
Since $(q, k)$ is arbitrary, $\mathcal{Z}_{\alpha, 12m\textrm{-}{\rm Bes}}$
is slim.

Next we show the $(q, k)$-quasi-continuity of 
$\mathcal{S}_{\alpha, 12m\textrm{-}{\rm Bes}}$.
Let $\ve>0$ be given.
We can easily see from \eqref{ineq.0321-1} that
there exists $\ell (\ve) \in\N$ such that
\[
\mathrm{Cap}_{q,k} \left( \cup_{i=1, \ldots, \lfloor 1/\alpha\rfloor} \cup_{l= \ell (\ve)}^\infty \mathcal{A}^i_{l} 
\right)
\le 
\sum_{i=1, \ldots, \lfloor 1/\alpha\rfloor}
\sum_{l= \ell (\ve)}^\infty
\mathrm{Cap}_{q,k} (\mathcal{A}^i_{l}) <\ve.
\]
On the complement of the open subset 
$\cup_{i=1, \ldots, \lfloor 1/\alpha\rfloor} \cup_{l= \ell (\ve)}^\infty \mathcal{A}^i_{l}$,  we have
\[
\|W(l+1)^i -W(l)^i\|_{i\alpha, 12m/i\textrm{-}{\rm Bes}} 
\le l^{-2}
\quad
\mbox{for all $l \ge \ell (\ve)$ and $i=1, \ldots, \lfloor 1/\alpha\rfloor$.}
\]
This implies that the sequence $\{W(l)\}_{l\in\N}$ of continuous maps
from $\cW$ to $G\Omega_{\alpha, 12m\textrm{-}{\rm Bes}} (\R^d)$
uniformly converges on the complement.
Therefore, $\mathcal{S}_{\alpha, 12m\textrm{-}{\rm Bes}}$ is 
also continuous restricted to the complement.
Thus, we have shown the $(q, k)$-quasi-continuity for every $(q, k)$.
This completes the proof of Theorem \ref{thm.0321}.
\end{proof}

\begin{corollary} \label{cor.0321}
Assume that the covariance of $w$
has finite 2D H\"older-controlled $\rho$-variation for $\rho \in [1,2)$.
Take any $\alpha \in (1/4, 1/2]$ with 
$2\rho < 1/\alpha <\lfloor 2\rho\rfloor+1$.
Then, the complement of 
$\mathcal{Z}_{\alpha\textrm{-}{\rm Hld}}$ is a slim set.
Moreover, 
$\mathcal{S}_{\alpha\textrm{-}{\rm Hld}}\colon \mathcal{W} \to 
G\Omega_{\alpha\textrm{-}{\rm Hld}} (\R^d)$ is $\infty$-quasi-continuous.
\end{corollary}

\begin{proof} 
We can find $(\alpha', m)$ satisfying both \eqref{cond.0315-1}
and $\alpha' - (12m)^{-1}\ge \alpha$.
With the Besov-H\"older embedding \eqref{ineq.0311-4} at hand, 
we just need to use Theorem \ref{thm.0321} for $(\alpha', m)$.
\end{proof}

Under the assumptions of Corollary \ref{cor.0321},
$F \circ \mathcal{S}_{\alpha\textrm{-}{\rm Hld}}= F(W)$ is 
also $\infty$-quasi-continuous on $\mathcal{W}$
if $F$ is a continuous map from $G\Omega_{\alpha\textrm{-}{\rm Hld}} (\R^d)$ to another topological space.
A prominent example of such an $F$ is the Lyons-It\^o map 
associated with an RDE.

Let $V_j~(j \in \llbracket 0, d\rrbracket)$ are vector fields on $\R^e$
(which are here viewed as $\R^e$-valued functions on $\R^e$) and 
consider the following controlled ODE: 
\begin{equation}\label{rde1}
dy_t = \sum_{j=1}^d V_j (y_t) dx^j_t  + V_0 (y_t) dt,
\qquad 
y_0 =a
\end{equation}
The initial value $a\in\R^e$ is arbitrary, but fixed.
The super-index $j$ on $dx^j$ stands for the coordinate of $\R^d$,
not the level of rough path.
Here, we understand \eqref{rde1} as an RDE,
namely, a controlled ODE controlled by a geometric rough path $X$.
If a unique global solution $y$ exists for $X$, we write $y:= \Phi_a (X)$.
(Since we only consider the first level of a solution, 
any formulation of RDEs will do.)

If $V_j~(j\in \llbracket 1, d\rrbracket)$ are of ${\rm Lip}^{\gamma}$
and 
$V_0$ is of ${\rm Lip}^{1+\delta}$ for some $\gamma >\alpha^{-1}$
and $\delta >0$,
 then RDE  \eqref{rde1} has a unique global solution 
 $y =\Phi_a (X)$ exists for every 
 $X\in G\Omega_{\alpha\textrm{-}{\rm Hld}} (\R^d)$ and, moreover,
 the Lyons-It\^o map 
$\Phi_a \colon G\Omega_{\alpha\textrm{-}{\rm Hld}} (\R^d)\to 
\tilde\cC^{\alpha\textrm{-Hld}}_a (\R^e)$ is continuous
(see \cite[Theorem 12.10 and Remark 12.7 (i)]{fvbook} for example.)

 Summarizing the above argument, we have the following corollary.
Note that $\Phi_a  \circ \mathcal{S}_{\alpha\textrm{-}{\rm Hld}}=\Phi_a  (W)$
is the solution of  \eqref{rde1} when the driving rough path $X$ is 
replaced by the Gaussian rough path 
$W =  \mathcal{S}_{\alpha\textrm{-}{\rm Hld}} (w)$.

\begin{corollary} \label{cor.0904}
Assume that the covariance of $w$
has finite 2D H\"older-controlled $\rho$-variation for $\rho \in [1,2)$.
Assume further that 
$V_j~(j \in \llbracket 1, d\rrbracket)$ are of ${\rm Lip}^{\gamma}$
for some $\gamma > 2\rho$
and $V_0$ is of ${\rm Lip}^{1+\delta}$ for some $\delta >0$.
Then, for every $\alpha \in (1/4, 1/2]$ with 
$2\rho < 1/\alpha < \gamma \wedge (\lfloor 2\rho\rfloor+1)$,
the correspondence 
$w \mapsto \Phi_a \circ \mathcal{S}_{\alpha\textrm{-}{\rm Hld}}=\Phi (W)$
is an $\infty$-quasi-continuous map 
from $\mathcal{W}$ to 
$\tilde\cC^{\alpha\textrm{-Hld}}_a (\R^e)$.
\end{corollary}


Now we provide an example of Gaussian process that 
satisfies the assumptions of Theorem \ref{thm.0321}
and Corollary \ref{cor.0321}.
Let $d\in\N$ and $\vec{H} =(H_1,\ldots, H_d) \in (0,1)^d$.
We set $\underbar{H}= \min_{1\le j\le d}H_j$. 
A $d$-dimensional continuous stochastic process 
$(w_t)_{t\in [0,T]} =(w^1_t, \ldots, w^d_t)_{t\in [0,T]}$ are 
called a mixed FBM with 
 Hurst parameter $\vec{H}$ if 
all components are independent one-dimensional
centered Gaussian processes whose covariances are given by 
\[
\mathbb{E} [w^j_s w^j_t] =\frac12 (s^{2H_j} + t^{2H_j}  - |t-s|^{2H_j} ),
\qquad
 s, t \in [0,T], \, j \in \llbracket 1, d\rrbracket.
\]
In other words, $w$ is a collection of 
 independent one-dimensional FBM's with the Hurst parameter $H_j$. 
The law of this process is denoted by $\mu =\mu^{\vec{H}}$, 
which is a centered Gaussian measure on $\cC_0 (\R^d)$.
The associated Cameron-Martin space is 
denoted by $\mathcal{H}^{\vec{H}}$. The triplet 
$(\cC_0 (\R^d), \mathcal{H}^{\vec{H}}, \mu^{\vec{H}})$ is an abstract Wiener space.
Obviously, the coordinate process on $\cC_0 (\R^d)$
is a standard realization of the mixed FBM with 
the Hurst parameter $\vec{H}$.

\begin{proposition} \label{prop.0322-1}
Let $w =(w_t)_{t\in [0,T]}$ be a mixed FBM with the Hurst parameter $\vec{H}$ as above and assume that $\underbar{H}\in (1/4, 1)$.
Then,  the covariance of $w$
has finite 2D H\"older-controlled $(2\underbar{H})^{-1} \vee 1$-variation.
\end{proposition}

\begin{proof} 
This claim reduces the one-dimensional case.
Hence, we may assume that $w$ is a one-dimensional FBM
with Hurst parameter $H \in (1/4, 1)$.
When $H \in (1/4, 1/2]$, it is shown in \cite[Example 1]{fv_2Dvar} that 
the covariance
has finite 2D H\"older-controlled $(2H)^{-1}$-variation.
Therefore, it suffices to check that the covariance
has finite 2D H\"older-controlled $1$-variation when $H \in (1/2, 1)$.
Recall that $w$ is positively correlated in this case, i.e., 
${\mathbb E}[w^1_{u,v} w^1_{s,t}]\ge 0$ if $u\le v \le s \le t$.
Here, we set $w^1_{s,t}:= w_t- w_s$.

Let both $\pi$ and $\pi'$ be partitions of $[s,t] \subset [0,T]$.
Since the sum on the right hand side of \eqref{def.2Dvar}
satisfies sub-additivity when $\rho =1$, 
the sum associated with $(\pi, \pi')$ is dominated by 
the sum associated with $(\pi\cup \pi', \, \pi\cup \pi')$.
Therefore, when we estimate $V_1 (\hat{R}, [s,t]^2)$, we may only 
consider the case $\pi =\pi'$ in \eqref{def.2Dvar}.

Due to the positive correlation, it holds 
for $\pi=\{s=t_0 <t_1 <\cdots <t_N =t\}$ that
\begin{align}  
\sum_k \Bigl|
{\mathbb E}[w^1_{t_{j-1}, t_j } w^1_{t_{k-1}, t_k}]
\Bigr|
=
\Bigl|
\sum_k
{\mathbb E}[w^1_{t_{j-1}, t_j } w^1_{t_{k-1}, t_k}]
\Bigr|
\le
\Bigl| 
{\mathbb E}[w^1_{t_{j-1}, t_j } w^1_{s,t}]
\Bigr|
\nn
\end{align}
for every $j$.
Hence, it suffices to prove that there exists a positive constant 
$C=C_H$ such that, for every $[u, v] \subset [s,t]$,   
\begin{equation}  
\bigl| 
{\mathbb E}[w^1_{u,v} w^1_{s,t}]
\bigr|
\le C_H (v-u)
\label{ineq.0322-1}
\end{equation} 
holds.

By straightforward computation, we have 
\begin{align}  
2\bigl| 
{\mathbb E}[w^1_{u,v} w^1_{s,t}]
\bigr|
&=
|
-(t-v)^{2H} - (u-s)^{2H} + (v-s)^{2H} + (t-u)^{2H} 
|
\nn\\
&=
\{ (t-u)^{2H}  - (t-v)^{2H} \} 
+
\{ (v-s)^{2H}  - (u-s)^{2H} \} 
\nn\\
&=
(v-u )\int_0^1 2H\{t-v+\theta (v-u)\}^{2H-1}d\theta
\nn\\
&\qquad 
+
(v-u )\int_0^1 2H\{u-s+\theta (v-u)\}^{2H-1}d\theta
\nn\\
&\le
4H T^{2H-1}(v-u ).
\nn
\end{align}
Here, we used the mean-value theorem and $2H >1$. 
Thus, we obtained \eqref{ineq.0322-1}, which completes the proof of the proposition.
\end{proof}

\begin{remark} \label{rem.fbm}
The usual
$d$-dimensional FBM with Hurst parameter
$H\in (1/4, 1/2]$ satisfies the 
assumptions of Proposition \ref{prop.0322-1}.
Therefore, Theorem \ref{thm.0321} and Corollary \ref{cor.0321} 
are available for this Gaussian process with $(2\rho)^{-1} =H$.
\end{remark}

\begin{remark}
Quasi-sure lift of Brownian rough path was first introduced by \cite{in0}
with respect to the $p$-variation topology ($2<p<3$).
Quasi-sure lift with respect to the $\alpha$-H\"odler topology ($1/3<\alpha<1/2$) and the $\infty$-quasi-continuity was 
first proved by \cite{aida}. (See also \cite{in1}.) 
Those results are used in many papers on diffusion bridges 
such as \cite{in1, in3, in4, in5}.
A quasi-sure non-intersecting property of
the signature process of Brownian motion is obtained in \cite{bglq}.
 
Quasi-sure lift of a certain class of Gaussian  processes,
including FBM with Hurst parameter $H\in (1/4,1/2]$,
was first obtained by \cite{bgq}.
Using this, a quasi-sure non-intersecting property of 
the solution of an RDE driven by fractional Brownian rough path 
with Hurst parameter $H\in (1/4,1/2]$ is proved in \cite{orv}.
These two works use the variation topology.
The $\infty$-quasi-continuity of the lift map 
does not seem to be known, yet.
\end{remark}

\section{Conditioning fractional diffusion processes}
\label{sec.conditionFDP}

The main purpose of this section is defining 
fractional diffusion bridges,
which is one of the two main results of this paper.
Throughout this section, we assume $H \in (1/4, 1/2]$ and $d\in\N$. 

Denote by $\mu^H$ the law of
$d$-dimensional FBM with Hurst parameter $H$,
which is a centered Gaussian measure on $\mathcal{W}=\cC_0 (\R^d)$.
The coordinate process $w=(w_t)_{t\in [0,T]}$ on $\cW$ 
is a canonical realization of FBM under $\mu^H$.
The Cameron-Martin space of $\mu^H$ is denoted by $\cH^H$.
As is well-known, $(\cW, \cH^H, \mu^H)$ is an abstract Wiener space.

We study RDE \eqref{rde1} driven by the canonical lift 
$W:= \mathcal{S}_{\alpha\textrm{-}{\rm Hld}} (w)$ of FBM, 
where $(\lfloor 1/H\rfloor+1)^{-1} < \alpha <H$. It reads:
\begin{equation}\label{rde2}
dy_t = \sum_{j=1}^d V_j (y_t) dw^{j}_t  + V_0 (y_t) dt,
\qquad 
y_0 =a\in \R^e.
\end{equation}
In other words, $y = \Phi_a (W)$.
We denoted by $\mathbb{Q}_{a}^H$ the law of this process $y$, that is,
$\mathbb{Q}_{a}^H := (\Phi_a \circ \mathcal{S}_{\alpha\textrm{-}{\rm Hld}})_* \mu^H$.
Since $\Phi_a$ takes values in the ``little $\alpha$-H\"older space" 
$\tilde{\cC}^{\alpha\textrm{-Hld}}_a (\R^e)$,
this Borel probability measure sits on $\tilde{\cC}^{\alpha\textrm{-Hld}}_a (\R^e)$.

\subsection{Some properties of Cameron-Martin paths}

In this subsection
 we discuss the  Cameron-Martin space $\cH^H$ associated with $\mu^H$.
It is known that there is a continuous embedding
$\cH^H \subset \cC^{q\textrm{-var}}_0 (\R^d)$ for any $q$ with 
\begin{equation}\label{ineq.parameter}
\left( H + \tfrac12 \right)^{-1} < q<2.
\end{equation}
More precisely, for every $q$ as in \eqref{ineq.parameter}, 
there exists a constant $c_1=c_1(q) >0$
independent of $h$ and $(s,t)$
such that
\begin{equation} \label{ineq.0523-2}
\| h\|_{q\textrm{-var}, [s,t]}
\le 
c_1 \| h\|_{\cH^H} (t-s)^{\frac{1}{q} - \frac{1}{2}},
\qquad 
h \in \cH^H, \, (s,t)\in \triangle_T.
\end{equation}
Here, $\|\cdot\|_{q\textrm{-var}, [s,t]}$ stands for the
$q$-variation (semi)norm restricted to the subinterval $[s,t]$.
Note that $\tfrac{1}{q} - \tfrac{1}{2}\nearrow H$
as $q \searrow ( H + \tfrac12 )^{-1}$.
(See \cite{fv_embed} for the facts in this paragraph.)

By \eqref{ineq.0523-2} and the general theory of Young integration,
there exists  a constant $c_2=c_2(q) >0$
independent of $h$ and $(s,t)$ such that
\begin{equation}\label{ineq.0523-5}
|S(h)^i_{s,t} | \le c_2 \| h\|_{q\textrm{-var}, [s,t]}^i 
\le 
c_2c_1^i  \| h\|_{\cH^H}^i  (t-s)^{i (\frac{1}{q} - \frac{1}{2})},
\quad 
h \in \cH^H, \, (s,t)\in \triangle_T
\end{equation}
for all $i=1,2,3$.
This implies that 
$S_{\lfloor 1/H\rfloor} (h)$ is an $\alpha$-H\"older weakly geometric 
rough path for every $\alpha\in (\lfloor 1/H\rfloor+1)^{-1}, \, H)$.
 By \cite[Theorem 8.22]{fvbook} and slight adjustment of the 
 parameter $\alpha$, 
$S_{\lfloor 1/H\rfloor} (h)\in G\Omega_{\alpha\textrm{-}{\rm Hld}} (\R^d)$ 
for every $\alpha\in (\lfloor 1/H\rfloor+1)^{-1}, \, H)$.
By a standard argument, 
$\cH^H \ni h \mapsto S_{\lfloor 1/H\rfloor} (h) \in G\Omega_{\alpha\textrm{-}{\rm Hld}} (\R^d)$ is locally Lipschitz continuous.


Now we consider the piecewise linear approximation of 
a Cameron-Martin path. 
Since $\lim_{l\to\infty} h(l) =h$ in $(q+\delta)$-variation topology
for every sufficiently small $\delta >0$, 
we have $\lim_{l\to\infty} S_{\lfloor 1/H\rfloor} (h(l))
= S_{\lfloor 1/H\rfloor} (h)$ in $G\Omega_{1/\alpha\textrm{-var}} (\R^d)$
for every $h\in \cH^H$.
However, whether 
$h \in \mathcal{Z}_{\alpha\textrm{-}{\rm Hld}}$ is not so obvious
though it looks natural heuristically.

 For the pair of parameters $(\alpha, m)$,  we 
 introduce the following condition:
\begin{equation}\label{cond.0523-3}
m \in \N\qquad \mbox{and}\qquad \frac{1}{ \lfloor 1/H\rfloor+1} 
<\alpha -\frac{1}{12m}
< \alpha < H
\end{equation}
This condition will be frequently used in what follows.

\begin{lemma} \label{lem.0411-1}
Let the notation be as above and assume \eqref{cond.0523-3}. 
Then, we have 
$\cH^H \subset \mathcal{Z}_{\alpha, 12m\textrm{-}{\rm Bes}}$.
In particular, 
$\{S_{\lfloor 1/H\rfloor} (h(l))\}_{l\in\N}$ converges to 
$S_{\lfloor 1/H\rfloor} (h) 
= \mathcal{S}_{\alpha, 12m\textrm{-}{\rm Bes}} (h)$
in $G\Omega_{\alpha, 12m\textrm{-}{\rm Bes}} (\R^d)$ for every $h\in \cH^H$.
\end{lemma}

\begin{proof} 
It suffices to show that 
$h\in \mathcal{Z}_{\alpha, 12m\textrm{-}{\rm Bes}}$,
that is, $\{S_{\lfloor 1/H\rfloor} (h(l))\}_{l\in\N}$ is a Cauchy sequence 
in $G\Omega_{\alpha, 12m\textrm{-}{\rm Bes}} (\R^d)$, 
 for every $h\in \cH^H$.
We only prove the case $1/4 < H\le 1/3$ since the case $1/3 < H\le 1/2$
is much easier.

In this proof, we use the following notation. 
We write $H(l):= S_3 (h(l))$ and $W(l):= S_3 (w(l))$.
Let $C>0$ and $\kappa \in (0,1)$ be constants 
independent of  $l$ and $h$, which may vary from line to line.
We write
$i, j, k \in \llbracket 1, d\rrbracket$ for coordinates of $\R^d$.
(These parameters do not stand for the level of a rough path.)

First, we calculate the first level paths.
For every $l\in \N$, we set 
\[
F^{1,i}_l (w) := 
\|W(l+1)^{1,i} -W(l)^{1,i}\|_{\alpha, 12m\textrm{-}{\rm Bes}}^{12m} 
=
\iint_{ \triangle_T}   \frac{| W(l+1)^{1,i}_{s,t} -W(l)^{1,i}_{s,t}|^{12m} }{(t-s)^{1+12\alpha m}}dsdt.
\]
Then, $F^{1,i}_l$ is continuous in $w$ and belongs to $\mathscr{C}^{\prime}_{12m}$.
By Corollary \ref{cor.0315-3}, there exists $\kappa \in (0,1)$
and $C >0$ such that
\[
\| F^{1,i}_l\|_{L^2} \le C \kappa^l, \qquad  l\in\N.
\]
Noting that $D_h \{ W(l+1)^{1,i}_{s,t} -W(l)^{1,i}_{s,t} \}
=H(l+1)^{1,i}_{s,t} -H(l)^{1,i}_{s,t}$ for every $h\in \cH^H$
and $12m$ is an even integer, we can easily see that 
\begin{align*}
\langle D^{12m}F^{1,i}_l(w), \, h^{ \otimes 12m}  \rangle
&=
(D_h)^{12m} F^{1,i}_l(w) 
\\
&=
C\iint_{ \triangle_T}   \frac{| H(l+1)^{1,i}_{s,t} -H(l)^{1,i}_{s,t}|^{12m} }{(t-s)^{1+12\alpha m}}dsdt
\\
&=
C \|H(l+1)^{1,i} -H(l)^{1,i}\|_{\alpha, 12m\textrm{-}{\rm Bes}}^{12m} .
\end{align*}
From \eqref{eq.0312-2} we can also see that
\begin{align*}
\| \langle D^{12m}F^{1,i}_l(w), \, h^{ \otimes 12m}  \rangle\|_{L^2}
\le
\|h\|_{\cH^H}^{12m} \| F^{1,i}_l\|_{{\bf D}_{2,12m}}
\le
C \|h\|_{\cH^H}^{12m} \| F^{1,i}_l\|_{L^2}
\le
 C \|h\|_{\cH^H}^{12m} \kappa^l
\end{align*}
for every $l\in\N$.
Thus, we have obtained 
\[ 
  \|H(l+1)^{1,i} -H(l)^{1,i}\|_{\alpha, 12m\textrm{-}{\rm Bes}}^{12m}
  \le C \|h\|_{\cH^H}^{12m} \kappa^l, \qquad l\in\N.
\]
This implies that the sequence $\{ H(l)^1\}_{l\in\N}$
is Cauchy with respect to the $(\alpha, 12m)$-Besov topology.

Next, we calculate the second level paths.
For every $l\in \N$, we set 
\[
F^{2,ij}_l (w) := 
\|W(l+1)^{2,ij} -W(l)^{2,ij}\|_{2\alpha, 6m\textrm{-}{\rm Bes}}^{6m} 
=
\iint_{ \triangle_T}   \frac{| W(l+1)^{2,ij}_{s,t} -W(l)^{2,ij}_{s,t}|^{6m} }{(t-s)^{1+12\alpha m}}dsdt.
\]
Then, $F^{2,ij}_l$ is continuous in $w$ and belongs to $\mathscr{C}^{\prime}_{12m}$.
By Corollary \ref{cor.0315-3}, there exists $\kappa \in (0,1)$
and $C>0$ such that
\[
\| F^{2,ij}_l\|_{L^2} \le C \kappa^l, \qquad  l\in\N.
\]
Since $w\mapsto W(l)^{2,ij}_{s,t}$ is a continuous quadratic 
Wiener functional, we can easily see that
\[
(D_h)^2 \{ W(l+1)^{2,ij}_{s,t} -W(l)^{2,ij}_{s,t} \}
=2\{H(l+1)^{2,ij}_{s,t} -H(l)^{2,ij}_{s,t}\}, \qquad h\in \cH^H.
\]
By applying $(D_h)^{12m}$ to $F^{2,ij}_l$, we can obtain 
\[
  \|H(l+1)^{2,ij} -H(l)^{2,ij}\|_{2\alpha, 6m\textrm{-}{\rm Bes}}^{6m}
  \le C \|h\|_{\cH^H}^{12m} \kappa^l, \qquad l\in\N,
\]
essentially in the same way as in the first level case.
Thus, the sequence $\{ H(l)^2\}_{l\in\N}$
is Cauchy with respect to the $(2\alpha, 6m)$-Besov topology.

Finally, we calculate the third level paths.
For every $l\in \N$, we set 
\[
F^{3,ijk}_l (w) := 
\|W(l+1)^{3,ijk} -W(l)^{3,ijk}\|_{3\alpha, 4m\textrm{-}{\rm Bes}}^{4m} 
=
\iint_{ \triangle_T}   \frac{| W(l+1)^{3,ijk}_{s,t} -W(l)^{3,ijk}_{s,t}|^{4m} }{(t-s)^{1+12\alpha m}}dsdt.
\]
Then, $F^{3,ijk}_l$ is continuous in $w$ and belongs to $\mathscr{C}^{\prime}_{12m}$.
By Corollary \ref{cor.0315-3}, there exists $\kappa \in (0,1)$
and $C>0$ such that
\[
\| F^{3,ijk}_l\|_{L^2} \le C \kappa^l, \qquad  l\in\N.
\]
In this case, we have 
\[
(D_h)^3 \{ W(l+1)^{3,ijk}_{s,t} -W(l)^{3,ijk}_{s,t} \}
=6\{H(l+1)^{3,ijk}_{s,t} -H(l)^{3,ijk}_{s,t}\}, \qquad h\in \cH^H.
\]
By applying $(D_h)^{12m}$ to $F^{3,ijk}_l$, we can obtain 
\[
  \|H(l+1)^{3,ijk} -H(l)^{3,ijk}\|_{3\alpha, 4m\textrm{-}{\rm Bes}}^{4m}
  \le C \|h\|_{\cH^H}^{12m} \kappa^l, \qquad l\in\N,
\]
essentially in the same way as in the first level case.
Thus, the sequence $\{ H(l)^3\}_{l\in\N}$
is Cauchy with respect to the $(3\alpha, 4m)$-Besov topology.
This completes the proof of the lemma.
\end{proof}


\begin{corollary} \label{cor.0411-2}
Assume $(\lfloor 1/H\rfloor+1)^{-1} < \alpha <H$ and 
let the notation be as above. 
Then, we have 
$\cH^H \subset \mathcal{Z}_{\alpha\textrm{-}{\rm Hld}}$.
In particular, 
$\{S_{\lfloor 1/H\rfloor} (h(l))\}_{l\in\N}$ converges to 
$S_{\lfloor 1/H\rfloor} (h) = \mathcal{S}_{\alpha\textrm{-}{\rm Hld}} (h)$
in $G\Omega_{\alpha\textrm{-}{\rm Hld}} (\R^d)$ for every $h\in \cH^H$.
\end{corollary}

\begin{proof}
For any such $\alpha$, we can find $(\alpha', m)$ 
which satisfies \eqref{cond.0523-3} and
$\alpha' -(12m)^{-1} >\alpha$.
Then, the corollary follows immediately from Lemma \ref{lem.0411-1} and 
the Besov-H\"older embedding \eqref{ineq.0311-4} .
\end{proof}

For $h\in \cH^H$, consider the following controlled ODE 
in the Young sense:
\begin{equation}\label{ode2}
d\psi^h_t = \sum_{j=1}^d V_j (\psi^h_t) dh_t  + V_0 (\psi^h_t) dt,
\qquad 
\psi^h_0 =a\in \R^e.
\end{equation}
This is called the skeleton ODE associated with RDE \eqref{rde2}.
(Note that this makes sense at least in the $q$-variational setting
since $q<2$.)
We write $\psi^h := \Psi_a (h)$.
Then, $\Psi_a \colon \cH^H \to \cC^{q\textrm{-var}}_a (\R^e)$ is 
a continuous map.

\begin{corollary}\label{cor.0523-1}
(1) Under the assumption of Lemma \ref{lem.0411-1},
$ \Phi_a \circ \mathcal{S}_{\alpha, 12m\textrm{-}{\rm Bes}}(h)= \Psi_a (h)$
for every $h\in \cH^H$.
\\
(2) Under the assumption of Corollary \ref{cor.0411-2}, 
we have $\Phi_a \circ \mathcal{S}_{\alpha\textrm{-}{\rm Hld}}(h)
= \Psi_a (h)$ for every $h\in \cH^H$.
\end{corollary}

\begin{proof}
We prove (2) only because we can prove (1) in the same way.
Let $q \in [1, 2)$ be as in \eqref{ineq.parameter} and 
take $\delta>0$ so small that $q+\delta <2$.
Since $\lim_{l\to\infty} h(l) =h$ in $(q+\delta)$-variation topology,
$\lim_{l\to\infty} \Psi_a ( h(l)) =\Psi_a (h)$ in $(q+\delta)$-variation topology.
On the other hand, $\Psi_a ( h(l))= \Phi_a  \circ S_{\lfloor 1/H\rfloor} (h(l))
\to \Phi_a  \circ 
\mathcal{S}_{\alpha\textrm{-}{\rm Hld}} ( h)$ as $l \to \infty$ by 
Corollary \ref{cor.0411-2}.
Thus, we have obtained $\Phi_a \circ \mathcal{S}_{\alpha\textrm{-}{\rm Hld}}(h)
= \Psi_a (h)$.
\end{proof}

Now, let us recall the Young translation.
Suppose that $\alpha\in (1/4, 1/2]$ and $q \in [1,2)$
satisfy $\alpha + q^{-1} >1$. 
Then, there is a unique continuous map 
$\mathcal{T} \colon 
\cC^{q\textrm{-var}}_0 (\R^d)\times
G\Omega_{1/\alpha\textrm{-var}} (\R^d) 
 \to G\Omega_{1/\alpha\textrm{-var}} (\R^d)$
with the following property:
\[
\mathcal{T}_k ( S_{\lfloor 1/\alpha\rfloor} (x)) = S_{\lfloor 1/\alpha\rfloor} 
(\tau_k (x)), \qquad
x, k \in \cC^{q\textrm{-var}}_0 (\R^d).
\]
Here, we set $\tau_k (x) := x+k$.
This map $\mathcal{T}$ is actually locally Lipschitz continuous.
We call
$\mathcal{T}_k (X)\in G\Omega_{1/\alpha\textrm{-var}} (\R^d)$ the Young translation of $X\in G\Omega_{1/\alpha\textrm{-var}} (\R^d)$ by 
$k \in \cC^{q\textrm{-var}}_0 (\R^d)$.
As is well-known, 
each level of $\mathcal{T}_k (X)$ can be written as 
an iterated Young integral.

In the next lemma, which will ve used in Section \ref{sec.proof.low},
we will check that the Young translation still works 
even if $\cC^{q\textrm{-var}}_0 (\R^d)$ and $G\Omega_{1/\alpha\textrm{-var}} (\R^d)$ are replaced by 
$\cH^H$ and $G\Omega_{\alpha, 12m\textrm{-}{\rm Bes}} (\R^d)$,
respectively.
Before providing a precise statement of the lemma, 
we introduce another condition on $(\alpha, m)$:
\begin{equation}\label{cond.0523-4}
\left( \alpha -\frac{1}{12m} \right)+\left( H +\frac12\right) >1
\qquad \mbox{and}\qquad
 \alpha +\frac{2}{12m} <H.
\end{equation}
Suppose that $H \in (1/4, 1/2]$ is given.
If we take $\alpha$ sufficiently close to $H$ and 
take $m$ sufficiently large (for that $\alpha$), 
then $(\alpha, m)$ satisfies both
\eqref{cond.0523-3} and \eqref{cond.0523-4}.

\begin{lemma} \label{lem.Ytrans}
Let $(\alpha, m)$ satisfy both \eqref{cond.0523-3} 
and \eqref{cond.0523-4}.
Then, the Young translation $\mathcal{T}$ maps
$\cH^H \times   G\Omega_{\alpha, 12m\textrm{-}{\rm Bes}} (\R^d)$
onto $G\Omega_{\alpha, 12m\textrm{-}{\rm Bes}} (\R^d)$.
Moreover, this map is locally Lipschitz continuous.
\end{lemma}

\begin{proof}
We only prove the level 3 case (i.e. $1/4 <H\le 1/3$) since 
the level 2 case (i.e. $1/3 <H\le 1/2$) is much easier.
We write $\beta := \alpha -(12m)^{-1}$.
In this proof, $C_j~(j=1,2,\ldots)$ are positive constants 
which depend only on parameters of (rough) path spaces.

Take $q$ as in \eqref{ineq.parameter} close enough to $(H +\tfrac12)^{-1}$
so that 
\[
\beta +\frac{1}{q} >1 \qquad\mbox{and}\qquad
(i-1)\beta + \frac{1}{q} -\frac12 > i\alpha \quad (i=1,2,3)
\]
 holds.
 Note that 
the first inequality above is the sufficient condition for Young integration.
Let 
$X =(X^1, X^2, X^3) \in G\Omega_{\alpha, 12m\textrm{-}{\rm Bes}} (\R^d)$.
By \eqref{ineq.0311-4}, $X\in G\Omega_{\beta\textrm{-Hld}} (\R^d)$.
So, at least in the variational setting, $\mathcal{T}_k (X)$ is well-defined
for $X$ and $k \in \cH^H$. 
For simplicity,
we write $K =(K^1, K^2, K^3)$ for $S_3 (k)$
and $x_t := X^1_{0,t}$.

It follows immediately from \eqref{ineq.0523-2} and 
$(1/q) -(1/2) >\alpha$ that
\[
\|\mathcal{T}_k (X)^1\|_{\alpha, 12m\textrm{-Bes}}
=\|X^1\|_{\alpha, 12m\textrm{-Bes}}+\|K^1\|_{\alpha, 12m\textrm{-Bes}}
\le 
\|X^1\|_{\alpha, 12m\textrm{-Bes}}+C_1 \| k\|_{\cH^H}  <\infty.
 \]
Next, we first estimate the second level path. Recall that 
\[
\mathcal{T}_k (X)^2_{s,t} =X^2_{s,t}
+
\int_s^t X^1_{s,u} \otimes dk_u +\int_s^t K^1_{s,u} \otimes dx_{u} 
+K^2_{s,t},
\]
where the second and third terms on the right hand side
are Young integrals. 
By the general theory of Young integration, we have 
\begin{align*}
\left| 
\int_s^t X^1_{s,u} \otimes dk_u 
\right| 
+
\left| 
\int_s^t K^1_{s,u} \otimes dx_{u} 
\right| 
&\le 
C_2 \| x\|_{1/\beta\textrm{-var}, [s,t]}\| k\|_{q\textrm{-var}, [s,t]}
\\
&\le
C_3 \| X^1\|_{\beta\textrm{-Hld}} \| k\|_{\cH^H} 
(t-s)^{\beta +\frac{1}{q} - \frac{1}{2}}
\end{align*}
for every $X$, $k$ and $(s,t) \in \triangle_T$.
Note that 
the H\"older exponent on the right hand side is larger than $2\alpha$.
As we have seen in \eqref{ineq.0523-5}, 
$\|K^2\|_{(2\alpha+\delta)\textrm{-Hld}} \le C_4 \|k\|_{\cH^H}^2$ for sufficiently small $\delta >0$.
Hence, we have shown
$\|\mathcal{T}_k (X)^2 -X^2\|_{(2\alpha+\delta)\textrm{-Hld}}<\infty$,
which implies 
$\|\mathcal{T}_k (X)^2\|_{2\alpha, 6m\textrm{-Bes}}<\infty$ since 
$\|X^2\|_{2\alpha, 6m\textrm{-Bes}}<\infty$ by assumption.

Next we will show $\|\mathcal{T}_k (X)^3\|_{3\alpha, 4m\textrm{-Bes}}<\infty$. We can easily see that $\mathcal{T}_k (X)^3$ is written as sum 
of eight terms.
By assumption, $\|X^3\|_{3\alpha, 4m\textrm{-Bes}}<\infty$.
By the same argument as for $K^2$, we have
$\|K^3\|_{3\alpha, 4m\textrm{-Bes}}<\infty$. 
If we set 
\[
A_{s,t} := \int_{s\le  t_1\le  t_2\le  t_3 \le t } \Bigl(
dk_{t_1} \otimes dk_{t_2}\otimes dx_{t_3} 
+ dk_{t_1} \otimes dx_{t_2}\otimes dk_{t_3} 
+dx_{t_1} \otimes dk_{t_2}\otimes dk_{t_3}  \Bigr),
\]
then we see from the general theory of Young integration that
\[
|A_{s,t}|
\le C_5\| x\|_{1/\beta\textrm{-var}, [s,t]}\| k\|^2_{q\textrm{-var}, [s,t]}
\le 
C_6\| X^1\|_{\beta\textrm{-Hld}} \| k\|_{\cH^H}^2 
(t-s)^{\beta +2(\frac{1}{q} - \frac{1}{2})}
\]
for every $X$, $k$ and $(s,t) \in \triangle_T$.
The H\"older exponent on the right hand side is larger than $3\alpha$.
Hence, $\|A\|_{3\alpha, 4m\textrm{-Bes}}<\infty$. 

Set 
\[
B_{s,t} := \int_{s\le  t_1\le  t_2\le  t_3 \le t }
dx_{t_1} \otimes dk_{t_2}\otimes dx_{t_3},
\qquad
D_{s,t} := \int_s^t X^2_{s,u} \otimes dk_{u}.
\]
Then, we can see in the same way as above that
\[
|B_{s,t}|
\le C_7 \| x\|^2_{1/\beta\textrm{-var}, [s,t]}
\| k\|_{q\textrm{-var}, [s,t]}
\le 
C_8 \| X^1\|^2_{\beta\textrm{-Hld}} \| k\|_{\cH^H}
(t-s)^{2\beta +(\frac{1}{q} - \frac{1}{2})}
\]
for every $X$, $k$ and $(s,t) \in \triangle_T$.
Since the $1/\beta$-variation norm of $u\mapsto X^2_{s,u}$ on $[s,t]$
is dominated by $C_9 ( \| X^1\|^2_{1/\beta\textrm{-var}, [s,t]} 
+ \| X^2\|_{1/(2\beta)\textrm{-var}, [s,t]} )$, we can see that
\[
|D_{s,t}|
\le C_{10} (\| X^1\|^2_{\beta\textrm{-Hld}} 
+ \| X^2\|_{2\beta\textrm{-Hld}} )
\| k\|_{\cH^H}
(t-s)^{2\beta +(\frac{1}{q} - \frac{1}{2})}
\]
for every $X$, $k$ and $(s,t) \in \triangle_T$.
The H\"older exponent on the right hand side is larger than $3\alpha$.
Hence, $\|B\|_{3\alpha, 4m\textrm{-Bes}} +\|D\|_{3\alpha, 4m\textrm{-Bes}}  <\infty$.

The eighth term can be written in the coordinate form as
\begin{align}  
E_{s,t}^{\mu\nu\lambda} &:=\int_{s\le  t_1\le  t_2\le  t_3 \le t }
dk_{t_1}^{\mu} \otimes dx_{t_2}^{\nu}\otimes dx_{t_3}^{\lambda}
\nn\\
&=
K_{s,t}^{1, \mu}X_{s,t}^{2, \nu\lambda} 
-
\int_{s\le  t_1\le  t_2\le  t_3 \le t }
dx_{t_1}^{\nu} \otimes dk_{t_2}^{\mu}\otimes dx_{t_3}^{\lambda}
-
\int_{s\le  t_1\le  t_2\le  t_3 \le t }
 dx_{t_1}^{\nu}\otimes dx_{t_2}^{\lambda}
 \otimes dk_{t_3}^{\mu}
\nn
\end{align}
for every $\mu, \nu, \lambda \in  \llbracket 1, d\rrbracket$
and $(s,t) \in \triangle_T$
when $X =S_3 (x)$ for some $x \in \cC^{1\textrm{-Hld}}_0 (\R^d)$.
Note that the shuffle relation was used here.
The two integrals on the right hand side both makes sense 
in the Young sense even for a general $X \in G\Omega_{\alpha, 12m\textrm{-}{\rm Bes}} (\R^d)$ 
and have already been estimated.
From these we can easily see that 
\[
|E_{s,t}|
\le C_{11} (\| X^1\|^2_{\beta\textrm{-Hld}} 
+ \| X^2\|_{2\beta\textrm{-Hld}} )
\| k\|_{\cH^H}
(t-s)^{2\beta +(\frac{1}{q} - \frac{1}{2})}
\]
for every $X$, $k$ and $(s,t) \in \triangle_T$,
which  implies that $\|E\|_{3\alpha, 4m\textrm{-Bes}}  <\infty$.

Combining the above estimates all, we see that 
$\mathcal{T}_k (X) \in \Omega_{\alpha, 12m\textrm{-}{\rm Bes}} (\R^d)$.
It is a routine to check that 
$
\mathcal{T}\colon 
\cH^H \times   G\Omega_{\alpha, 12m\textrm{-}{\rm Bes}} (\R^d)
\to \Omega_{\alpha, 12m\textrm{-}{\rm Bes}} (\R^d)
$
is locally Lipschitz continuous.
If $X =S_3 (x)$ for some $x \in \cC^{1\textrm{-Hld}}_0 (\R^d)$,
then 
\[
\|x+ k\|_{q\textrm{-var}, [s,t]}\le 
C_{11}( \|x\|_{1\textrm{-Hld}} 
+\| k\|_{\cH^H} )(t-s)^{\frac{1}{q} - \frac{1}{2}},\qquad (s,t)\in \triangle_T.
\]
Repeating the same argument just below \eqref{ineq.0523-5},
$\mathcal{T}_k (X)=
S_{\lfloor 1/H\rfloor} (x+ k)\in G\Omega_{\beta\textrm{-}{\rm Hld}} (\R^d)$ 
for every $\beta~(\beta <H)$ sufficiently close to $H$.
Hence, $\mathcal{T}_k (X) \in  G\Omega_{\alpha, 12m\textrm{-}{\rm Bes}} (\R^d)$ in this case.
To prove $\mathcal{T}_k (X) \in  G\Omega_{\alpha, 12m\textrm{-}{\rm Bes}} (\R^d)$ for a general $X \in G\Omega_{\alpha, 12m\textrm{-}{\rm Bes}} (\R^d)$, it is enough to
take a sequence $\{ x(n)\}_{n\in\N} \subset \cC^{1\textrm{-Hld}}_0 (\R^d)$ such that 
$S_3 (x(n))$ converges to $X$ as $n\to\infty$ in $(\alpha, 12m)$-Besov topology and use the continuity of $\mathcal{T}$.
\end{proof}

\begin{remark} \label{rem.0910}
By combining  Corollary \ref{cor.0411-2}, Corollary \ref{cor.0523-1} and Lemma \ref{lem.Ytrans}, we can easily show that
the map $h \mapsto \psi^h =\Psi_a (h)$ is actually locally Lipschitz
continuous from $\cH^H$ to $\tilde\cC^{\alpha\textrm{-Hld}}_a (\R^e)$
for every $\alpha \in (0,H)$.
\end{remark}

\subsection{Fractional diffusion bridge measures}

In this subsection 
we assume the coefficient vector fields $V_j~(j\in \llbracket 0, d\rrbracket)$
of RDE \eqref{rde1} are of $C_{{\rm b}}^\infty$ (when 
they are viewed as $\R^e$-valued functions)
and study the RDE driven by the canonical lift of FBM 
$w =(w_t)_{t\in [0,T]}$ by using 
Malliavin calculus and quasi-sure analysis.
Under this condition on $V_j$'s, it is known that 
$y_t =y (t,a)\in {\bf D}_{\infty} (\R^e)$ for every $t \in [0, T]$ (see Inahama \cite{in2}).
We denote by 
$\mathbb{Q}^H_{a}$ the law of the solution process $y =y (\,\cdot\,,a)$.

Let $t\in (0,T]$ and 
suppose for a while that $y_t =y (t,a)$ is non-degenerate 
in the sense of Malliavin.
We denote its density function by $p(t, a, \cdot)$, namely,
\begin{equation}\label{def.mitsudo}
\mu^H (y(t, a) \in A) =\int_A  p(t, a, b) db
\end{equation}
for every Borel subset $A \subset \R^e$,
where $db$ stands for the Lebesgue measure on $\R^e$.
Obviously, $p(t, a, b)= {\mathbb E} [\delta_b (y (t,a)) ]\ge 0$.

When $p(T, a, b)>0$, we can define a certain 
``bridge measure" of ``conditioned process" by using 
quasi-sure analysis as follows.
Since $\delta_b (y (t,a))$ is a positive Watanabe distribution,
Sugita's theorem implies that there exists a unique 
Borel probability measure $\mu^H_{a,b}$ on $\cW$ corresponding to 
$p(T, a, b)^{-1}\delta_b (y (t,a))$.
Its rough path lift is denoted by $\nu^H_{a,b} := 
(\mathcal{S}_{\alpha\textrm{-}{\rm Hld}})_* \mu^H_{a,b}$
for $\alpha$ with $(\lfloor 1/H\rfloor+1)^{-1} < \alpha <H$.
Similarly, we can also define $\nu^H_{a,b} := 
(\mathcal{S}_{\alpha, 12m\textrm{-}{\rm Bes}})_* \mu^H_{a,b}$
for $(\alpha, m)$ with \eqref{cond.0523-3}.
Note that $\mu^H_{a,b} ((\mathcal{Z}_{\alpha\textrm{-}{\rm Hld}})^c )=0
= \mu^H_{a,b} ((\mathcal{Z}_{\alpha, 12m\textrm{-}{\rm Bes}})^c )$ 
since both $(\mathcal{Z}_{\alpha\textrm{-}{\rm Hld}})^c$ and
$(\mathcal{Z}_{\alpha, 12m\textrm{-}{\rm Bes}})^c$ are slim.
(We omit the H\"older or Besov parameters from the symbol $\nu^H_{a,b}$ since these measure can be naturally viewed as the same measure through the embedding of rough path spaces.
Similar notation is used for $\mathbb{Q}^H_{a}$ and $\mathbb{Q}^H_{a,b}$, the latter of which will be introduced in the next paragraph.)

Next we set  
$\mathbb{Q}^H_{a,b} :=(\Phi_a)_* \nu^H_{a,b} 
=(\Phi_a \circ \mathcal{S}_{\alpha\textrm{-}{\rm Hld}})_*\mu^H_{a,b}$
for $\alpha$ with $(\lfloor 1/H\rfloor+1)^{-1} < \alpha <H$.
This is a probability measure on $\tilde\cC^{\alpha\textrm{-}{\rm Hld}}_a (\R^e)$.
Recall that $y (\,\cdot\,,a) = \Phi_a \circ \mathcal{S}_{\alpha\textrm{-}{\rm Hld}}$,
$\mu^H$-a.s. on $\cW$.
(When $p(T, a, b)=0$, we set $\mathbb{Q}^H_{a,b} :=\delta_{\ell(a,b)}$,
where $\ell(a,b)_t := a + (b-a) (t/T)$. However,
this definition is almost irrelevant, because the law of 
$y(T, a)$ does not charge the subset $\{b \in \R^e\mid p(T, a, b)=0\}$.)

\begin{lemma} \label{lem.0510_1}
Let $H \in (1/4, 1/2]$ and $(\lfloor 1/H\rfloor+1)^{-1} < \alpha <H$.
Assume that 
$y_T =y (T,a)$ is non-degenerate in the sense of Malliavin and $p(T, a, b) >0$.
Then, $\mathbb{Q}^H_{a,b}$ sits on 
$\tilde{\cC}^{\alpha\textrm{-}{\rm Hld}}_{a, b} (\R^e) 
:=\{z\in \tilde{\cC}^{\alpha\textrm{-}{\rm Hld}} (\R^e) 
\mid z_0 =a, z_T=b\}$.
Moreover, 
\begin{align}  \label{eq.0510_1}
\lefteqn{
 {\mathbb E} \left[
g (y_{t_1}, \ldots, y_{t_{N-1}}, y_T) \delta_b (y_T)
\right]
}
\nn\\
&=
p(T, a, b) \int_{\cW}  g (y_{t_1}(w), \ldots, y_{t_{N-1}} (w),  b) 
\mu^H_{a,b} (dw)
\nn\\
&=
p(T, a, b)
\int_{\tilde{\cC}^\alpha_a (\R^e)}  g (z_{t_1}, \ldots, z_{t_{N-1}},  b) 
\mathbb{Q}^H_{a,b} (dz).
\end{align}
\end{lemma}

\begin{proof} 
Since $y_T (w)$ is $\infty$-quasi-continuous in $w$, 
we can see from Lemma \ref{lem.jigoku2} that 
\[
\mathbb{Q}^H_{a,b} (\{z\mid z_T \neq b\})
=\mu^H_{a,b} (\{ w\mid y_T (w)\neq b\})
=0.
\]
Then, \eqref{eq.0510_1} follows immediately from Sugita's theorem 
and the definition of $\mu^H_{a,b}$  and $\mathbb{Q}^H_{a,b}$.
\end{proof}

For a topological space $\cY$, we denote the Borel $\sigma$-field
of $\cY$ by $\mathcal{B} (\cY)$
and the set of real-valued, bounded, Borel measurable functions on $\cY$
by $\cM_{\mathrm{b}} (\cY)$.

\begin{theorem} \label{thm.0510_2}
Let $H \in (1/4, 1/2]$ and $(\lfloor 1/H\rfloor+1)^{-1} < \alpha <H$.
Assume that 
$y_T =y (T,a)$ is non-degenerate in the sense of Malliavin.
Then, the family 
\[
\{ \mathbb{Q}^H_{a,b} (A) \mid A\in 
\mathcal{B} (\tilde{\cC}^{\alpha\textrm{-}{\rm Hld}}_{a} (\R^e)),\,  b\in \R^e\}
\]
is (a version of) the regular conditional probability 
of $\mathbb{Q}^H_{a}$ on $\tilde{\cC}^{\alpha\textrm{-}{\rm Hld}}_{a} (\R^e)$
given the evaluation map at the time $T$.
\end{theorem}

\begin{proof} 
In this proof, we write $\cY: =\tilde{\cC}^\alpha_{a} (\R^e)$
for simplicity
and denote the evaluation map at the time $s$ by $\Pi_s\colon \cY \to \R^e$, i.e. 
$\Pi_s (z) = z_s$ for $z\in\cY$.
Since balls in $\cY$ belong
to $\sigma \{\Pi_s \mid 0\le s\le T \}$ and $\cY$
is Polish, we can easily see that 
$\mathcal{B} (\cY )=\sigma \{\Pi_s \mid 0\le s\le T \}$.
We write $U := \{ b \in \R^e \mid p(T, a, b) >0\}$, which is obviously open.

We say that $G \colon\cY\to \R$
is a $C^\infty_{\mathrm{b}}$-cylinder function if there are 
a partition $\{0=t_0 <t_1 <\cdots < t_N = T\}$ and 
$g \in C^\infty_{\mathrm{b}}((\R^e)^N)$ such that 
$G (z) = g (z_{t_1}, \ldots, z_{t_{N-1}}, z_{t_N})$ for all $z\in \cY$.
We denote by $\mathcal{F} C^\infty_{\mathrm{b}}$ the totality 
of $C^\infty_{\mathrm{b}}$-cylinder functions.
Since $\mathcal{F} C^\infty_{\mathrm{b}}$ is an algebra,
it is a routine to check that 
that $\sigma (\mathcal{F} C^\infty_{\mathrm{b}})$
(i.e. the smallest $\sigma$-field with respect to which 
all $G\in \mathcal{F} C^\infty_{\mathrm{b}}$ is measurable)
coincides with $\mathcal{B} (\cY)$.

First, we will show that $\R^e \ni b \mapsto \mathbb{Q}^H_{a, b} (A)$
is Borel measurable for every $A \in \mathcal{B} (\cY)$.
For that purpose, it suffices to show that 
\[
\cK := \{F \in \cM_{\mathrm{b}} (\cY)
\mid
\mbox{$\R^e \ni b \mapsto \int  F\, d\mathbb{Q}^H_{a,b} \in \R$ is Borel measurable}
\}
\]
coincides with $\cM_{\mathrm{b}} (\cY)$.
Clearly, $\cK$ contains all constant functions.
Moreover, $\cK$ is a monotone class in the sense that, 
if $\{F_n\}_{n\in\N} \subset \cK$ is an non-decreasing 
and uniformly bounded sequence, then the pointwise limit
$\lim_{n\to\infty} F_n \in \cK$.
($\cK$ is also closed under uniform convergence.)
Due to the monotone class theorem for functions, 
it remains to check $\mathcal{F} C^\infty_{\mathrm{b}}\subset \cK$.

Pick any $G \in  \mathcal{F} C^\infty_{\mathrm{b}}$.
We may assume that $G$ is of the above form. 
We see from \eqref{eq.0510_1} that 
\[
\int_{\cY}  G(z) \mathbb{Q}^H_{a,b} (dz)
=
p(T, a, b)^{-1}  {\mathbb E} \left[
g (y_{t_1}, \ldots, y_{t_{N-1}}, y_T) \delta_b (y_T)
\right]
\]
if $b \in U$. Since $b \mapsto \delta_b \in 
\mathscr{S}^{\prime}({\mathbb R}^e)$ is continuous, $b \mapsto \int G \, d\mathbb{Q}^H_{a,b}$ is 
continuous on $U$.
If $b \notin U$, $ \int G \, d\mathbb{Q}^H_{a,b} =G (\ell(a,b))$.
Since $b \mapsto \ell(a,b) \in \cY$ is continuous,
the restriction of 
the map $b \mapsto \int G \, d\mathbb{Q}^H_{a,b}$ to $U^c$ is
also continuous. Thus, we have shown $G \in \cK$.

To finish the proof, we will show that
\[
\mathbb{Q}^H_{a} (A \cap \Pi_T^{-1} (E) )
=
\int_{E}    \mathbb{Q}^H_{a, b} (A)  \, [(\Pi_T)_* \mathbb{Q}^H_{a}] (db),
\quad
A \in \mathcal{B} (\cY), 
E \in \mathcal{B} (\R^e).
\]
Since $[(\Pi_T)_* \mathbb{Q}^H_{a}] (db) = p(T, a, b) db$
equals the law of $y(T, a)$, the above condition reads: 
\begin{equation} \label{eq.0511-2}
\mathbb{E} [\mathbf{1}_A (y ) \,   \mathbf{1}_E (y_T)]
=
\int_{\R^e}  \left( \int_{\cY} \mathbf{1}_A (z)
 \mathbb{Q}^H_{a, b} (dz)  \right)  
 \mathbf{1}_E (b) p(T, a, b) db
\end{equation}
for every 
$A \in \mathcal{B} (\cY)$ and $E \in \mathcal{B} (\R^e)$.
(Since $p(T, a, \,\cdot\,) \equiv 0$ on $U^c$, 
how $\mathbb{Q}^H_{a, b}$ is defined for $b\notin U$ is actually  irrelevant.)

Pick any $\chi \in \mathscr{S} ({\mathbb R}^e)$ and set 
\[
\hat\cK (\chi) := \{F \in \cM_{\mathrm{b}} (\cY)
\mid
\text{Eq. \eqref{eq.0511-3} below holds}
\},
\]
where 
\begin{equation} \label{eq.0511-3}
\mathbb{E} [ F(y ) \,   \chi (y_T)]
=
\int_{\R^e}  \left( \int_{\cY} F (z) 
\mathbb{Q}^H_{a, b} (dz)  \right)   \chi(b) p(T, a, b) db.
\end{equation} 

Pick any $G \in  \mathcal{F} C^\infty_{\mathrm{b}}$.
We can easily see from \eqref{eq.1008-1} and \eqref{eq.0510_1} that 
\begin{align*}  
\mathbb{E} [ G(y ) \,   \chi (y_T)] 
&=
\int_{\R^e} 
\mathbb{E} [ G(y ) \,   \delta_b (y_T)] \, \chi(b)db
\\
&=
\int_{\R^e} 
\left(
p(T, a, b)
\int_{\cY}  G(z)
\mathbb{Q}^H_{a,b} (dz)
\right) \chi(b)db.
\end{align*}
%
%
%
Thus, we have seen
$\mathcal{F} C^\infty_{\mathrm{b}}\subset \hat\cK (\chi)$.
By using the monotone class theorem again, 
we obtain $\hat\cK (\chi)= \mathcal{B} (\cY)$.

Next, for an arbitrary $A \in \mathcal{B} (\cY)$,
we set 
\[
\cN (A) := \{\xi \in \cM_{\mathrm{b}} (\R^e)
\mid
\text{Eq. \eqref{eq.0511-4} below holds}
\}.
\]
Here, 
\begin{equation} \label{eq.0511-4}
\mathbb{E} [ \mathbf{1}_A (y ) \,   \xi (y_T)]
=
\int_{\R^e}  \left( \int_{\cY} \mathbf{1}_A (z) 
\mathbb{Q}^H_{a, b} (dz)  \right)   \xi(b) p(T, a, b) db.
\end{equation} 
We have already seen that $\mathscr{S} ({\mathbb R}^e)\subset \cN (A)$.
By the dominated convergence theorem, the constant function $\mathbf{1}$ on $U$ also belongs to $\cM (A)$.
By using the monotone class theorem again, 
we obtain $\cM_{\mathrm{b}} (\R^e) = \cN (A)$.
Thus, we have shown \eqref{eq.0511-2}.
This completes the proof of Theorem \ref{thm.0510_2}.
\end{proof}

\begin{remark}\label{rem.RCPD}
It should be noted that the conditional distribution
$\mathbb{Q}^H_{a, b}$ is defined 
for {\it every fixed} $b$ with $p (T, a, b) >0$, thanks to Malliavin 
calculus and quasi-sure analysis.
Recall that the regular conditional probability
given a  random variable is merely an {\it equivalence class} with 
respect to the law of the random variable.
Therefore, if one only uses the general theory of regular conditional probability given a measurable map, 
one could not construct $\mathbb{Q}^H_{a, b}$ for a fixed $b$ because 
the singleton $\{b\}$ is a zero set with respect to the law of $y (T, a)$,
which equals $p(T, a, b) db$.
\end{remark}

\begin{remark}\label{rem.positive}
Let the notation be as above and assume that 
$y_T =y (T,a)$ is non-degenerate in the sense of Malliavin.
It is known that $p(T, a, b)>0$ if and only if there exists $h\in \cH^H$
satisfying the following two conditions:
\begin{enumerate} 
\item[(1)]~$\Psi_a (h)_T =b$.
\item[(2)]~$D\Psi_a (h)_T$ is a surjective linear map from 
$\cH^H$ to $\R^e$.
Here, $D\Psi_a (h)_T$ stands for the Fr\'echet derivative of 
$k \in \cH^H \mapsto \Psi_a (k)_T\in \R^e$ at $h\in \cH^H$.
\end{enumerate}
Recall that $\Psi_a (h)\,(= \psi^h)$ is the unique solution of 
the skeleton ODE \eqref{ode2}. (See \cite{ip}  for this fact.)
Note that Condition (2) is equivalent to the invertibility of 
 the deterministic Malliavin 
covariance of $\Psi_a (\,\cdot\,)_T$ at $h$.
\end{remark}
%
%


In the following remark,
we recall a famous sufficient conditions 
for non-degeneracy of the solution 
$(y_t)_{t\in [0,T]} =(y(t, a))_{t\in [0,T]}$ of RDE \eqref{rde1}.  
\begin{remark}\label{rem.nondeg}
Let $V_j~(j\in \llbracket 0, d\rrbracket)$ 
be the coefficient vector fields of RDE \eqref{rde1}.
Here, they are viewed as  first-order differential operators on $\R^e$.
Set $\Sigma_1 =\{V_1, \ldots, V_d\}$ and, recursively,
$\Sigma_k =\{ [Z, V_i] \mid  Z\in \Sigma_{k-1},  0\le i \le d\}$
for $k\ge 2$. 
If $\{ Z(a) \mid Z \in \cup_{k=1}^\infty\Sigma_k \}$ linearly spans $\R^e$
at the starting point $a$, then for all $t \in (0, T]$, $y (t,a)$ is non-degenerate in the sense of Malliavin.
This fact was first shown in \cite{chlt} (see also \cite{got1}).
\end{remark}
%
%
%
%
\subsection{Brownian Case}

In this subsection, we consider the Brownian case,
i.e. the case $H=1/2$, and check that $\mathbb{Q}^{1/2}_{a,b}$
actually coincides with the classical diffusion bridge measure
under mild assumptions.
In this subsection we will write $\mathbb{Q}_{a}:=\mathbb{Q}^{1/2}_{a}$ and $\mathbb{Q}_{a,b}:=\mathbb{Q}^{1/2}_{a,b}$ for simplicity.

In this case, the solution $y=(y_t)$ of \eqref{rde2} coincides with 
the corresponding Stratonovich SDE: 
\begin{equation}\label{sde8}
dy_t = \sum_{j=1}^d V_j (y_t)\circ dw^{j}_t  + V_0 (y_t) dt,
\qquad 
y_0 =a\in \R^e.
\end{equation}
Here, $\circ dw^{j}_t$ stands for the Stratonovich stochastic integral 
with respect to the usual Brownian motion.
Since we let $a$ vary, we will often write $y_t =y (t,a)$.

Throughout this subsection we assume that 
\begin{itemize}
\item
$V_j$ are of $C_{{\rm b}}^\infty$ $(j\in \llbracket 0, d\rrbracket)$.
\item
H\"ormander's bracket-generating condition as in Remark \ref{rem.nondeg} is satisfied at every $a \in \R^e$.
\end{itemize}
Under these two conditions, SDE \eqref{sde8} and the associated 
density function has been extensively studied.
The solution $y (t,a)$ is non-degenerate in the sense of Malliavin for every $a \in \R^e$ and $t\in (0,T]$.
The density of the law of $y (t,a)$ by $p(t, a, b)$
as in \eqref{def.mitsudo}.
As is well-known, $\{ y (\cdot,a) \mid a\in\R^e\}$ 
(or $\{ \mathbb{Q}_{a} \mid a\in\R^e\}$) is a 
time-homogeneous diffusion 
process with generator $V_0 + (1/2)\sum_{j=1}^d V_j^2$
and transition density function $p$.

When $p(T, a, b)>0$,
the associated diffusion bridge measure (i.e. pinned diffusion measure) from $a$ to $b$
in the classical sense is denoted by $\mathbb{U}_{a,b}$, which 
is a unique Borel probability measure on 
$\tilde{\cC}^{\alpha\textrm{-}{\rm Hld}}_{a} (\R^e)$~$(0<\alpha <1/2)$
satisfying Condition \eqref{def.CK} below:

For every rational partition $\{0=t_0 <t_1 <\cdots < t_N =T\}$ 
\footnote{
The partition $\{0=t_0 <t_1 <\cdots < t_N =T\}$  is called rational
if $t_i$ is rational for all $1\le i \le N-1$.
}
and every $g_i \in C_{{\rm K}}^\infty (\R^e)$  ($1\le i \le N$),
it holds that
\begin{align}  \label{def.CK}
\lefteqn{
\int  \prod_{i=1}^{N} g_i (z_{t_i})\mathbb{U}_{a,b} (dz)
}
\nn\\
&=
p(T, a, b)^{-1} \int_{(\R^e)^{N-1}} 
\prod_{i=1}^{N-1} g_i (x_{i}) g_N (b)
\prod_{i=1}^{N} p(t_{i}- t_{i-1}, x_{i-1}, x_{i}) 
\prod_{i=1}^{N-1} dx_i.
\end{align}
Here, we set $x_0 :=a$, $x_N :=b$ and 
the integral on the left hand side is done on $\tilde{\cC}^{\alpha\textrm{-}{\rm Hld}}_{a} (\R^e)$. 
As is well-known, $\mathbb{U}_{a,b}$ is supported on $\tilde{\cC}^{\alpha\textrm{-}{\rm Hld}}_{a, b} (\R^e)$.

When $p(T, a, b)=0$, we set $\mathbb{U}_{a,b} = \mathbb{Q}_{a,b}$
for convenience. 
(See the paragraph just above Lemma \ref{lem.0510_1}.
However, this definition is basically irrelevant.)


\begin{proposition} \label{prop.pdp}
Let the situation be as above.
Then, for every fixed $a\in \R^e$, 
we have $\mathbb{Q}_{a,b} = \mathbb{U}_{a,b}$ for 
almost all $b$ with respect to the law of $y (T, a)$.
\end{proposition}

\begin{proof} 
In view of Theorem \ref{thm.0510_2} and uniqueness 
of regular conditional probability, it suffices to check that 
$
\{ \mathbb{U}_{a,b} (A) \mid A\in 
\mathcal{B} (\tilde{\cC}^{\alpha\textrm{-}{\rm Hld}}_{a} (\R^e)),\,  b\in \R^e\}
$
is also (a version of) the regular conditional probability 
of $\mathbb{Q}_{a}$ on $\tilde{\cC}^{\alpha\textrm{-}{\rm Hld}}_{a} (\R^e)$ given the evaluation map $\Pi_T$ at the time $T$.
Denote by $\mathbb{V}_{a,b}:=\mathbb{Q}_{a} \left( \,\cdot\, \middle|  \Pi_T =b \right)$ the above-mentioned regular conditional probability.
It suffices to show that, for almost all $b$,
$\mathbb{V}_{a,b}$ satisfies Condition \eqref{def.CK} (with $\mathbb{U}_{a,b}$ being replaced by $\mathbb{V}_{a,b}$.)

Take an arbitrary 
rational partition $\{0=t_0 <t_1 <\cdots < t_N =T\}$ and fix it. 
The projection 
$z \mapsto (z_{t_1}, \ldots, z_{t_{N-1}}, z_{t_N})$
from $\cY: =\tilde{\cC}^\alpha_{a} (\R^e)$ to $(\R^e)^N$
is denoted by $\Sigma$.
A generic element of $(\R^e)^N$ is denoted by 
${\bf x}=(x_{1}, \ldots, x_{N-1}, x_{N})$.
One can easily see from the definition of regular conditional probability that 
\[
\Sigma_* \mathbb{V}_{a,b}
=(\Sigma_* \mathbb{Q}_{a}) \left( \,\cdot\, \middle|  x_N =b \right) 
\qquad
\mbox{for almost all $b$.}
\]
Hence, we have only to calculate the right hand side.

As is well-known,
the finite dimensional distribution $\Sigma_* \mathbb{Q}_{a}$
on $(\R^e)^N$ is given by 
\[
\Sigma_* \mathbb{Q}_{a} (d{\bf x}) 
=
\prod_{i=1}^{N} p(t_{i}- t_{i-1}, x_{i-1}, x_{i}) 
\prod_{i=1}^{N} dx_i.
\]
Here and below, we write $x_0:=a$.
From this, one can easily see that 
\begin{align*}
\lefteqn{
(\Sigma_* \mathbb{Q}_{a}) \left( \,\cdot\, \middle|  x_N =b \right) 
}
\\
&=
p(T, a, b)^{-1}
\prod_{i=1}^{N-1} p(t_{i}- t_{i-1}, x_{i-1}, x_{i}) 
p(t_{N}- t_{N-1}, x_{i-1}, b) 
\prod_{i=1}^{N-1} dx_i \delta_b (dx_N)
\end{align*}
for almost all $b$ with $p(T, a, b) >0$.
(Note that $\{b\mid p(T, a, b)=0\}$ is a zero set with respect to
the law of $y(T,a)$, which equals $p(T, a, b)db$, and is 
therefore irrelevant.)
Therefore, 
for such $b$, $(\Sigma_* \mathbb{Q}_{a}) \left( \,\cdot\, \middle|  x_N =b \right)$ satisfies \eqref{def.CK} 
with $\mathbb{U}_{a,b}$ being replaced by this probability measure.

Since  there are countably many rational partitions, 
we have $\mathbb{U}_{a,b} = \mathbb{V}_{a,b}$ for 
almost all $b$. This completes the proof.
\end{proof}

Under the same condition of the above proposition, 
we have the following corollary, which is the main aim of this subsection.
\begin{corollary} \nn
For every $a, b \in\R^e$ with $p (T,a,b) >0$, we have $\mathbb{Q}_{a,b} = \mathbb{U}_{a,b}$.
\end{corollary}

\begin{proof} 
We use Proposition \ref{prop.pdp} for every fixed $a\in\R^e$.
We already know that the set of $b$ satisfying
$\mathbb{Q}_{a,b} = \mathbb{U}_{a,b}$ is a dense subset of the 
open set $\{b\mid p (T,a,b) >0\}$.  For such $b$, we have 
\begin{align}\label{eq.CK2}
\lefteqn{
\int  \prod_{i=1}^{N} g_i (z_{t_i})\mathbb{Q}_{a,b} (dz)
}
\nn\\
&=
p(T, a, b)^{-1} \int_{(\R^e)^{N-1}} 
\prod_{i=1}^{N-1} g_i (x_{i}) g_N (b)
\prod_{i=1}^{N} p(t_{i}- t_{i-1}, x_{i-1}, x_{i}) 
\prod_{i=1}^{N-1} dx_i
\end{align}
with $x_0:=a$ and $x_N :=b$
for every rational partition and $\{g_i\}_{i=1}^N\subset C_{{\rm K}}^\infty (\R^e)$.
It suffices to show that both sides of \eqref{eq.CK2} are 
continuous in $b$ on $\{b\mid p (T,a,b) >0\}$.

The left hand side of \eqref{eq.CK2} equals
\[
p(T, a, b)^{-1}
{\mathbb E} \left[\prod_{i=1}^{N-1} g_i (y_{t_i}) \delta_b (y_T) \right] \, g_N(b).
\]
Since $b \mapsto \delta_b (y_T) \in {\bf D}_{p, -2l}$ is continuous
when $l \in N$ is large enough, the left hand side is continuous in $b$. 
Since the map $(a,b) \mapsto p(t,a,b)$ is continuous for every $t\in (0, T]$
and the supports of $g_i$'s are compact, we can see from 
Lebesgue's bounded convergence theorem that the right hand side of \eqref{eq.CK2} is also continuous in $b$.
\end{proof}

\section{Large deviations for fractional diffusion bridges}
\label{sec.LDP_state}

In this section we prove a small noise large deviation principle 
for scaled fractional diffusion bridges under a suitable
ellipticity condition on the coefficient vector fields.
We continue to assume $H \in (1/4, 1/2]$, 
$(\lfloor 1/H\rfloor+1)^{-1} < \alpha <H$
and that the coefficient vector fields 
$V_j$ are of $C_{{\rm b}}^\infty$ $(j\in \llbracket 0, d\rrbracket)$. 
We basically work on the abstract Wiener space $(\cW, \cH^H, \mu^H)$
and let $\beta \colon [0,1] \to \R$ is a $C^2$-function.
Typical examples of $\beta$ are
$\beta_\ve \equiv 1$ and $\beta_\ve =\ve^{1/H}$.
We set $\lambda_t :=t$ and denote by $ \R\langle \lambda \rangle$ 
be a one-dimensional linear subspace of $\cC^{1\textrm{-Hld}}_0 (\R)$
spanned by $\lambda$.

For a small parameter $\ve\in (0,1]$, 
we consider the following scaled RDE:
\begin{equation}\label{rde5}
dy^{\ve}_t = \ve \sum_{j=1}^d V_j (y^{\ve}_t) dw^{j}_t  + 
\beta_\ve V_0 (y^{\ve}_t) dt,
\qquad 
y^{\ve}_0 =a\in \R^e.
\end{equation}
This can be viewed as  a ``scaled version" of RDE \eqref{rde2}.
We sometime write $y^{\ve}_t =y^{\ve} (t, a)$.
The corresponding skeleton ODE is given as follows:
For $h\in \cH^H$, 
\begin{equation}\label{ode5}
d\hat\psi^h_t = \sum_{j=1}^d V_j (\hat\psi^h_t) dh_t  +\beta_0 V_0 (\hat\psi^h_t) dt,
\qquad 
\hat\psi^h_0 =a\in \R^e.
\end{equation}
This is very similar to ODE \eqref{ode2}
(the drift vector field $V_0$ is replaced by $\beta_0 V_0$).
We will write $\hat\psi^h =: \hat\Psi_a (h)$.

Now we introduce the following notation.
We assume $\alpha\in (1/4, 1/2]$ for a while.
First let us recall the Young pairing $\cP\colon 
G\Omega_{\alpha\textrm{-Hld}} (\R^d) \times 
\cC^{1\textrm{-Hld}}_0 (\R) \to G\Omega_{\alpha\textrm{-Hld}} (\R^{d+1})$.
It is defined to be a unique continuous map such that
\[
\mathcal{P} ( S_{\lfloor 1/\alpha\rfloor} (x), k) = S_{\lfloor 1/\alpha\rfloor} 
((x,k)), \qquad
x\in \cC^{1\textrm{-Hld}}_0 (\R^d), k\in  \cC^{1\textrm{-Hld}}_0 (\R).
\]
We denote by 
$\hat\Phi_a \colon G\Omega_{\alpha\textrm{-}{\rm Hld}} (\R^{d+1})\to 
\tilde\cC^{\alpha\textrm{-Hld}}_a (\R^e)$ 
the Lyons-It\^o map associated with the coefficients $\{V_j\}_{ 0\le j \le d}$.
(Unlike the preceding sections, $\hat\Phi_a$ is 
defined on the geometric rough path space over $\R^{d+1}$, not 
over $\R^{d}$.
Note that $\hat\Phi_a ( \mathcal{P} (X,\lambda))= \Phi_a (X)$.)

Assume $(\lfloor 1/H\rfloor+1)^{-1} <\alpha <H$ again,
Using these symbols, the solution of \eqref{rde5} can be written as 
$y^{\ve}=\hat\Phi_a ( \mathcal{P} (\ve W, \beta_\ve \lambda))$
with $W= \mathcal{S}_{\alpha\textrm{-}{\rm Hld}} (w)$.
Note also that $ \hat\Psi_a (h) = 
\hat\Phi_a ( \mathcal{P} 
(\mathcal{S}_{\alpha\textrm{-}{\rm Hld}}(h), \beta_0 \lambda))$
for every $h\in \cH^H$.

We say that the coefficient vector fields $\{V_j\}_{ 0\le j \le d}$
of RDE \eqref{rde5} and of ODE \eqref{rde5} 
satisfies the ellipticity condition 
at $z \in\R^e$ if $\{ V_1 (z), \ldots, V_d (z)\}$ linearly spans $\R^e$.
(Note that $V_0$ is actually irrelevant.)
We say that $\{V_j\}_{ 0\le j \le d}$ satisfies 
the ellipticity condition on an open subset $U\subset \R^e$
if it satisfies the ellipticity condition at every $z \in U$.

We sometimes write $\mathbf{V} (z) := [V_1 (z)| \cdots |V_d (z) ]$,
which is an $e\times d$-matrix valued function on $\R^e$.
(In this definition, $V_j$'s are viewed as $\R^e$-valued functions
in an obvious way.)
It is well-known that
$\mathbf{V} (z)\mathbf{V} (z)^{\top}$ is invertible if and only if 
$\{V_j\}_{ 0\le j \le d}$ satisfies the ellipticity condition at $z$,
where $A^{\top}$ stands for the transposed matrix of a matrix $A$.

From now on, we assume that $\{V_j\}_{ 0\le j \le d}$ satisfies 
the ellipticity condition at the starting point $a$ 
of RDE \eqref{rde5} and of ODE \eqref{ode5}.
Then, $y^{\ve}_t$ is known to be non-degenerate in the sense of 
Malliavin for every $0< t\le T$ (see Remark \ref{rem.nondeg}).
Likewise, thanks to the ellipticity assumption,
$\hat\Psi_a (\,\cdot\,)_t$ is non-degenerate 
at every $h \in \cH^H$ for every $0< t\le T$.
Recall that $\hat\Psi_a (\,\cdot\,)_t\colon \cH^H \to \R^e$ is called  
non-degenerate at $h$ if its Fr\'echet derivative
$D \hat\Psi_a (h)_t\colon \cH^H \to \R^e$ at $h$ is a surjective linear map.
This is equivalent to the invertibility of 
 the deterministic Malliavin 
covariance of $\hat\Psi_a (\,\cdot\,)_t$ at $h$.

Let $U\subset \R^e$ be a connected open subset
on which $\{V_j\}_{ 0\le j \le d}$ satisfies the ellipticity condition
and assume $a, b \in U$.
Then, there exists a $C^\infty$-path $\varphi \colon [0,T] \to U$
such that $\varphi_0 =a$ and $\varphi_T =b$.
Set $h_0 =0$ and 
\[
h^{\prime}_t := \mathbf{V} (\varphi_t)^{\top} 
\{ \mathbf{V} (\varphi_t)\mathbf{V} (\varphi_t)^{\top}\}^{-1}
[ \varphi^{\prime}_t - \beta_0 V_0 (\varphi_t) ],
\qquad t \in [0,T].
\]
Then, $h  \colon [0,T] \to \R^d$ is also $C^\infty$ and $\hat\Psi_a (h) =\varphi$. 
Indeed, 
$
 \mathbf{V} (\varphi_t) h^{\prime}_t =\varphi^{\prime}_t - \beta_0V_0 (\varphi_t)$.
Note that $C^\infty$-paths starting at $0$ belong to $\cH^H$.
(See \cite[Lemma 6.2]{bg15} for example.
This fact still holds even when $1/2 <H< 1$.)
Hence, there exists $h \in \cH^H$ such that
$\hat\Psi_a (h)$ connects $a$ and $b$ and $\hat\Psi_a (\,\cdot\,)_T$ is non-degenerate at $h$.
Likewise, setting $h(\ve)^{\prime}_0 =0$ and
\[
h(\ve)^{\prime}_t := \ve^{-1}\mathbf{V} (\varphi_t)^{\top} 
\{ \mathbf{V} (\varphi_t)\mathbf{V} (\varphi_t)^{\top}\}^{-1}
[ \varphi^{\prime}_t - \beta_{\ve} V_0 (\varphi_t) ],
\qquad t \in [0,T],
\]
for every fixed $\ve >0$,  we have 
\[
d\varphi_t= \ve \sum_{j=1}^d V_j (\varphi_t) dh(\ve)^{\prime}_t 
  + 
\beta_\ve V_0 (\varphi_t) dt,
\qquad 
\varphi_0 =a\in \R^e.
\]
This is ``$\ve$-dependent" skeleton ODE corresponds to \eqref{rde5}.
The solution map (at time $T$) of the above ODE
is also non-degenerate at every point of $\cH^H$ 
since the coefficients still satisfy 
the ellipticity condition at the starting point $a$.

In this case, as we have seen in Remark \ref{rem.positive},
$\delta_b (y^\ve (t,a))$ is a positive Watanabe distribution
with $p^\ve (T, a, b) := {\mathbb E} [\delta_b (y^\ve (t,a)) ]>0$.
We denote by $\mu_{a,b}^{H, \ve}$ the
Borel probability measure on $\cW$ corresponding to 
$p^\ve (T, a, b)^{-1}\delta_b (y^\ve (t,a))$. 
We denote its rough path lift by $\nu_{a,b}^{H, \ve} := 
(\ve\mathcal{S}_{\alpha\textrm{-}{\rm Hld}})_* \mu_{a,b}^{H, \ve}$.
(It is worth recalling that $\mathcal{Z}_{\alpha\textrm{-}{\rm Hld}}$ 
is invariant under the dilation by $\ve$
and $\ve\mathcal{S}_{\alpha\textrm{-}{\rm Hld}} (w)=
\mathcal{S}_{\alpha\textrm{-}{\rm Hld}}(\ve w)$ for all 
$w\in \mathcal{Z}_{\alpha\textrm{-}{\rm Hld}}$.
Of course, an analogous fact holds in the Besov case, too.)

Set $\tilde{I}\colon  G\Omega_{\alpha\textrm{-Hld}} (\R^d)\to [0, \infty]$ by 
\[
\tilde{I} (X) = \begin{cases}
\tfrac12 \|h\|^2_{\cH^H} 
& (\mbox{if $X= S_{\lfloor 1/H\rfloor} (h)$ for some 
$h\in \cH^H$ with $\hat\Psi_a (h)_T =b$}),\\
\infty & (\mbox{otherwise}).
\end{cases}
\]
Note that, for every $r \in [0, \infty)$, we have
\begin{align*}  
\{X\in G\Omega_{\alpha\textrm{-Hld}} (\R^d)\mid  \tilde{I}(X) \le r \} 
&= \{ S_{\lfloor 1/H\rfloor} (h) \mid  
\mbox{$h \in \cH^H$ with $\tfrac12 \|h\|^2_{\cH^H} \le r$} \} 
\nn\\
&\qquad
\cap 
\{ X \in G\Omega_{\alpha\textrm{-Hld}} (\R^d) \mid 
 \hat\Phi_a ( \mathcal{P} 
(X, \beta_0 \lambda))_T =b\}.
\end{align*}
We can easily see from 
and the Schilder type LDP for fractional Brownian rough path 
and Lyons' continuity theorem
that the right hand side is compact.
Hence, $\tilde{I}$ is good.
We also introduce another good rate function 
$I\colon  G\Omega_{\alpha\textrm{-Hld}} (\R^d)\to [0, \infty]$ by 
\begin{align*}  
I (X) 
&= \tilde{I}(X) - \min\{ \tilde{I}(X) \mid X \in G\Omega_{\alpha\textrm{-Hld}} (\R^d)\}
\nn\\
&= 
\tilde{I}(X) -\min\{ \tfrac12 \|h\|^2_{\cH^H} \mid 
\mbox{$h\in \cH^H$ with $\hat\Psi_a (h)_T=b$} \}.
\end{align*}  
The minimum on the right hand side is actually attained
since $\tilde{I}$  good. Hence, the minimum of $I$ is $0$.

Of course, essentially the same objects can be defined 
in the Besov case, too.
(They are denoted by the same symbols.)
Let $(\alpha, m)$ satisfy \eqref{cond.0523-3}.
Then, the lifted measure
$\nu_{a,b}^{H, \ve} := 
(\ve\mathcal{S}_{\alpha, 12m\textrm{-}{\rm Bes}})_* \mu_{a,b}^{H, \ve}$
is introduced on the Besov rough path space
$G\Omega_{\alpha, 12m\textrm{-}{\rm Bes}} (\R^d)$, too.
The good rate functions $\tilde{I}$ and $I$ can also be 
defined on $G\Omega_{\alpha, 12m\textrm{-}{\rm Bes}} (\R^d)$
in the same way as in the H\"older case.


Now we provide our main large deviation results.
The speed of these large deviations is $1/\ve^2$.
The key is proving large deviations for the family of unnormalized measures
$\{p^\ve (T,a,b)\nu_{a,b}^{H,\ve}\}_{\ve \in (0,1]}$
on the Besov-type geometric rough path space, which is quite hard
and will be postponed to the next section.

\begin{theorem} \label{thm.LDP8/6}
Let $H \in (1/4, 1/2]$.
We assume that $C_{{\rm b}}^\infty$-vector fields
$V_j~(j\in \llbracket 0, d\rrbracket)$ on $\R^e$ satisfy the following condition: 
There exists a connected open subset $U$ of $\R^e$ 
such that $a, b \in U$ and $\{V_j\}_{ 0\le j \le d}$
satisfies the ellipticity condition on $U$.
Then, the following three statements hold:
\\
{\rm (1)}~The following Varadhan type asymptotics holds:
\[
\lim_{\ve\searrow 0}\ve^2 \log p^\ve (T, a, b)=-\min\{ \tfrac12 \|h\|^2_{\cH^H} \mid 
\mbox{$h\in \cH^H$ with $\hat\Psi_a (h)_T =b$} \}.
\]
{\rm (2)}~For every $\alpha$ with $(\lfloor 1/H\rfloor+1)^{-1}<\alpha< H$,
$\{\nu_{a,b}^{H, \ve}\}_{\ve \in (0,1]}$ satisfies a 
large deviation principle 
on $G\Omega_{\alpha\textrm{-}\rm{Hld}} (\R^d)$ as $\ve \searrow 0$
with good rate function $I$. 
Moreover, $\{\nu_{a,b}^{H,\ve}\}_{\ve \in (0,1]}$  is exponentially tight 
as $\ve \searrow 0$.
\\
{\rm (3)}~For every $(\alpha, m)$ satisfying  \eqref{cond.0523-3}
and \eqref{cond.0523-4}, 
$\{\nu_{a,b}^{H,\ve}\}_{\ve \in (0,1]}$ satisfies a large deviation principle 
on $G\Omega_{\alpha, 12m\textrm{-}{\rm Bes}} (\R^d)$ 
as $\ve \searrow 0$ with good rate function $I$. 
Moreover, $\{\nu_{a,b}^{H,\ve}\}_{\ve \in (0,1]}$  is exponentially tight 
as $\ve \searrow 0$.
\end{theorem}

\begin{proof} 
By the Besov-H\"odler embedding \eqref{ineq.0311-4},
Assertion (2) immediately follows from Assertion (3).
So, we only explain (1) and (3).

The large deviation principle for the {\it unnormalized} measures 
$\{\tilde\nu_{a,b}^{H,\ve} \}_{\ve \in (0,1]}$ will be shown 
in \eqref{low_0808-1} and \eqref{up_0826-1} below, 
where $\tilde\nu_{a,b}^{H,\ve} := p^\ve (T,a,b)\nu_{a,b}^{H,\ve}$.
The exponential tightness will be shown in \eqref{low_0819-1}.
Moreover,
since the whole set $G\Omega_{\alpha, 12m\textrm{-}{\rm Bes}} (\R^d)$ is both open and closed, \eqref{low_0808-1} and \eqref{up_0826-1}
also imply that
\begin{align*}
\lim_{\ve\searrow 0}\ve^2 \log p^\ve (T, a, b)
&=
\lim_{\ve\searrow 0}\ve^2 \log \tilde\nu_{a,b}^\ve (G\Omega_{\alpha, 12m\textrm{-}{\rm Bes}} (\R^d)) 
\nn\\
&=-\min\{ \tfrac12 \|h\|^2_{\cH^H} \mid 
\mbox{$h\in \cH^H$ with $\hat\Psi_a (h)_T =b$} \}.
\end{align*}
Thus, we have shown Assertion (1) and (3).
\end{proof}

We denote by $\mathbb{Q}_{a,b}^{H, \ve}$ the law of $y^{\ve}$ 
conditioned that $y^{\ve}_T =b$, that is, 
the law of $y^\ve = \hat\Phi_a (\cP (\ve W, \beta_\ve \lambda))$
under $\mu_{a,b}^{H, \ve}$.
Since $y^\ve = \hat\Phi_a (\cP (\ve W, \beta_\ve \lambda))$
and the law of $\ve W$ under $\mu_{a,b}^{H, \ve}$ equals $\nu_{a,b}^{H,\ve}$,
we have 
$\mathbb{Q}_{a,b}^{H, \ve} 
= [\hat\Phi_a \circ\cP]_* 
(\nu_{a,b}^{H, \ve} \otimes \delta_{\beta_\ve \lambda})$.

With Theorem \ref{thm.LDP8/6} at hand,
we can easily obtain a large deviation principle for 
the laws $\{ \mathbb{Q}_{a,b}^{H,\ve}\}_{\ve (0,1]}$ 
of (scaled) fractional diffusion bridges as $\ve\searrow 0$.
This is one of our main results in this paper.

\begin{corollary} \label{cor.ldp_fdb}
Let $H \in (1/4, 1/2]$.
We assume that $C_{{\rm b}}^\infty$-vector fields
$V_j~(j\in \llbracket 0, d\rrbracket)$ on $\R^e$ satisfy the following condition: 
There exists a connected open subset $U$ of $\R^e$ 
such that $a, b \in U$ and $\{V_j\}_{ 0\le j \le d}$
satisfies the ellipticity condition on $U$.

Then, for every $\alpha$ with $(\lfloor 1/H\rfloor+1)^{-1}<\alpha< H$, 
$\{\mathbb{Q}_{a,b}^{H, \ve}\}_{\ve \in (0,1]}$  
satisfies a large deviation principle 
on $\tilde{\cC}^{\alpha\textrm{-}{\rm Hld}}_{a, b} (\R^e)$ as $\ve \searrow 0$
with good rate function $J$, where we set
\begin{align*}
J (\xi) &:= \inf\{  \tfrac12 \|h\|^2_{\cH^H}  \mid 
\mbox{$h\in \cH^H$ with  $\xi =\hat\Psi_a (h)$}\}, 
\\
&\qquad -\min\{ \tfrac12 \|h\|^2_{\cH^H} \mid 
\mbox{$h\in \cH^H$ with $\hat\Psi_a (h)_T =b$} \},
\qquad \xi \in \tilde{\cC}^{\alpha\textrm{-}{\rm Hld}}_{a, b} (\R^e).
 \end{align*}
As usual, $\inf \emptyset =\infty$ by convention.
\end{corollary}

\begin{proof} 
It is almost obvious that,  as $\ve \searrow 0$,
$\{\delta_{\beta_\ve \lambda}\}_{\ve \in (0,1]}$ is exponentially tight and 
satisfies a large deviation principle 
on $\R\langle \lambda \rangle$ as $\ve \searrow 0$
with good rate function $L$, where $L (s \lambda) :=0$ if $s=\beta_0$ 
and $L (s \lambda) :=\infty$ if $s\neq\beta_0$.

By this fact and Theorem \ref{thm.LDP8/6} (2), 
$\{\nu_{a,b}^{H,\ve} \otimes\delta_{\beta_\ve \lambda}\}_{\ve \in (0,1]}$  
satisfies a large deviation principle 
on $G\Omega_{\alpha\textrm{-Hld}} (\R^{d})
\times \R\langle \lambda \rangle$ as $\ve \searrow 0$ 
with good rate function $K$ defined by 
$K (X, s\lambda):= \hat{I}(X)+L (s \lambda)$.

Since $\mathbb{Q}_{a,b}^{H, \ve}$ is the law 
of $\nu_{a,b}^{H, \ve} \otimes\delta_{\beta_\ve \lambda}$ under the
continuous map (i.e. the composition of the Young pairing and $\hat\Phi_a$), 
$\{\mathbb{Q}_{a,b}^{H, \ve}\}_{\ve \in (0,1]}$  
satisfies a large deviation principle 
on $\tilde{\cC}^{\alpha\textrm{-}{\rm Hld}}_{a, b} (\R^e)$ as $\ve \searrow 0$
with good rate function.
(This is due to the contraction principle. See \cite[Theorem 4.2.1]{dzbook}
for example.)
The explicit form of the rate function follows of 
the contraction principle and the fact that 
$\hat\Psi_a (h)=
\hat\Phi_a (\cP ( S_{\lfloor 1/H\rfloor} (h), \beta_0 \lambda))$ for $h \in \cH^H$.
\end{proof}

\begin{remark} \label{rem.relax}
We assumed for simplicity that $\beta \colon [0,1] \to \R$ is of $C^2$,
but this $C^2$-condition can be relaxed to the 
${\rm Lip}^{\gamma}$-function for some $\gamma \in (1,2]$.
(This condition means that $\beta$ is of $C^1$ and 
$\beta^{\prime}$ is $(\gamma -1)$-H\"older continuous.)
By this relaxation, the following two minor modifications must be made
in the proof of lower estimate:
\begin{itemize} 
\item
``$O (\ve^2)$" in Eq. \eqref{est.0815-1} should be relpaced 
by ``$O (\ve^\gamma)$."  
\item
``$O (\ve)$" in Eq. \eqref{est.0926-1} should be relpaced 
by ``$O (\ve^{\gamma-1})$." 
\end{itemize}
On the other hand, the proofs of exponential tightness
and of upper estimate need not be modified.
\end{remark}

\begin{remark} 
There are several preceding works on large deviations of
Freidlin-Wentzell type for bridge measures.
First, the case of Brownian bridges on 
compact Riemannian manifolds was proved in \cite{hsu}.
The case of elliptic and hypoelliptic
diffusion bridges on Euclidean spaces was proved 
in \cite{in1, in3}, respectively.
The case of hypoelliptic diffusion bridges on manifolds 
was studied in \cite{bai, in4}. 
In a recent preprint \cite{hw}, the case of (a certain class of) jump processes was shown.
The method of combining quasi-sure analysis and rough path theory
was developed in the author's previous works \cite{in1, in3, in4}.
\end{remark}

\section{Proof of Theorem \ref{thm.LDP8/6} (1) \& (3)}
\label{sec.proof.low}

In this section we prove Theorem \ref{thm.LDP8/6} (1) and (3).
We set $\tilde\mu_{a,b}^{H, \ve}= p^\ve (T, a, b)\mu_{a,b}^{H, \ve}$, which corresponds to the positive Watanabe distribution
 $\delta_b (y^\ve (t,a))$ via Sugita's theorem.
Similarly, we set 
$\tilde\nu_{a,b}^{H, \ve}= p^\ve (T, a, b)\nu_{a,b}^{H, \ve}$.
We mainly calculate these unnormalized finite measures.

Throughout this section, we will work 
under the assumption of Theorem \ref{thm.LDP8/6} (3).
We will often write $W:= \mathcal{S}_{\alpha, 12m\textrm{-}{\rm Bes}}(w)$ for simplicity.
Recall that $y^\ve = \hat\Phi_a (\cP(\ve W, \beta_\ve \lambda))$.
Precisely, we understand the right hand side in the following way
(we write $\gamma :=\alpha - (12m)^{-1}$ for simplicity):
\begin{itemize} 
\item
Though 
$\ve W$ is $G\Omega_{\alpha, 12m\textrm{-}{\rm Bes}} (\R^d)$-valued,
it can also be viewed as 
$G\Omega_{\gamma\textrm{-Hld}} (\R^d)$-valued, 
due to the embedding in \eqref{ineq.0311-4}.
(This embedding will not be written explicitly.)
\item
$\cP (\ve W, \beta_\ve \lambda)\in G\Omega_{\gamma\textrm{-Hld}} (\R^{d+1})$ stands for the Young pairing of 
$\ve W\in G\Omega_{\gamma\textrm{-Hld}} (\R^d)$ and 
$\beta_\ve \lambda \in \R \langle \lambda \rangle$.
\item
$\hat\Phi_a$ is viewed as a continuous map from 
$G\Omega_{\gamma\textrm{-Hld}} (\R^{d+1})$ to $\tilde\cC^\gamma_a (\R^e)$. 
\item
Therefore, $y^\ve$ is a $\tilde
\cC^\gamma_a (\R^e)$-valued random variable
and a continuous image of the (direct product) pair of 
$\ve W\in G\Omega_{\alpha, 12m\textrm{-Bes}} (\R^d)$ and 
$\beta_\ve \lambda \in \R \langle \lambda \rangle$.
\end{itemize}

\subsection{Large deviation lower bound}
\label{subsec.low}

In this subsection we will show that 
\begin{equation} \label{low_0808-1}
\varliminf_{\ve\searrow 0}  \ve^2 \log \tilde\nu_{a,b}^{H,\ve} (O)
 \ge -\inf_{X \in O} \tilde{I} (X)
\end{equation}
holds for every open subset 
$O$ of $G\Omega_{\alpha, 12m\textrm{-}{\rm Bes}} (\R^d)$.

Denote by $\iota^* \colon \cW^* \hookrightarrow (\cH^H)^* \cong \cH^H$
be the adjoint of the inclusion $\iota \colon\cH^H \hookrightarrow\cW$.
By the Riesz identification $(\cH^H)^* \cong \cH^H$, 
we view $\iota^*  (\cW^*)$ as a dense subspace of $\cH^H$.
More concretely, 
$\iota^*  (\cW^*)=\{h \in \cH^H\mid 
\mbox{$\la h, \,\cdot\,\ra_{\cH}$ extends to an element of $\cW^*$} \}$.

Take any $h \in \cH^H$ with $\hat\Psi_a (h)_T=b$. 
Since $D\hat\Psi_a (h)_T$ is non-degenerate, 
i.e. a surjective linear map from $\cH^H$ to $\R^e$,
there exists a sequence $\{h_n \}_{n\in\N} \subset \iota^*  (\cW^*)$ such that 
$\lim_{n\to\infty} \|h-h_n\|_{\cH} =0$ and $\hat\Psi_a (h_n)_T=b$ for 
all $n\in\N$. 
(See \cite[Lemma 7.3]{in1} for this fact.)
By (a special case of) Lemma \ref{lem.Ytrans},
$\lim_{n\to\infty}S_{\lfloor 1/H\rfloor} (h_n)=S_{\lfloor 1/H\rfloor} (h)$
in the $(\alpha, 12m)$-Besov topology. 
Hence, we have
\begin{align*}  
\inf_{X \in O} I(X) 
&=
\{\tfrac12 \|h\|^2_{\cH^H}  \mid 
\mbox{
$h\in \cH^H$ with $\hat\Psi_a (h)_T =b$ and $S_{\lfloor 1/H\rfloor} (h)\in O$}
\}
\nn\\
&=
\{\tfrac12 \|k\|^2_{\cH^H}  \mid 
\mbox{
$k\in \iota^*  (\cW^*)$ with $\hat\Psi_a (k)_T =b$ and $S_{\lfloor 1/H\rfloor} (k)\in O$}
\}.
\end{align*}

For simplicity,
we write $K :=S_{\lfloor 1/H\rfloor} (k)$ and set 
\[
B_r (X) := \bigcap_{1\le i\le \lfloor 1/H\rfloor}
\left\{\hat{X}\in G\Omega_{\alpha, 12m\textrm{-}{\rm Bes}} (\R^d)
\middle| 
\|\hat{X}^i-X^i\|_{i \alpha, 12m/i\textrm{-Bes}}^{12m/i} <r
\right\}
\]
for $r \in (0,\infty)$ and $X\in G\Omega_{\alpha, 12m\textrm{-}{\rm Bes}} (\R^d)$.
Note that $\{ B_r (X)\}_{r \in (0,\infty)}$ forms a 
fundamental system of neighborhood around $X$.
Similarly, thanks to Lemma \ref{lem.Ytrans}, 
$\{ \mathcal{T}_k B_r (\mathbf{0})\}_{r \in (0,\infty)}$ forms a 
fundamental system of neighborhood around $K$.
(Here, $\mathbf{0}$ stands for the zero rough path
and $\mathcal{T}_k$ for the Young translation by $k$.)

Therefore,
in order to obtain \eqref{low_0808-1}, it suffices to show that 
\begin{equation} \label{low_0808-2}
\lim_{r \searrow 0}
\varliminf_{\ve\searrow 0}  \ve^2 \log \tilde\nu_{a,b}^{H,\ve} (\mathcal{T}_k B_r (\mathbf{0}))
 \ge -\frac12 \|k\|^2_{\cH^H} 
\end{equation}
for every $k\in \iota^*  (\cW^*)$ with $\hat\Psi_a (k)_T =b$.
We will prove this lower estimate in the rest of this subsection.

Take $\alpha' \in  (\alpha, H)$ sufficiently close to $H$.
Since the $\alpha'$-H\"older topology is stronger than 
the $(\alpha, 12m)$-Besov norm, we have
$\mathcal{Z}_{\alpha'\textrm{-}{\rm Hld}}\subset \mathcal{Z}_{\alpha, 12m\textrm{-}{\rm Bes}}$. 
By Corollary \ref{cor.0321} and Cameron-Martin theorem, 
$(\mathcal{Z}_{\alpha'\textrm{-}{\rm Hld}})^c$ and 
$( \mathcal{Z}_{\alpha'\textrm{-}{\rm Hld}} -k)^c= 
(\mathcal{Z}_{\alpha'\textrm{-}{\rm Hld}})^c -k$ are both slim
for every $k \in \cH^H$.
Hence, 
$\mathcal{S}_{\alpha, 12m\textrm{-}{\rm Bes}}
= \mathcal{S}_{\alpha'\textrm{-}{\rm Hld}}$, 
\footnote{
Precisely speaking, the right hand side should be the composition 
of $\mathcal{S}_{\alpha'\textrm{-}{\rm Hld}}$ and the inclusion map
$G\Omega_{\alpha'\textrm{-Hld}} (\R^d) \hookrightarrow 
G\Omega_{\alpha, 12m\textrm{-}{\rm Bes}} (\R^d)$.
}
quasi-surely.
Moreover, if $k \in \cH^H$,
$\mathcal{S}_{\alpha'\textrm{-}{\rm Hld}} (\,\cdot\, +k)
= \mathcal{T}_k \mathcal{S}_{\alpha'\textrm{-}{\rm Hld}}$, quasi-surely.
(Consequently, 
$\mathcal{S}_{\alpha, 12m\textrm{-}{\rm Bes}}(\,\cdot\, +k)
= \mathcal{T}_k \mathcal{S}_{\alpha, 12m\textrm{-}{\rm Bes}} $, quasi-surely. Actually, this will be used below.)
Indeed, if $w\in (\mathcal{Z}_{\alpha'\textrm{-}{\rm Hld}}) 
\cap (\mathcal{Z}_{\alpha'\textrm{-}{\rm Hld}} -k)$, then 
\begin{align*}  
\mathcal{S}_{\alpha'\textrm{-}{\rm Hld}} (w +k)_{s,t}
&=
\lim_{l\to\infty}S_{\lfloor 1/H\rfloor} ((w+k)(l))_{s,t}
\\
&=
\lim_{l\to\infty}\mathcal{T}_{k(l)} S_{\lfloor 1/H\rfloor} (w(l))_{s,t}
\\
&=
\mathcal{T}_{k} \mathcal{S}_{\alpha'\textrm{-}{\rm Hld}} (w)_{s,t}, 
\qquad\qquad (s,t) \in \triangle_T.
\end{align*}
The first equality is due to
$w\in \mathcal{Z}_{\alpha'\textrm{-}{\rm Hld}} -k$.
The last equality is due to the Young translation
in the variation topologies because 
(i)~$w\in \mathcal{Z}_{\alpha'\textrm{-}{\rm Hld}}$, 
(ii)~$\lim_{l\to\infty} \|k(l) -k\|_{q\textrm{-var}}=0$ for 
every $q > ( H + 1/2)^{-1}$,
(iii)~(we may assme that) $\alpha' +1/q >1$.

Let $\chi \colon \R \to [0,1]$ be an even smooth function 
with compact support such that 
$\chi (x) =1$ if $|x| \le 1/2$ and $\chi (x) =0$ if $|x| \ge 1$.
We write $\chi_r (x) =\chi (x/r)$ for $r \in (0, \infty)$.
Clearly, 
\[
\mathbf{1}_{B_r (\mathbf{0})} (X)
\ge 
\prod_{1\le i \le \lfloor 1/H\rfloor}
\chi_r (\|X^i \|_{i \alpha, 12m/i\textrm{-Bes}}^{12m/i} ), 
\qquad X \in  G\Omega_{\alpha, 12m\textrm{-}{\rm Bes}} (\R^d).
\]
For simplicity, we write $y^{\ve}(t,a) =y^{\ve}_t$.
We also write $W:= \mathcal{S}_{\alpha, 12m\textrm{-}{\rm Bes}}(w)$,
which is $\infty$-quasi-continuous in $w$ as we have seen in Theorem \ref{thm.0321}.

By Sugita's theorem and the Cameron-Martin theorem
for the translation by $k/\ve$, we have
\begin{align}  \label{est.0809-1}
\lefteqn{
\tilde\nu_{a,b}^{H, \ve} (\mathcal{T}_k B_r (\mathbf{0}))
}
\nn\\
&=
\tilde\mu_{a,b}^{H, \ve} \left(\{ w \in \cW \mid
\ve W 
\in \mathcal{T}_k B_r (\mathbf{0})\}
\right)
\nn\\
&=
\int_{\cW} \mathbf{1}_{B_r (\mathbf{0})} (  \mathcal{T}_{-k} (\ve W)) 
\, \tilde\mu_{a,b}^{H, \ve} (dw)
\nn\\
&\ge
\int_{\cW}
\prod_{i}
\chi_r (\| \mathcal{T}_{-k} (\ve W)^i \|_{i \alpha, 12m/i\textrm{-Bes}}^{12m/i} )
\, \tilde\mu_{a,b}^{H, \ve} (dw)
\nn\\
&=
\mathbb{E}\left[
\prod_{i}
\chi_r (\| \mathcal{T}_{-k} (\ve W)^i \|_{i \alpha, 12m/i\textrm{-Bes}}^{12m/i} )
\, \delta_b (y^{\ve}_T)
\right]
\nn\\
&=
\mathbb{E}\left[
\exp \left(
-\frac{\la k, \,\cdot\,\ra}{\ve} - \frac{\|k\|_{\cH^H}^2}{2\ve^2}
\right)
\prod_{i}
\chi_r (\|(\ve W)^i \|_{i \alpha, 12m/i\textrm{-Bes}}^{12m/i} )
\, \delta_b (\tilde{y}^{\ve}_T)
\right]
\nn\\
&=
\mathbb{E}\left[
\exp \left(
-\frac{\la k, \,\cdot\,\ra}{\ve} - \frac{\|k\|_{\cH^H}^2}{2\ve^2}
\right)
\prod_{i}
\chi_r (\|(\ve W)^i \|_{i \alpha, 12m/i\textrm{-Bes}}^{12m/i} )
\, \frac{1}{\ve^e}\delta_0 \left(\frac{\tilde{y}^{\ve}_T-b}{\ve} \right)
\right].
\end{align}
Here, we set $\tilde{y}^{\ve}_T := y^{\ve}_T (\,\cdot\, +k/\ve)$.
Note that if 
$\|\ve W^1 \|_{\alpha, 12m\textrm{-Bes}}^{12m} \le r$, 
then
\[
|\la k, w\ra/\ve|\le  \|\la k, \,\cdot\,\ra\|_{\cW^*} \|w\|_{\cW}/\ve
\le  c\|\la k, \,\cdot\,\ra\|_{\cW^*} \|W^1 \|_{\alpha, 12m\textrm{-Bes}} /\ve
\le cr^{1/(12m)} /\ve^2
\]
for some constant $c>0$ which depends only on $(\alpha, 12m)$ and $k$.
Applying Sugita's theorem to the right hand side of \eqref{est.0809-1},
we can easily see that 
\begin{align}  
\lefteqn{
\tilde\nu_{a,b}^{H, \ve} (\mathcal{T}_k B_r (\mathbf{0}))
}
\nn\\
&\ge 
\exp \left(
-\frac{cr^{1/(12m)}}{\ve^2} - \frac{\|k\|_{\cH^H}^2}{2\ve^2}
\right)
\frac{1}{\ve^e}
\mathbb{E}\left[
\prod_{i}
\chi_r (\|(\ve W)^i \|_{i \alpha, 12m/i\textrm{-Bes}}^{12m/i} )
\, \delta_0 \left(\frac{\tilde{y}^{\ve}_T-b}{\ve} \right)
\right].
\nn
\end{align}
Therefore, in order to obtain \eqref{low_0808-2}, it suffices to show that
\begin{equation}\label{est.0809-3}
\lim_{\ve\searrow 0}
\mathbb{E}\left[
\prod_{i}
\chi_r (\|(\ve W)^i \|_{i \alpha, 12m/i\textrm{-Bes}}^{12m/i} )
\, \delta_0 \left(\frac{\tilde{y}^{\ve}_T-b}{\ve} \right)
\right]
\,\, \in (0,\infty)
\end{equation}
for every fixed $r>0$.

First, we can easily see that as $\ve\searrow 0$
\begin{equation}\label{est.0809-4}
\prod_{i}
\chi_r (\|(\ve W)^i \|_{i \alpha, 12m/i\textrm{-Bes}}^{12m/i} )
=\prod_{i}
\chi_r (\ve^{12m}\|W^i \|_{i \alpha, 12m/i\textrm{-Bes}}^{12m/i} )
=
1 + O(\ve)
\quad \mbox{in ${\bf D}_{\infty}$}.
\end{equation}
Next, it is known that 
$(\tilde{y}^{\ve}_T-b)/\ve =\{y^{\ve}_T (\,\cdot\, +k/\ve) - \hat\Psi_a (k)_T\}/\ve$ is uniformly non-degenerate 
in the sense of Malliavin.  
(See \cite[Proposition 5.4]{InaNag}. 
Since we work in the elliptic setting, a direct proof is 
not difficult in this case, however.)
It is also known that there exists 
$\zeta_T (\,\cdot\, ; k)\in \mathscr{C}^\prime_1 (\R^e)$
such that
\begin{equation}\label{est.0815-1}
y^{\ve}_T (\,\cdot\, +k/\ve) - \hat\Psi_a (k)_T
= \ve \zeta_T (\,\cdot\, ; k) + O (\ve^2) 
\quad \mbox{in ${\bf D}_{\infty} (\R^e)$}
\end{equation}
as $\ve\searrow 0$.
(This is a consequence of the Taylor expansion of Lyons-It\^o map. 
$\zeta_T (\,\cdot\, ; k)$ is the term of order 1 in the expansion.
This kind of expansion was already used in \cite{InaNag}, too.)
Moreover, $w \mapsto \zeta_T (w; k)$ is a Gaussian random variable whose covariance matrix
coincides with the deterministic Malliavin covariance matrix of 
$h \mapsto \hat\Psi_a (h)_T$ at $k$, which is non-degenerate due to the 
ellipticity assumption. (The mean may not vanish.)
Therefore the density function of 
the law of $\zeta_T (\,\cdot\, ; k)$ is strictly positive 
everywhere, in particular, $\mathbb{E} [\delta_0 (\zeta_T (\,\cdot\, ; k)) ] >0$.

Finally, we apply (a very special case of) the asymptotic theory for Watanabe distributions (see \cite[Section V-9]{iwbk}).
By \eqref{est.0815-1} and the uniform non-degeneracy, we have
\begin{equation}\label{est.0926-1}
\delta_0 \left(\frac{\tilde{y}^{\ve}_T-b}{\ve} \right) 
= 
\delta_0 (\zeta_T (\,\cdot\, ; k)) +O(\ve)
\quad \mbox{in ${\bf D}_{-\infty}$}
\end{equation}
as $\ve\searrow 0$. 
Combining this with \eqref{est.0809-4},  we see that the left hand side of 
\eqref{est.0809-3} equals $\mathbb{E} [\delta_0 (\zeta_T (\,\cdot\, ; k)) ] 
\in (0,\infty)$.
Thus, we have verified \eqref{est.0809-3} and therefore
finished the proof of \eqref{low_0808-1}.

\subsection{Exponential tightness}

In this subsection we show that, under the condition of 
Theorem \ref{thm.LDP8/6} (3), $\{\tilde\nu_{a,b}^{H, \ve}\}_{\ve \in (0,1]}$
is exponentially tight on 
$G\Omega_{\alpha, 12m\textrm{-}{\rm Bes}} (\R^d)$, that is, 
for every $R \in [1, \infty)$ there exists a precompact subset 
$K_R$ of $G\Omega_{\alpha, 12m\textrm{-}{\rm Bes}} (\R^d)$ 
satisfying that 
\begin{equation} \label{low_0819-1}
\varlimsup_{\ve\searrow 0}  \ve^2 \log \tilde\nu_{a,b}^{H, \ve} (K_R^c)\le -R.
\end{equation}
The key to prove this fact is that 
 the capacities of the complement of large 
 ``homogeneous balls" admit Gaussian decay.

\begin{lemma} \label{lem.0819-1}
For every $q \in (1, \infty)$ and $k \in \N$, 
there exist positive constants $c$ and $c^\prime$ satisfying that
\[
{\rm Cap}_{q,k} \left(\left\{ w\in \cW
\,\middle| \,
\|W^i \|_{i \alpha, 12m/i\textrm{-}{\rm Bes}}^{1/i} \ge  R
\right\} \right)
\le c^\prime e^{- cR^2}, \qquad  1\le i \le \lfloor 1/H\rfloor, 1\le R <\infty.
\]
Here, $c$ and $c^\prime$ may depend on $q,k, \alpha, m$,
but not on $R$.
\end{lemma}

\begin{proof} 
From a Fernique type theorem for fractional Brownian rough path
with respect to 
 the H\"older rough path topology (see \cite[Theorem 15.33]{fvbook}),  
we can easily see that 
there exist constants $c_1, c_1^\prime >0$ independent of $R$ such that
\begin{equation} 
\mu^H \left( \left\{ w\in \cW
\,\middle| \,
\|W^i \|_{i \alpha, 12m/i\textrm{-}{\rm Bes}}^{1/i} \ge R
\right\}
\right)
\le 
c_1^\prime e^{- c_1 R^2}, \quad  1\le i \le \lfloor 1/H\rfloor, 1\le R <\infty.
\nn
\end{equation}

Take a non-negative,  non-decreasing, smooth function
$\chi$ such that 
$\chi \equiv 0$ on $(-\infty, 0]$ and $\chi \equiv 1$ on $[1,\infty)$.
All derivatives of $\chi$ are bounded.
Set $G(w) = \|W^i \|_{i \alpha, 12m/i\textrm{-}{\rm Bes}}^{12m/i}$.
Then, $G \in {\bf D}_{\infty}$ and $\chi (G -R^{12m} +1)$
is $\infty$-quasi-continuous. 
Moreover, $\chi (G -R^{12m} +1) \ge 1$ holds 
on 
the subset $\{\|W^i \|_{i \alpha, 12m/i\textrm{-}{\rm Bes}}^{1/i} \ge R\}$. 
By \cite[P. 106, Theorem 4.4]{ma}, we have
\[
{\rm Cap}_{q,k} 
\bigl( \{\|W^i \|_{i \alpha, 12m/i\textrm{-}{\rm Bes}}^{1/i} \ge R  \} \bigr) \le \| \chi (G -R^{12m} +1)  \|_{\mathbf{D}_{q,k}}.
\]
Therefore, it suffices to estimate the Sobolev norm on the right hand side.

First we calculate the $L^q$-norm.
Since there exists a constant $r_0 \ge 2$ such that 
$R \ge r_0$ implies $(R^{12m} -1)^{1/12m} \ge R/2$, we have 
\begin{align*}  
\| \chi (G -R^{12m} +1)  \|_{L^q} 
&\le 
\mu^H  \bigl(  \{ 
\|W^i \|_{i \alpha, 12m/i\textrm{-}{\rm Bes}}^{1/i} 
\ge (R^{12m} -1)^{1/12m} \} \bigr)^{1/q}
\nn\\
&\le 
\mu^H \bigl(  \{ 
\|W^i \|_{i \alpha, 12m/i\textrm{-}{\rm Bes}}^{1/i} 
\ge R/2 \} \bigr)^{1/q}
\nn\\
&\le 
 (c_1^\prime)^{1/q} e^{ - (c_1 /4q)R^2}
\end{align*}
for all $R\ge r_0$.
By adjusting  positive constants, we can easily show
that, for some constants $c_2, c_2^\prime >0$, 
$\| \chi (G -R^{12m} +1)  \|_{L^q} 
\le c_2^\prime e^{-c_2 R^2}$ holds for all $R \ge 1$.

Since
$D [\chi (G -R^{12m} +1) ]= \chi^{\prime} (G -R^{12m} +1) DG$, we have
\begin{align}
\lefteqn{
\| D [\chi (G -R^{12m} +1) ]\|_{L^q}  
}
\nn\\
&\le 
\| \chi^{\prime} \|_{\infty}    \bigl\| I_{  \{ G -R^{12m} +1 \ge 0\} }  \|DG\|_{{\cal H}^H}  \bigr\|_{L^q} 
\nn\\
&\le
\| \chi^{\prime} \|_{\infty}  \|DG\|_{L^{2q}}
 \mu^H \bigl(  \{ 
 \|W^i \|_{i \alpha, 12m/i\textrm{-}{\rm Bes}}^{1/i} 
   \ge (R^{12m} -1)^{1/12m} \} \bigr)^{1/2q}
\nn\\
&\le
\| \chi^{\prime} \|_{\infty}  \|DG  \|_{L^{2q}}   \cdot
 (c_1^\prime)^{1/2q} e^{ - (c_1 /8q)R^2}
\nn
\end{align}
for all $R\ge r_0$.
Hence, there are constants $c_3, c_3^\prime >0$ such that
$\| D[ \chi (G -R^{12m} +1) ] \|_{L^q} 
\le c_3^\prime e^{-c_3 R^2}$ holds for all $R \ge 1$.
From this and Meyer's equivalence of the Sobolev norms,  
there exists $c_4, c_4^\prime >0$ such that 
\[
\| \chi (G -R^{4m} +1)  \|_{\mathbf{D}_{q,1}}  \le 
c_4^\prime e^{ - c_4R^2}
\]
for all $R\ge 1$.

Repeating essentially the same argument, 
we can estimate $D^l [\chi (G-R^{4m} +1)]$ for any 
$l~(1 \le l \le k)$ to obtain the Gaussian decay of  
$\| \chi (G -R^{4m} +1)  \|_{\mathbf{D}_{q,k}}$.
This completes the proof of the lemma.
\end{proof}

\begin{lemma} \label{lem.0820-1}
There exists a constant $\ell \in\N$ independent of $q$ and $\ve$
such that
\[
\|\delta_b (y^\ve (T,a)) \|_{\mathbf{D}_{q, -2e}}  
=
O (\ve^{-\ell}) \qquad \mbox{as $\ve\searrow 0$.} 
\]
for every $q \in (1,\infty)$.
\end{lemma}

\begin{proof} 
In this proof, $\ell_j\in \N$ are certain constants 
independent of $(q, \ve)$
and $c_j >0$ are certain constants independent of $\ve$ ($j=1,2,\ldots$).
For simplicity, we denote by $\sigma_\ve$ 
 the Malliavin covariance matrix of $y^\ve (T,a):=y^\ve_T$.
We also write $\gamma_\ve := (\sigma_\ve)^{-1}$.

It is well-known that every $\|y^\ve_T\|_{\mathbf{D}_{q, k}}$ 
is bounded in $\ve$ for every $(q, k) \in  (1,\infty)\times\N$.
As was explained in Subsection \ref{subsec.low},
$\{y^\ve_T/\ve\}_{\ve \in (0,1]}$ is uniformly non-degenerate 
in the sense of Malliavin calculus, which implies that
$
\| (\det \sigma_\ve )^{-1}\|_{L^q}= O (\ve^{-2e})
$
 as $\ve\searrow 0$ for every $q \in (1,\infty)$.
Hence, for some $\ell_1$, we have 
$
\| \gamma_\ve \|_{L^q}= O (\ve^{-\ell_1})
$
 as $\ve\searrow 0$ for every $q \in (1,\infty)$.

Set $A (z) := \prod_{j=1}^e \{(z^j -b^j) \vee 0\}$
for $z= (z^1, \ldots, z^e)$.
Then, we have 
$\partial_1^2 \cdots \partial_e^2 A =\delta_b$
in the distributional sense.
Now we  apply the multiple version of the
 by parts formula \eqref{ipb8.eq} in Remark \ref{notation_ibp} 
 with $F=y^\ve_T$, $T=A$ and $\mathbf{i} =(1,1,2,2,\ldots, e,e)$.
(see Remark \ref{notation_ibp} for the notation) to obtain
\begin{align}  
{\mathbb E} \bigl[
\delta_b ( y^\ve_T)  \, G 
\bigr]
=
{\mathbb E} \bigl[
A(  y^\ve_T) \, \Phi_{(1,1,\ldots, e,e)} (\, \cdot\, ;G)
\bigr],
\qquad
G \in {\mathbf D}_{\infty}.
\nn
\end{align}
It is easy to see that $\|A(  y^\ve_T)\|_{L^r}$ is bounded in $\ve$ for every $r\in (1,\infty)$.
For a given $q$, denote its conjugate exponent by $p$, i.e. $1/q +1/p =1$.
Then, there exists $\ell_2$ and $c_1$ such that
\[
\| \Phi_{(1,1,\ldots, e,e)} (\, \cdot\, ;G)\|_{L^{(p+1)/2}}
\le 
c_1 \|G\|_{\mathbf{D}_{p, 2e}} \ve^{-\ell_2},
\qquad
G \in {\mathbf D}_{\infty}, \ve \in (0,1].
\]
Combining these all, we have 
 \[
\left| {\mathbb E} \bigl[
\delta_b ( y^\ve_T)  \, G 
\bigr]
 \right|
\le 
c_2 \|G\|_{\mathbf{D}_{p, 2e}} \ve^{-\ell_2},
\qquad
G \in {\mathbf D}_{\infty}, \ve \in (0,1].
\]
Since ${\mathbf D}_{\infty}$ is dense in $\mathbf{D}_{p, 2e}$,
this estimate proves the lemma.
\end{proof}

Now we show the exponential tightness \eqref{low_0819-1}.
For $R \ge 1$, set 
\[
K_R := 
\left\{ X \in G\Omega_{\alpha, 12m\textrm{-}{\rm Bes}} (\R^d)
\middle|
\|X^i \|_{i \alpha, 12m/i\textrm{-}{\rm Bes}}^{1/i} <  R
\,\, \mbox{for all $i=1, \ldots,  \lfloor 1/H\rfloor$}
\right\}.
\]
Clearly, this is a bounded subset. 
By \eqref{sugisugi}, Lemmas \ref{lem.0819-1} and \ref{lem.0820-1}, we have
\begin{align*}  
\tilde\nu_{a,b}^{H, \ve} (K_R^c)
&=
\tilde\mu_{a,b}^{H, \ve} (\{w\in\cW \mid 
\ve W \in K_R^c\})
\nn\\
&\le 
\|\delta_b (y^\ve (T,a)) \|_{\mathbf{D}_{2, -2e}} 
{\rm Cap}_{2, 2e} (\{w\in\cW \mid 
\ve W \in K_R^c\})
\nn\\
&\le 
c^\prime \ve^{-\ell}  e^{-cR^2/\ve^2}, \qquad\qquad \ve\in (0,1], R\ge 1.
\end{align*}
Here, $c, c^\prime, \ell$ are positive constants 
independent of $R$ and $\ve$.
Thus, we have
\begin{align}  \label{est.0820-2}
\varlimsup_{\ve\searrow 0}  \ve^2 \log \tilde\nu_{a,b}^\ve (K_R^c)\le -cR^2.
\end{align}
Since $cR^2$ can be arbitrarily large, we have shown 
\eqref{low_0819-1} except the relative compactness of $K_R$.

For the relative compactness, we take $\alpha^\prime \in (\alpha, H)$.
When $\alpha^\prime$ is sufficiently close to $\alpha$,
$(\alpha^\prime, 12m)$ still satisfies the assumption of 
Theorem \ref{thm.LDP8/6} (3) and the embedding 
$G\Omega_{\alpha^\prime, 12m\textrm{-}{\rm Bes}} (\R^d)\hookrightarrow
G\Omega_{\alpha, 12m\textrm{-}{\rm Bes}} (\R^d)$ is compact (i.e. 
every bounded subset of the left set is a relatively compact 
 subset of the right set).
Therefore, using \eqref{est.0820-2} for the 
$(\alpha^\prime, 12m)$-Besov topology finishes the proof of 
the exponential tightness \eqref{low_0819-1}.

\subsection{Large deviation upper bound}

In this subsection we will show that 
\begin{equation} \label{up_0826-1}
\varlimsup_{\ve\searrow 0}  \ve^2 \log \tilde\nu_{a,b}^{H,\ve} (F)
 \le -\inf_{X \in F} \tilde{I} (X)
\end{equation}
holds for every closed subset 
$F$ of $G\Omega_{\alpha, 12m\textrm{-}{\rm Bes}} (\R^d)$.

In view of the exponential tightness \eqref{low_0819-1}, 
the upper estimate 
\eqref{up_0826-1} immediately follows from the following inequality:
\begin{equation} \label{up_0826-2}
\lim_{r \searrow 0}
\varlimsup_{\ve\searrow 0}  \ve^2 \log \tilde\nu_{a,b}^{H,\ve} ( B_r (X))
 \le - \tilde{I}(X)
 \end{equation}
for every $X\in G\Omega_{\alpha, 12m\textrm{-}{\rm Bes}} (\R^d)$.
We will show \eqref{up_0826-2} in Lemmas \ref{lem.0826-1}
and \ref{lem.0826-2} below.

Now let us prove \eqref{up_0826-2}.

\begin{lemma} \label{lem.0826-1}
If $\hat\Phi_a (\cP (X, \beta_0 \lambda))_T \neq b$,
then \eqref{up_0826-2} holds.
\end{lemma}

\begin{proof} 
First, note that
since $w \mapsto \hat\Phi_a (\cP (\ve W, \beta_\ve \lambda))_T$ is 
$\infty$-quasi continuous,
$\tilde\mu_{a,b}^{H,\ve}$ is concentrated on the subset
$\{w\in\cW \mid \hat\Phi_a (\cP (\ve W, \beta_\ve \lambda))_T =b\}$
(see Lemma \ref{lem.0510_1} and its proof).
By the continuity, there exists $\delta, r>0$ and $\ve_0 \in (0,1)$ such that
if $\hat{X} \in B_r (X)$ and $\ve \in (0,\ve_0)$, then
$|\hat\Phi_a (\cP (\hat{X}, \beta_\ve \lambda))_T -b|\ge \delta$.
Therefore,  we have
\[
\tilde\nu_{a,b}^{H,\ve} ( B_r (X))
 =  \tilde\mu_{a,b}^{H, \ve} \left( \ve W \in B_r (X)\right)
=0
\]
for sufficiently small $\ve >0$.
This implies  \eqref{up_0826-2} since $\tilde{I} (X) =\infty$.
\end{proof}

Before proving the other case $\hat\Phi_a (\cP (X, \beta_0 \lambda)) =b$,
let us recall a Schilder type large deviation principle
for fractional Brownian rough path with respect to the
$\alpha'$-H\"older topology for every $\alpha' <H$. 
(See \cite[Theorem 15.55]{fvbook}, in which 
the expression of the rate function seems to contain a minor typo.)
By the embedding in \eqref{ineq.0311-4},
the same large deviation principle still holds on $G\Omega_{\alpha, 12m\textrm{-}{\rm Bes}} (\R^d)$.
Due to the general fact in the large deviation theory
(see \cite[Theorem 4.1.18]{dzbook} for instance), we have 
\begin{equation} \label{eq.0826-3}
\lim_{r \searrow 0}
\varlimsup_{\ve\searrow 0}  \ve^2 \log 
\mu^H  (\{\ve W \in  B_r (X)\})
 \le -I_{{\rm Sch}} (X)
\end{equation}
for every $X\in G\Omega_{\alpha, 12m\textrm{-}{\rm Bes}} (\R^d)$,
where we set
\[
I_{{\rm Sch}} (X) := \begin{cases}
\tfrac12 \|h\|^2_{\cH^H} 
& (\mbox{if $X= S_{\lfloor 1/H\rfloor} (h)$ for some 
$h\in \cH^H$}),\\
\infty & (\mbox{otherwise}).
\end{cases}
\]

\begin{lemma} \label{lem.0826-2}
It holds that 
\[
\lim_{r \searrow 0}
\varlimsup_{\ve\searrow 0}  \ve^2 \log \tilde\nu_{a,b}^{H, \ve} ( B_r (X))
 \le -I_{{\rm Sch} }(X)
\]
for every $X\in G\Omega_{\alpha, 12m\textrm{-}{\rm Bes}} (\R^d)$.
\end{lemma}

\begin{proof} 
Take a non-negative,  smooth, even function
$\varphi$ such that 
$\varphi \equiv 1$ on $[0, 1]$ and $\varphi \equiv 0$ on $[2,\infty)$.
Clearly, $\varphi$ is compactly supported.
If we set 
\[
F^{\ve}_r (w) = \prod_{1\le i\le \lfloor 1/H\rfloor}
\varphi \left(\|(\ve W)^i-X^i\|_{i \alpha, 12m/i\textrm{-Bes}}^{12m/i} /r\right),
\]
then $F^{\ve}_r \in {\bf D}_{\infty}$, $F^{\ve}_r \ge 0$ and 
$F^{\ve}_r\equiv 1$ on $B_r(X)$.
Moreover, $F^{\ve}_r$ is $\infty$-quasi-continuous in $w$.
For every $k\in\N$ and $1 <p<\tilde{p}<\infty$, 
there exist 
positive constants $c_1, c_2$ 
which may depend on $k, p, \tilde{p}, r$ but not on $\ve$ such that
\begin{align*}  
\|F^{\ve}_r\|_{\mathbf{D}_{p, k}} 
&\le
c_1 \sum_{j=0}^{k}\|D^j F^{\ve}_r\|_{L^p}
\le
c_2 \mu^H \left( \{ \ve W\in B_{2r} (X) \} \right)^{1/ \tilde{p}}.
\end{align*}
Using this, Sugita's theorem and Lemma \ref{lem.0820-1}, we have
\begin{align}  
 \tilde\nu_{a,b}^{H, \ve} ( B_r (X))
&=
 \tilde\mu_{a,b}^{H,\ve} \left( \{\ve W \in B_r (X)\}\right)
 \nn\\
 &\le
  \int_{\cW}  F^{\ve}_r (w) \, \tilde\mu_{a,b}^{H, \ve} (dw)
   \nn\\
 &= \mathbb{E} [\delta_b ( y^\ve_T)  \, F^{\ve}_r ]
  \nn\\
 &\le
 \|\delta_b (y^\ve (T,a)) \|_{\mathbf{D}_{q, -2e}}   
 \|F^{\ve}_r\|_{\mathbf{D}_{p, 2e}}
   \nn\\
 &\le
  c_3 \ve^{-\ell}  \mu^H \left( \{\ve W\in B_{2r} (X)  \}\right)^{1/ \tilde{p}},
  \nn
\end{align}
where $p^{-1} + q^{-1}=1$ and $\ell, c_3$ are certain positive constants
independent of $\ve$.
This inequality and \eqref{eq.0826-3} immediately imply that 
\[
\lim_{r \searrow 0}
\varlimsup_{\ve\searrow 0}  \ve^2 \log  \tilde\nu_{a,b}^{H, \ve} ( B_r (X))
 \le -\frac{1}{\tilde{p}}I_{{\rm Sch}} (X). 
\]
Setting $\tilde{p} =2p -1$ and then letting $p\searrow 1$,
we obtain the desired estimate.
\end{proof}

\section{The Young case: $H>1/2$}\label{sec.Young}

In this section we assume $H \in (1/2, 1)$. 
Denote by $\mu^H$ the law of
$d$-dimensional FBM with Hurst parameter $H$,
which is a Gaussian measure on
 $\mathcal{W}=\tilde\cC^{\alpha\textrm{-Hld}}_0 (\R^d)$
  with $\alpha  \in (1/2, H)$.
The coordinate process $w=(w_t)_{t\in [0,T]}$ on $\cW$ 
is a canonical realization of FBM under $\mu^H$.
The Cameron-Martin space of $\mu^H$ is denoted by $\cH^H$.
As is well-known, $(\cW, \cH^H, \mu^H)$ is an abstract Wiener space.
Throughout this section $V_j~(j\in \llbracket 0, d\rrbracket)$ 
are assumed to be $C_{{\rm b}}^\infty$-vector fields on $\R^e$.

We consider the Young differential equation (YDE)
driven by FBM $w=(w_t)_{t\in [0,T]}$:
\begin{equation}\label{rde8}
dy_t = \sum_{j=1}^d V_j (y_t) dw^{j}_t  + V_0 (y_t) dt,
\qquad 
y_0 =a\in \R^e.
\end{equation}
Under our condition on the coefficients, 
this YDE is deterministic and has a unique solution for every fixed $w$.
We write $y = \Phi_a (w)$ and view $\Phi$ as a map from 
$\mathcal{W}$ to $\tilde\cC^{\alpha\textrm{-Hld}}_a (\R^e)$.
It is known that $\Phi$ is Fr\'echet-$C^\infty$ (see \cite{ina_young} for example).
We denote by $\mathbb{Q}_{a}^H$ the law of this process $y$, that is,
$\mathbb{Q}_{a}^H := (\Phi_a)_* \mu^H$.
Since $\Phi_a$ takes values in the ``little $\alpha$-H\"odler space" 
$\tilde{\cC}^\alpha_a (\R^e)$,
this Borel probability measure sits on $\tilde{\cC}^{\alpha\textrm{-Hld}}_a (\R^e)$.

Under the  $C_{{\rm b}}^\infty$-condition on the coefficients, 
$y_t \in {\bf D}_{\infty} (\R^e)$ for every $t \in [0,T]$.
(See \cite{ina_young}. See also \cite{bh} for the driftless case. 
These proofs are different.)
Suppose that $y_T =y (T,a)$ is non-degenerate 
in the sense of Malliavin and  
that $p(T, a, b):= {\mathbb E} [\delta_b (y (T,a)) ]> 0$.
In this case, a unique 
Borel probability measure on $\cW$ corresponding to 
$p(T, a, b)^{-1}\delta_b (y (t,a))$ is denoted by $\mu^H_{a,b}$.
We then set  
$\mathbb{Q}^H_{a,b} :=(\Phi_a)_* \circ \mu^H_{a,b}$.
(When $p(T, a, b)=0$, we set $\mathbb{Q}^H_{a,b} :=\delta_{\ell(a,b)}$,
where $\ell(a,b)_t := a + (b-a) (t/T)$. )

As in the rough path case (i.e. $1/4 <H\le 1/2$), $\mathbb{Q}^H_{a,b}$ can be 
regarded as (the law of) a bridge process.
\begin{theorem} \label{thm.0921}
Let $1/2 <\alpha <H <1$.
Assume that 
$y_T =y (T,a)$ is non-degenerate in the sense of Malliavin.
Then, the family 
\[
\{ \mathbb{Q}^H_{a,b} (A) \mid A\in 
\mathcal{B} (\tilde{\cC}^{\alpha\textrm{-}{\rm Hld}}_{a} (\R^e)),\,  b\in \R^e\}
\]
is (a version of) the regular conditional probability 
of $\mathbb{Q}^H_{a}$ on 
$\tilde{\cC}^{\alpha\textrm{-}{\rm Hld}}_{a} (\R^e)$
given the evaluation map at the time $T$.
\end{theorem}

\begin{proof} 
The proof of this theorem is much easier than that of Theorem \ref{thm.0510_2} since the rough path lift map is not involved in the Young case.
So, we omit it. 
\end{proof}


%
\begin{remark}\label{rem.positive2}
Assume that $y_T =y (T,a)$ is non-degenerate in the sense of Malliavin.
As in the case $H \in (1/4,1/2]$,  
$p(T, a, b)>0$ if and only if there exists $h\in \cH^H$
satisfying the following two conditions:
\begin{enumerate} 
\item[(1)]~$\Phi_a (h)_T =b$
\item[(2)]~$D\Phi_a (h)_T$ is a surjective linear map from 
$\cH^H$ to $\R^e$.
Here, $D\Phi_a (h)_T$ stands for the Fr\'echet derivative of 
$k \in \cH^H \mapsto D\Phi_a (k)_T\in \R^e$ at $h\in \cH^H$.
\end{enumerate}
This equivalence can be checked as follows.
The map $w \mapsto \Phi_a (w)$ is known to be 
Fr\'echet-$C^3$ whose derivatives of order $k~(0 \le k\le 3)$
are of at most polynomial growth, which
implies that this map is twice $\cK$-regularly differentiable in 
the sense of Aida-Kusuoka-Stroock \cite{aks}.
Then, the above equivalence is a special case of \cite[Theorem 2.8]{aks}.
\end{remark}

In the following remark, we recall some results on the
 non-degeneracy of the solution $(y_t)_{t\in [0,T]}$ of RDE \eqref{rde1}.  
In almost all problems concerning stochastic equations 
driven by FBM, 
the case $H >1/2$ is much easier than the case $H\le  1/2$.
On the Malliavin non-degeneracy, however,
the case $H >1/2$ might be less known than the case $1/4 <H\le  1/2$.

\begin{remark}\label{rem.nondeg2}
Let $\{V_j\}_{0\le j\le d}$ be the coefficient vector fields of YDE \eqref{rde8}.
Here, they are viewed as  first-order differential operators on $\R^e$.
We provide two known examples of sufficient condition
for  $y_t =y (t,a)$ to be non-degenerate in the sense of Malliavin 
for all $t \in (0, T]$.
\begin{enumerate} 
\item
$\{V_j\}_{0\le j\le d}$ satisfies the ellipticity condition at 
the starting point $a\in \R^e$ (see \cite{ina_young}).
\item
$V_0 \equiv 0$ and $\{V_j\}_{1\le j\le d}$ satisfies H\"ormander's
bracket-generating condition at the starting point $a\in \R^e$  (see \cite{bh}).
\end{enumerate}
As far as the author knows, there is no published paper which proves 
the Malliavin non-degeneracy 
under the same bracket-generating condition on $\{V_j\}_{0\le j\le d}$
as in Remark \ref{rem.nondeg}.
However, this issue seems to have been solved by 
a recent unpublished work \cite{paul}.
\end{remark}

Now we turn to the large deviation principle.
For a small parameter $\ve\in (0,1]$, 
we consider the following scaled YDE:
\begin{equation}\label{rde9}
dy^{\ve}_t = \ve \sum_{j=1}^d V_j (y^{\ve}_t) dw^{j}_t  + 
\beta_\ve V_0 (y^{\ve}_t) dt,
\qquad 
y^{\ve}_0 =a\in \R^e.
\end{equation}
We sometime write $y^{\ve}_t =y^{\ve} (t, a)$.
Here, $\beta \colon [0,1] \to \R$ is a ${\rm Lip}^\gamma$-function for some $\gamma >1$.
(See Remark \ref{rem.relax}.)
Typical examples of $\beta$ are
$\beta_\ve =1$ and $\beta_\ve =\ve^{1/H}$.


We denote by 
$\hat\Phi_a \colon \tilde\cC^{\alpha\textrm{-Hld}}_0 (\R^{d+1})
\to \tilde\cC^{\alpha\textrm{-Hld}}_a (\R^e)$ 
the Young-It\^o map (i.e. the solution map of YDE)
associated with the coefficients $\{V_j\}_{ 0\le j \le d}$,
where $1/2 <\alpha <H$.
Then, $y^\ve = \hat\Phi_a ((\ve w, \beta_\ve \lambda))$.
The associated skeleton map is defined by 
$ \hat\Psi_a (h) := \hat\Phi_a ((h, \beta_0 \lambda))$ for $h\in \cH^H$.

We assume that the coeffiecient vector fields
$V_j~(j\in \llbracket 0, d\rrbracket)$ are of $C_{{\rm b}}^\infty$
satisfying the following condition: 
There exists a connected open subset $U$ of $\R^e$ 
such that $a, b \in U$ and $\{V_j\}_{ 0\le j \le d}$
satisfies the ellipticity condition on $U$.
Then, as we have seen in Remark \ref{rem.positive2},
$\delta_b (y^\ve (t,a))$ is a positive Watanabe distribution
with $p^\ve (T, a, b) := {\mathbb E} [\delta_b (y^\ve (t,a)) ]>0$.
We denote by $\mu_{a,b}^{H, \ve}$ the
Borel probability measure on $\cW$ corresponding to 
$p^\ve (T, a, b)^{-1}\delta_b (y^\ve (t,a))$.

Now we provide our main large deviation results in the Young case.
The speed of these large deviations is $1/\ve^2$.
In this setting, the rate functions are defined by 
\[
\tilde{I} (w) 
= 
\begin{cases}
\tfrac12 \|w\|^2_{\cH^H} 
& (\mbox{if $w\in \cH^H$ with $\hat\Psi_a (w)_T =b$}),\\
\infty & (\mbox{otherwise}).
\end{cases}
\]
and 
\[
I (w)= \tilde{I}(w) -\min\{ \tfrac12 \|h\|^2_{\cH^H} \mid 
\mbox{$h\in \cH^H$ with $\hat\Psi_a (h)_T=b$} \}.
\]

\begin{theorem} \label{thm.LDP9/25}
Let $H \in (1/2, 1)$.
We assume that $C_{{\rm b}}^\infty$-vector fields
$V_j~(j\in \llbracket 0, d\rrbracket)$ on $\R^e$ satisfy the following condition: 
There exists a connected open subset $U$ of $\R^e$ 
such that $a, b \in U$ and $\{V_j\}_{ 0\le j \le d}$
satisfies the ellipticity condition on $U$.
Then, the following two statements hold:
\\
{\rm (1)}~The following Varadhan-type asymptotics holds:
\[
\lim_{\ve\searrow 0}\ve^2 \log p^\ve (T, a, b)=-\min\{ \tfrac12 \|h\|^2_{\cH^H} \mid 
\mbox{$h\in \cH^H$ with $\hat\Psi_a (h)_T =b$} \}.
\]
{\rm (2)}~For every $\alpha\in (1/2, H)$,
$\{\mu_{a,b}^{H, \ve}\}_{\ve \in (0,1]}$ satisfies a 
large deviation principle 
on $\tilde\cC^{\alpha\textrm{-}{\rm Hld}}_0 (\R^d)$ as $\ve \searrow 0$
with good rate function $I$. 
Moreover, $\{\mu_{a,b}^{H,\ve}\}_{\ve \in (0,1]}$  is exponentially tight 
as $\ve \searrow 0$.
\\
{\rm (3)}~For every $(\alpha, m)\in (1/2, H)\times \N$ with 
$\alpha -(2m)^{-1} > 1/2$, 
$\{\mu_{a,b}^{H,\ve}\}_{\ve \in (0,1]}$ satisfies a large deviation principle 
on $\cC^{\alpha, 2m\textrm{-}{\rm Bes}}_0 (\R^d)$
as $\ve \searrow 0$ with good rate function $I$. 
Moreover, $\{\mu_{a,b}^{H,\ve}\}_{\ve \in (0,1]}$  is exponentially tight 
as $\ve \searrow 0$.
\end{theorem}

\begin{proof} 
Due to the continuous embedding 
$\cC^{\alpha, 2m\textrm{-}{\rm Bes}}_0 (\R^d) 
\hookrightarrow \tilde\cC^{\gamma\textrm{-}{\rm Hld}}_0 (\R^d)$ with
$\gamma := \alpha - (2m)^{-1}$, Assertion (2) immediately follows from
Assertion (3).

The proofs of Assertion (1) and (3) are essentially the same as 
(actually much easier than) those of Theorem \ref{thm.LDP8/6} (1) and (3). So we omit details. (See also Remark \ref{rem.relax})
\end{proof}

We denote by $\mathbb{Q}_{a,b}^{H, \ve}$ the law of $y^{\ve}$ 
conditioned that $y^{\ve}_T =b$, that is, 
the law of $y^\ve = \hat\Phi_a ((\ve w, \beta_\ve \lambda))$
under $\mu_{a,b}^{H, \ve}$.
Then, we have 
$\mathbb{Q}_{a,b}^{H, \ve} 
= [\hat\Phi_a]_* 
(\mu_{a,b}^{H, \ve} \otimes \delta_{\beta_\ve \lambda})$.

\begin{corollary} \label{cor.ldp_fdb2}
Let $H \in (1/2, 1)$.
We assume that $C_{{\rm b}}^\infty$-vector fields
$V_j~(j\in \llbracket 0, d\rrbracket)$ on $\R^e$ satisfy the following condition: 
There exists a connected open subset $U$ of $\R^e$ 
such that $a, b \in U$ and $\{V_j\}_{ 0\le j \le d}$
satisfies the ellipticity condition on $U$.

Then, for every $\alpha\in (1/2, H)$, 
$\{\mathbb{Q}_{a,b}^{H, \ve}\}_{\ve \in (0,1]}$  
satisfies a large deviation principle 
on $\tilde{\cC}^{\alpha\textrm{-}{\rm Hld}}_{a, b} (\R^e)$ as $\ve \searrow 0$
with good rate function $J$, which is defined by
\begin{align*}
J (\xi) &:= \inf\{  \tfrac12 \|h\|^2_{\cH^H}  \mid 
\mbox{$h\in \cH^H$ with $\xi =\hat\Psi_a (h)$}\}, 
\\
&\qquad -\min\{ \tfrac12 \|h\|^2_{\cH^H} \mid 
\mbox{$h\in \cH^H$ with $\hat\Psi_a (h)_T =b$} \},
\qquad \xi \in \tilde{\cC}^{\alpha\textrm{-}{\rm Hld}}_{a, b} (\R^e).
 \end{align*}
As usual, $\inf \emptyset =\infty$ by convention.
\end{corollary}

\begin{proof} 
Since the Young-It\^o map $\hat\Phi_a$ is continuous, 
we can show this corollary from Theorem \ref{thm.LDP9/25} (2)
in the same way as in Corollary \ref{cor.ldp_fdb}.
So, we omit details.
\end{proof}

\noindent
{\bf Acknowledgement:}~
The author is supported by JSPS KAKENHI (Grant No. 25KF0152).


\bigskip
\begin{flushleft}
  \begin{tabular}{ll}
    Yuzuru \textsc{Inahama}
    \\
    Faculty of Mathematics,
    \\
    Kyushu University,
    \\
    744 Motooka, Nishi-ku, Fukuoka, 819-0395, JAPAN.
    \\
    Email: {\tt inahama@math.kyushu-u.ac.jp}
  \end{tabular}
\end{flushleft}


\begin{thebibliography}{00}
\bibitem{aida}
Aida, S.;  Vanishing of one-dimensional $L^2$-cohomologies of loop groups. 
J. Funct. Anal. 261 (2011), no. 8, 2164--2213.

\bibitem{aks}
Aida, S.; Kusuoka, S.; Stroock, D.; 
On the support of Wiener functionals. 
Asymptotic problems in probability theory: Wiener functionals and asymptotics (Sanda/Kyoto, 1990), 3--34, 
Pitman Res. Notes Math. Ser., 284, Longman Sci. Tech., Harlow, 1993. 



\bibitem{bai}
Bailleul, I.,
Large deviation principle for bridges of sub-Riemannian diffusion processes.
S\'eminaire de Probabilit\'es XLVIII, 189--198, 
Lecture Notes in Math., 2168, Springer, Cham, 2016. 


\bibitem{bh}
Baudoin, F.; Hairer, M.;
A version of H\"ormander's theorem for the fractional Brownian motion.  Probab. Theory Related Fields  139  (2007),  no. 3-4, 373--395.



\bibitem{bg15}
Boedihardjo, H.; Geng, X.;
The uniqueness of signature problem in the non-Markov setting.
Stochastic Process. Appl. 125 (2015), no. 12, 4674--4701.


\bibitem{bglq}
Boedihardjo, H.; Geng, X.; Liu, X.; Qian, Z.;
A quasi-sure non-degeneracy property for the Brownian rough path. 
Potential Anal. 51 (2019), no. 1, 1--21. 

\bibitem{bgq}
Boedihardjo, H.; Geng, X.; Qian, Z.;
Quasi-sure existence of Gaussian rough paths and large deviation principles for capacities. 
Osaka J. Math. 53 (2016), no. 4, 941--970. 

\bibitem{chlt}
Cass, T.; Hairer, M.; Litterer, C.; Tindel, S.;
Smoothness of the density for solutions to Gaussian rough
  differential equations.
Ann. Probab. 43 (2015), no. 1, 188--239.

\bibitem{dzbook}
Dembo, A., Zeitouni, O.; Large deviations techniques and applications.
Second edition.  
Springer-Verlag, New York, 1998. 



\bibitem{fv_embed}
Friz, P.; Victoir, N.; 
A variation embedding theorem and applications.
J. Funct. Anal. 239 (2006),  no. 2, 631--637.

\bibitem{fvbook}
Friz, P.; Victoir, N.;
Multidimensional stochastic processes as rough paths, 
Cambridge University Press, Cambridge, 2010.


\bibitem{fv_2Dvar}
Friz, P.; Victoir, N.; 
A note on higher dimensional $p$-variation.
Electron. J. Probab. 16 (2011), no. 68, 1880--1899.



\bibitem{got1}
Gess, B.; Ouyang, C.; Tindel, S.;
Density bounds for solutions to differential equations driven by Gaussian rough paths.
J. Theoret. Probab. 33 (2020), no. 2, 611--648.


\bibitem{hsu}
Hsu, P.;  
Brownian bridges on Riemannian manifolds. 
Probab. Theory Related Fields 84 (1990), no. 1, 103--118.



\bibitem{hw}
H\"ogele, M. A.; Wetzel, T.; 
Large deviations for light-tailed L\'evy bridges on short time scales.
Preprint. arXiv:2505.23972.

\bibitem{hu}
Hu, Y.;
{\it Analysis on Gaussian spaces.} World Scientific, 2017.

\bibitem{iwbk}
Ikeda, N., Watanabe, S.; {\it Stochastic differential equations and diffusion processes. }
Second edition. 
North-Holland Publishing Co., Amsterdam; Kodansha, Ltd., Tokyo, 1989.




\bibitem{in0}
Inahama, Y.;
Quasi-sure existence of Brownian rough paths and a construction of Brownian pants. Infin. Dimens. Anal. Quantum Probab. Relat. Top. 9 (2006), no. 4, 513--528.




\bibitem{ina_young}
Inahama, Y.;
Short time kernel asymptotics for Young SDE by means of Watanabe distribution theory.
J. Math. Soc. Japan 68 (2016), no. 2, 535--577.

\bibitem{in1}
Inahama, Y.;
Large deviation principle of Freidlin-Wentzell type for pinned diffusion processes.
Trans. Amer. Math. Soc. 367 (2015), 8107-8137.


\bibitem{in2}
Inahama, Y.; 
Malliavin differentiability of solutions of rough differential equations. 
J. Funct. Anal. 267 (2014), no. 5, 1566--1584. 

\bibitem{ina_FBM}
Inahama, Y.;
Short time kernel asymptotics for rough differential equation driven by fractional Brownian motion.
Electron. J. Probab. 21 (2016), Paper No. 34, 29 pp.


\bibitem{in3}
Inahama, Y.; 
Large deviations for rough path lifts of Watanabe's pullbacks of delta functions. 
Int. Math. Res. Not. IMRN 2016, no. 20, 6378--6414. 


\bibitem{in4}
Inahama, Y.; 
Large deviations for small noise hypoelliptic diffusion bridges on 
sub-Riemannian manifolds. 
Publ. Res. Inst. Math. Sci. 60 (2024), no. 1, 145--184.

\bibitem{in5}
Inahama, Y.; 
Support theorem for pinned diffusion processes.
Nagoya Math. J. 253 (2024), 241--264.


\bibitem{InaNag}
Inahama, Y.; Naganuma, N.;
Asymptotic expansion of the density for hypoelliptic rough differential equation.
Nagoya Math. J. 243 (2021), 11--41.

\bibitem{ip}
Inahama, Y.; Pei, B.;
Positivity of the Density for Rough Differential Equations. 
J. Theoret. Probab. 35 (2022), no. 3, 1863--1877.




\bibitem{lclbook}
Lyons, T.; Caruana, M.; L\'evy, T.;
{\it Differential equations driven by rough paths. }
  Lecture Notes in Math., 1908. Springer, Berlin, 2007.


\bibitem{ma}
Malliavin, P.;  {\it Stochastic analysis.}
Springer-Verlag, Berlin, 1997.


\bibitem{mt}
Matsumoto, H.; Taniguchi, S.;
{\it Stochastic analysis. 
It\^o and Malliavin calculus in tandem. }
Cambridge University Press, Cambridge, 2017.


\bibitem{nu}
 Nualart, D.; 
 {\it The Malliavin calculus and related topics. }
 Second edition. Springer-Verlag, Berlin, 2006. 


\bibitem{orv}
Ouyang, C.; Roberson-Vickery, W.;
Quasi-sure non-self-intersection for rough differential equations driven by fractional Brownian motion. 
Electron. Commun. Probab. 27 (2022), Paper No. 15, 12 pp.

\bibitem{paul}
Paulovics, P.:
H\"ormander's theorem for Gaussian rough differential equations.
Master thesis, Corvinus University of Budapest (2023). 
Unpublished.

\bibitem{sh}
Shigekawa, I.; {\it Stochastic analysis.}
Translations of Mathematical Monographs, 224. Iwanami Series in Modern Mathematics. 
American Mathematical Society, Providence, RI, 2004. 


%

\bibitem{tw}
Takanobu, S.; Watanabe, S; 
Asymptotic expansion formulas of the Schilder type for a class of conditional Wiener functional integrations. 
Asymptotic problems in probability theory: Wiener functionals and asymptotics (Sanda/Kyoto, 1990), 
194--241, Pitman Res. Notes Math. Ser., 284, Longman Sci. Tech., Harlow, 1993.


\end{thebibliography}
\end{document}